\documentclass{article}

\usepackage{bbm}
\usepackage[utf8]{inputenc}
\usepackage[english]{babel}
\usepackage{color}
\usepackage{amsmath,amssymb,amsthm,yhmath,amsfonts,mathrsfs,mathtools,dsfont, enumerate}
\usepackage{tikz}
\usetikzlibrary{automata,shapes,arrows,positioning}

\newcommand{\pr}{\mathbb{P}}
\newcommand{\E}{\mathbb{E}}

\renewcommand{\d}{\mathrm{d}}

\newcommand{\1}[1]{{\mathds{1}}_{\left\{#1\right\}}}

\renewcommand{\S}{\mathcal{S}}

\newtheorem{theorem}{Theorem}[section]
\newtheorem{definition}[theorem]{Definition}
\newtheorem{remark}[theorem]{Remark}
\newtheorem{lemma}[theorem]{Lemma}

\newtheorem{corollary}[theorem]{Corollary}
\newtheorem{conjecture}[theorem]{Conjecture}
\newtheorem{claim}[theorem]{Claim}
\newtheorem{proposition}[theorem]{Proposition}

\DeclareMathOperator{\cov}{cov}

\newcommand{\Var}{\mathrm{Var}}

\title{Concentration and mean field approximation results for Markov processes on large networks
	\footnote{ Partially supported by the ERC Synergy under Grant No. 810115 - DYNASNET, the grant NKFI-FK-142124 of NKFI
(National Research, Development and Innovation Office). Supported by the ÚNKP-23-3-II-BME-221 New National Excelence Program of the Ministry for Culture and Innovation from the source of the National Research Development and Innovation Fund.
\begin{center}
\includegraphics[width=0.1\textwidth]{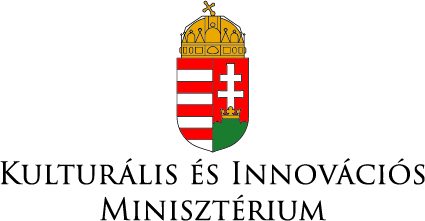}
\includegraphics[width=0.1\textwidth]{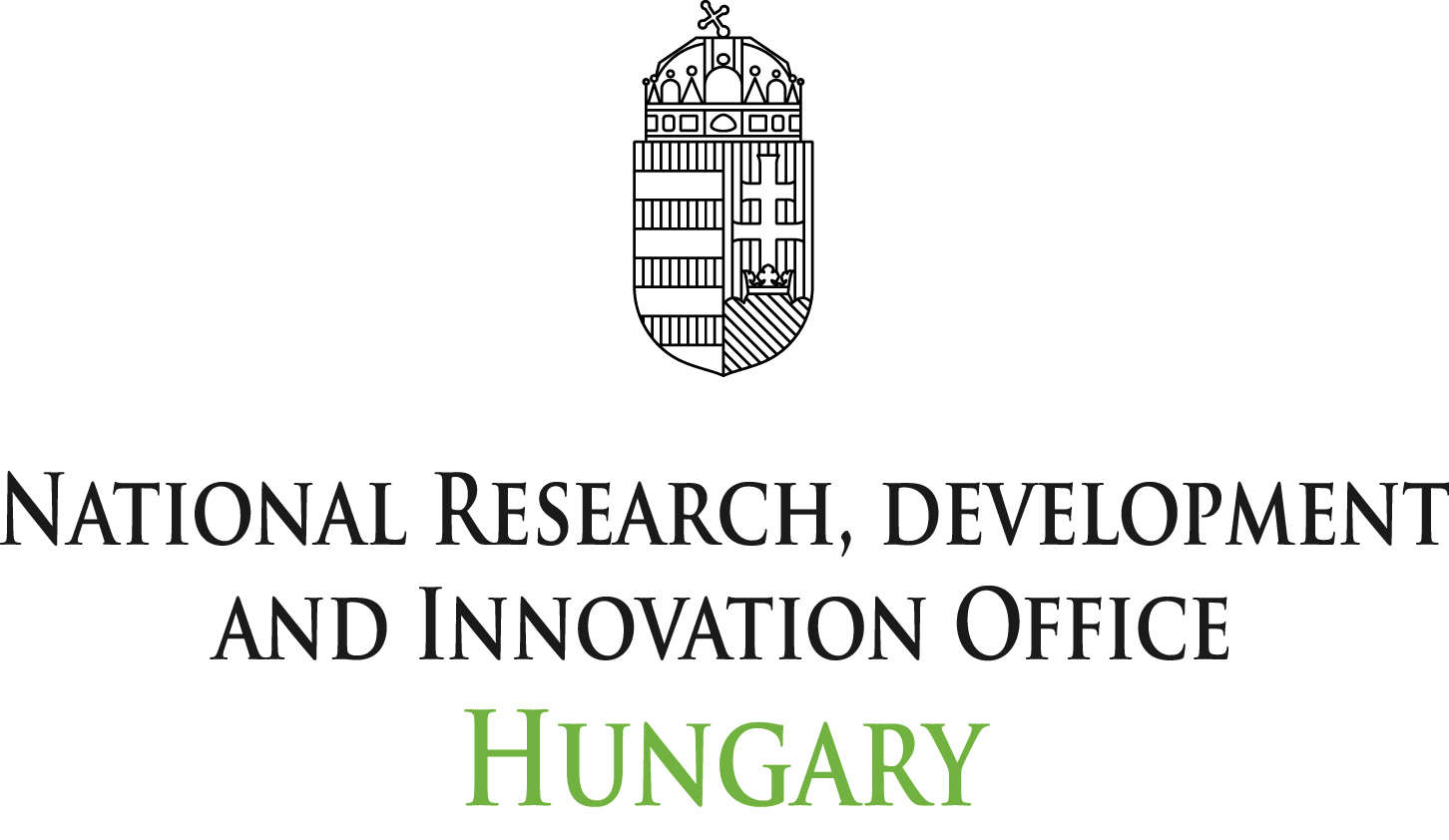}
\includegraphics[width=0.1\textwidth]{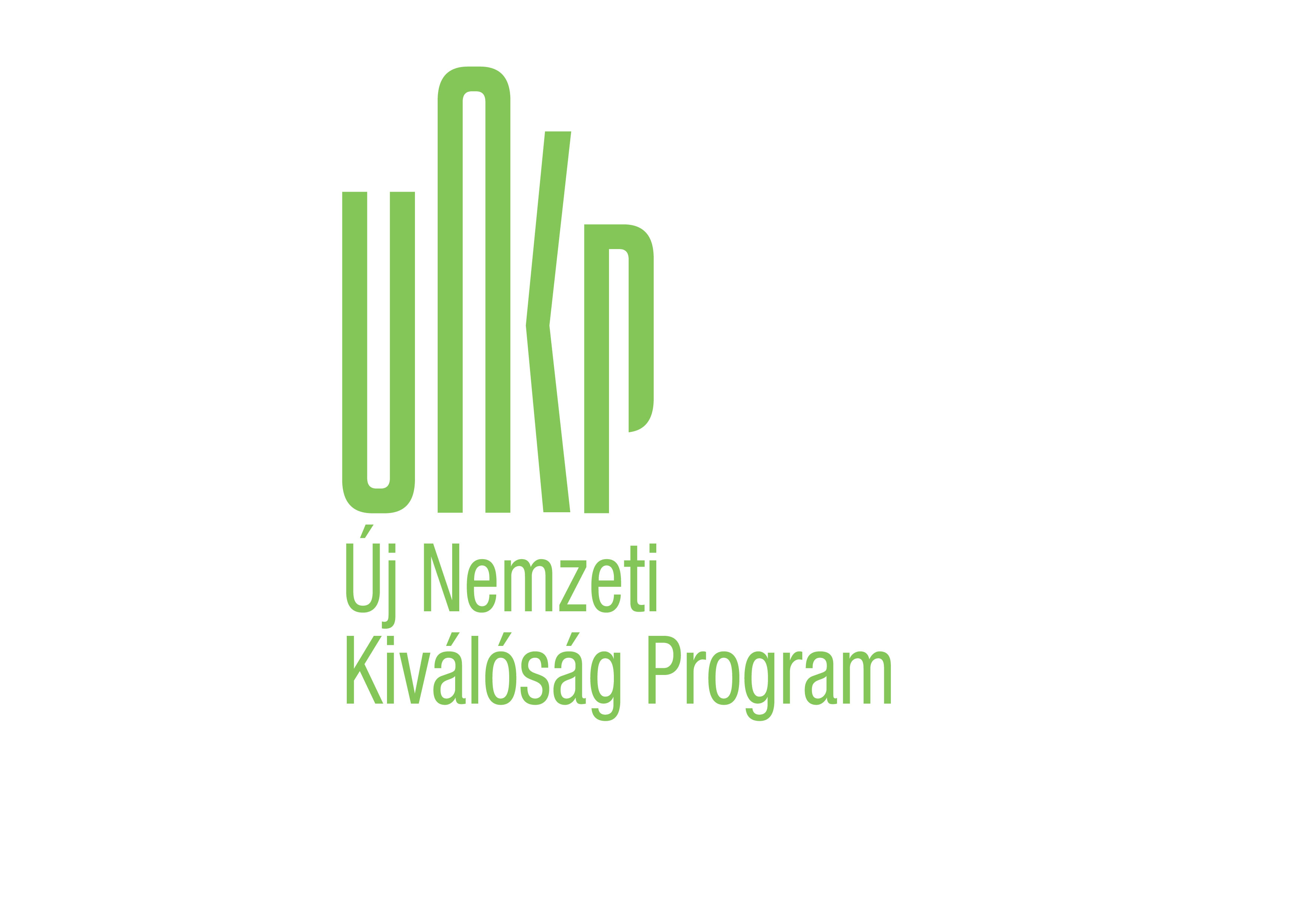}
\end{center}
}
}

\author{
	D\'aniel Keliger, Bal\'azs R\'ath\\
	{\small Department of Stochastics, Institute of Mathematics,}\\
	{\small Budapest University of Technology and Economics}\\
	{\small M\H{u}egyetem rkp. 3., H-1111 Budapest, Hungary;
	}\\
	{\small HUN-REN Alfr\'ed R\'enyi Institute of Mathematics}\\
 {\small Re\'altanoda utca 13-15, H-1053, Budapest, Hungary}\\
	{\small e-mail: \{perfectumfluidum, rathbalazs\}@gmail.com} \\[3mm]
}

\begin{document}
	\maketitle
	
	\begin{abstract}
	We study Markov processes on weighted directed hypergraphs where the state of at most one vertex can change at a time. Our setting is general enough to include simplicial epidemic processes, processes on multilayered networks or even the dynamics of the edges of a graph.
	
	Our results are twofold. Firstly, we prove concentration bounds for the number of vertices in a certain state under mild assumptions. Our results imply that even the empirical averages of subpopulations of diverging but possibly sublinear size are well concentrated around their mean. In the case of undirected weighted graphs, we completely characterize when said averages concentrate around their expected value.
	
	Secondly, we prove (under assumptions which are tight in some significant cases) upper bounds for the error of the N-Intertwined Mean Field Approximation (NIMFA). In particular, for undirected unweighted graphs, the error has the same order of magnitude as the reciprocal of the average degree, improving the previously known state of the art bound which had the same order of magnitude as the reciprocal of the square root of the average degree.

\medskip 



	\end{abstract}

\section{Introduction}

Interacting particle systems on graphs have been studied intensely as a flexible modeling tool. A wide range of applications include epidemic processes \cite{volzproof}, opinion dynamics \cite{Bakhshi2010a}, computer protocols \cite{hht2014}, financial models \cite{credit_risk} and statistical physics \cite{Glauber}.

The prime difficulty emerging while facing such processes is the rapid increase in complexity as the size of the system grows. The exact calculation of occupation probabilities involves the numerical solution of an exponential amount of ordinary differential equations, a task infeasible for modern computers even for  graphs of moderate size.

The two main approaches to circumvent said issue are stochastic simulations and mean field approximations. The advantage of the former is its versatility and relative ease in implementation. However, Monte-Carlo methods require a sizable set of generated samples and assessing the effect of many parameters might be computationally demanding. As for mean field approximations, they offer a lower number of differential equations which need to be numerically integrated only once - as opposed to the stochastic simulations - and in same cases they even provide analytical tractability for certain quantities such as the basic reproduction number in epidemic models. On the other hand, their derivation usually relies on heuristics, making it hard to judge under what circumstances are they reliable.

Our paper focuses on the latter approach regarding the so called N-Intertwined Mean Field Approximation (NIMFA), also known as Quenched Mean Field Approximation \cite{NIMFA2011}. NIMFA is an individual-based method approximating the occupation probabilities of vertices. 

Our contributions are twofold. 

Firstly, we show that, under some conditions on the interaction rates,  (i) the maximal and (ii) the average error of NIMFA are of order $O \left( \frac{1}{d} \right)$, where $d$ is the average degree of the graph (note that (i)  is a stronger result, but we derive (ii) under weaker conditions).
This is a major improvement compared to the previously known $O \left( \frac{1}{\sqrt{d}} \right)$	bounds. We also show that both of the error bounds (i) and (ii) are in some sense tight. In other words, we rigorously clarify the precision (and thus the limitations) of NIMFA for a wide class of cases.

Secondly, we study whether the empirical averages of a  subpopulation $\mathcal{M}$ (say, the ratio of infected among people above the age of 65) concentrate around their expectation. We give an upper bound for the variance of empirical averages which is shown to be $O \left( \frac{1}{\left|\mathcal{M} \right|} \right)$ under mild conditions for systems with pair interactions. Note that this bound is also tight since the variance of the empirical average of
$\left|\mathcal{M} \right|$ i.i.d.\ Bernoulli variables is also comparable to $\frac{1}{\left|\mathcal{M} \right|}$.

Lastly, when restricted to symmetric pair interactions, we give a useful equivalent characterization for the concentration of empirical averages, namely, we show that concentration happens if and only if an associated SI infection process cannot blow up if we start it from a single infected vertex that is uniformly chosen from $\mathcal{M}$.

We work under a quite general setup with Markovian transitions allowing higher order interactions - that is, two or more vertices together affecting the rate at which another vertex changes its state - nonetheless, we only allow vertices to change their states one at the time. Said setup enables modeling interacting particle systems on hypergraphs, multilayered networks, the evolution the graphs themselves or the co-evolution of the states of vertices and the underlying graph at the same time. (Although, the NIMFA is only shown to be accurate in the case when the evolution of the states of vertices does not affect the evolution of the edges of the graph.)

The outline of the article is as follows: In Section \ref{s:related_works} we give a detailed discussion of the literature. In Section \ref{s:setup} we introduce the necessary notation and concepts to state the theorems precisely in Section \ref{s:results}. In Section \ref{s:applications} we give a variety of examples of models for which the stated results are applicable. In Section \ref{section_coupling} we introduce the graphical construction of the original interacting particle system and an auxiliary  branching structure that used in the proofs. Lastly, the proofs are presented in Section \ref{s:proof}.

\section{Related works}
\label{s:related_works}

It is impossible to talk about interacting particle systems without addressing the works of Thomas G. Kurtz. His foundational works show that the Homogenous Mean Field Approximation (HMFA) arises as the limit object of the state densities of Markov processes on the complete graph with fluctuations of order $\frac{1}{\sqrt{N}}$ (where $N$ stands for the number of particles) \cite{kurtz70, kurtz78}.

 Gast et al.\ proved many results related to the so-called refined mean field approximation which provides higher order approximations for occupation probabilities \cite{refined_meanfield}. Their result for NIMFA in \cite{ihomogenous_refined_meanfield} on dense weighted (hyper)graphs show an error bound of order $\frac{1}{N}$ which we generalize to the sparse case. Their method is based on power series expansion and it does not seem to generalize to the sparse case.

In the works of Shidhar and Kar \cite{Sridhar_Kar} and in our previous work \cite{NIMFA_Illes}  Markov processes on weighted networks were studied. In particular in the case of unweighted graphs we obtained an error bound of order $\frac{1}{\sqrt{d}}$ for the accuracy of NIMFA, where $d$ stands for the average degree. In the current paper we present an upper bound of order $\frac{1}{d}$ and show that (under some conditions on the norm of the matrix of rates) this   bound is tight.

There is an intense research about the connection of Markov processes  and graph limits both for jump processes \cite{Delmas2023, Keliger2022} and diffusion processes \cite{Kavita2023, Kuehn2023}.

We refer to the work of Garbe et al.\ \cite{flip_process} on Markov processes \emph{of} graphs which we discuss in detail in Section \ref{s:graph_dynamics}. 

We prove our results using the notion of graphical construction of the interacting particle system in terms of a family of Poisson point processes called instructions. In particular, we make use of the so-called backward construction by recursively tracing the instructions that could potentially affect the
  the state  of a vertex at time $t$ backwards in time. We refer the reader to the introduction of
Section \ref{section_coupling} for brief survey of the literature of the backward construction.

\section{Setup}
\label{s:setup}

In Section \ref{s:transition_rates} we introduce our notation for the transition rates of interacting particle systems with state-dependent higher order interactions. 
In Section \ref{s:hypergraphs} we discuss a notable special case of our setup, namely weighted and unweighted hypergraphs. In Section \ref{subsection_NIMFA_intro} we motivate and introduce the notion of $N$-intertwined mean field approximation (NIMFA).
In Section \ref{subsection_SIS} we discuss a notable special case of our model and the corresponding system of NIMFA differential equations, namely the simplicial SIS process.

\subsection{Transition rates}
\label{s:transition_rates}

Most of the notations in this section are inspired by \cite{ihomogenous_refined_meanfield}.

The set of vertices is denoted by $V$. We denote the cardinality of $V$ by $|V|=N$. $V$ can usually be identified with $[N]=\{ 1, \dots, N\}$. At each time $t \in \mathbb{R}_+$ each vertex is in a state from the set  $\mathcal{S}$ of possible states. We assume that $\mathcal{S}$ is finite.

$\xi_{i,s}(t)$ denotes the indicator of the event that vertex $i\in V$ is in state $s \in \mathcal{S}$ at time $t$. For the corresponding vector we use the convention $\xi_{i}(t):=\left( \xi_{i,s}(t)\right)_{s \in \mathcal{S}}.$

We always assume that the initial conditions $\left(\xi_{i}(0)\right)_{i \in V}$ are  independent, but not necessarily identically distributed.

Vertices can change their states via self-, pair- or higher order interactions. An order $m$ interaction involves $m+1$ vertices: $\underline{j}=(j_1, \dots, j_m) \in V^m$ and $ i \in V$. If the vertices $\underline{j}$ are in state $\underline{\tilde{s}}=(\tilde{s}_1, \dots, \tilde{s}_m) \in \mathcal{S}^m$  and $i$ is in state $s'$ then $\underline{j}$ changes the state of $i$ to $s$ at rate \begin{equation*}
    r^{(m)}_{\underline{j}i;(\underline{\tilde{s}},s') \to (\underline{\tilde{s}},s)}.
\end{equation*}

For example, the rate of an order $0$ (self) interaction has the form $r_{i, s' \to s}^{(0)}$ and it corresponds to the rate at which vertex $i$ changes its state from $s'$ to $s$ by itself. Similarly, the rate of an order $1$ (pair) interaction is of form $r^{(1)}_{ji; (\tilde{s},s') \to (\tilde{s},s)}$ and it corresponds to the rate at which vertex $j$ in state $\tilde{s}$ changes the state of vertex $i$ from $s'$ to $s$.

The upper index $(m)$ is omitted when the order is clear from the context.

Let us stress that this setup only allows transitions where exactly one vertex changes its state at a time.

Without loss of generality we can assume that $r^{(m)}_{\underline{j}i;(\underline{\tilde{s}},s') \to (\underline{\tilde{s}},s)}=0$ whenever two or more indices from the list $j_1,\dots, j_m, i$ have the same value, as these are either contradictory descriptions (like $j_1=j_2$ while $\tilde{s}_1 \neq \tilde{s}_2$ ) or they can be described via a lower order interaction.

Furthermore, for any permutation $\sigma$ of $m$ elements we may assume
\begin{align}
\label{eq:symetry}
r^{(m)}_{\sigma(\underline{j})i; \left(\sigma(\tilde{\underline{s}}),s' \right) \to \left(\sigma(\tilde{\underline{s}}),s' \right) }=r^{(m)}_{\underline{j}i;(\underline{\tilde{s}},s') \to (\underline{\tilde{s}},s)}
\end{align}
as they describe the same event.

We assume 
$$0 \leq m \leq M$$
for some fixed $M \in \mathbb{N}$ which can be interpreted as the maximal order of interactions.

We make use of the convention
\begin{align}
	\label{eq:convenction}
	r_{\underline{j}i; (\underline{\tilde{s}},s) \to (\underline{\tilde{s}},s)}^{(m)}:=&-\sum_{\substack{s' \in \S \\ s' \neq s}} r_{\underline{j}i; (\underline{\tilde{s}},s) \to (\underline{\tilde{s}},s')}^{(m)}
\end{align}
representing the transition rates of vertex $i \in V$ leaving state $s \in \mathcal{S}$ with a minus sign.
Additionally, we denote by
\begin{align}
\label{eq:r_bar}
\bar{r}_{\underline{j}i}^{(m)}:= \sum_{\underline{\tilde{s}} \in \mathcal{S}^m} \sum_{\substack{s,s'\in \mathcal{S} \\ s' \neq s}} r_{\underline{j}i; (\underline{\tilde{s}},s) \to (\underline{\tilde{s}},s')}^{(m)}	
\end{align}
 the total rate at which the $m-$tuple of vertices $\underline{j}=(j_1,\dots, j_m)$ can potentially influence vertex $i$ via an order $m$ interaction, including ``phantom'' interactions that do not actualize. 
The total incoming rate of vertex $i$ is denoted by
\begin{align}
\label{eq:delta^m}
\delta_{i}^{(m)}:=& \frac{1}{m!} \sum_{\underline{j} \in V^m} \bar{r}_{\underline{j}i}^{(m)}.
\end{align}
where the factor $ m!$ takes into account the over counting of the permutations of the vertices $j_1, \dots, j_m$. Due to property \eqref{eq:symetry} in these type of sums one can replace $\frac{1}{m!}\sum_{ \underline{j} \in V^m}$ with $\sum_{j_1<\dots <j_m}$.

 The concentration and error bounds that we will prove, in general, do not depend on self-interactions as they do not create correlations between vertices, thus, $\delta_i^{(m)}$ only interests us for $1 \leq m \leq M$ (excluding $m=0$), thus we define the maximal incoming rate as follows:
\begin{align}
	\label{eq:delta_max}
	\delta_{max}:= \max_{1 \leq m \leq M} \max_{i \in V}\delta_{i}^{(m)}.
\end{align}

Let $i \neq j \in V$. The \emph{influence} of vertex $j$ on vertex $i$, denoted by $\tilde{r}_{ji}$, is defined as the total rate of interactions in which $j$ could participate in that could potentially have an effect on the state of vertex $i$. With formulas:
\begin{align}
	\label{eq:r_tilde}
	\tilde{r}_{ji}:=\sum_{m=1}^{M} \frac{1}{m!}\sum_{\substack{\underline{k} \in V^m \\ j \in \underline{k}}}\bar{r}^{(m)}_{\underline{k}i},
\end{align}
where $j \in \underline{k}$ means that there is an index $1 \leq l \leq m$ such that $j=k_l.$ 

Note that if $M=1$ then $\tilde{r}_{ji}=\bar{r}_{ji}$.

The maximal influence is denoted by
\begin{align}
\label{eq:r_max}
\tilde{r}_{max}:=\max_{i,j \in V }  \tilde{r}_{ij}.
\end{align}

We pay special attention to the case when we only have pair interactions ($M=1$).
In this case the total rates form a matrix \begin{equation}\label{def_R_matrix}
R= \left(\bar{r}_{ji} \right)_{j,i \in V}.
\end{equation}
It follows from our assumptions that the diagonal elements of $R$ are zeros: $\bar{r}_{ii}=0$ for all $i \in V$.
We will call the rates  \emph{symmetric} if the corresponding matrix is symmetric, i.e., if $R^\intercal=R$.

$\|R \|_2$ refers to the induced $L^2$ operator norm of $R$.

\subsection{Weighted and unweighted hypergraphs}
\label{s:hypergraphs}
An important special case of the general setting described in Section \ref{s:transition_rates} is when we have dynamics on a  hypergraph with weights
$$(w_{\underline{j}i}^{(m)})_{ \substack{  \underline{j} \in V^m \\ i \in V \\ m \in \{1,\dots,M\} }}. $$

The setup of Section \ref{s:transition_rates} also includes other possible models such as dynamics on multilayered networks or edge dynamics, cf.\ Section \ref{s:applications}.

In the special case of weighted hypergraphs the transition rates can be factorized into
\begin{align}
\label{eq:r_factorized}
r_{\underline{j}i;(\tilde{\underline{s}},s') \to (\underline{\tilde{s}},s)}^{(m)}=w_{\underline{j}i}^{(m)}q_{(\tilde{\underline{s}},s') \to (\underline{\tilde{s}},s)}^{(m)},
\end{align}
separating the effect of the connectivity between vertices $j_1,\dots,j_m,i$ and the effect of their state on the transition rates.
We have
\begin{align}
\label{eq:bar_q}
	\bar{r}_{\underline{j}i}^{(m)} \stackrel{\eqref{eq:r_bar}}{=}  w_{\underline{j}i}^{(m)} \underbrace{\sum_{\underline{\tilde{s}} \in \mathcal{S}^m} \sum_{\substack{s,s'\in \mathcal{S} \\ s' \neq s}}q_{(\tilde{\underline{s}},s') \to (\underline{\tilde{s}},s)}^{(m)}}_{=:\bar{q}^{(m)}}
\end{align}
making $\bar{r}_{\underline{j}i}^{(m)}$ proportional to $w_{\underline{j}i}^{(m)}$ for fixed $m$.

For unweighted hypergraphs, let $a_{\underline{j}i}^{(m)}$ denote the indicator of $j_1,\dots, j_m,i$ forming a hyperedge. The $m$-in-degree of vertex $i \in V$ and the average $m$-in-degree are defined as
\begin{equation*}
	d^{(m)}_i:=\frac{1}{m!} \sum_{\underline{j} \in V^m}a_{\underline{j}i}^{(m)}, \qquad
	\bar{d}^{(m)}:= \frac{1}{N} \sum_{i \in V}d^{(m)}_i.
\end{equation*}

Without further normalization imposed, the total potential incoming weights through $m$-interactions would be proportional to the in-degrees $d^{(m)}_i$. Thus, when we have a graph sequence with diverging average degree then the dynamics would eventually become too fast. Therefore, we will apply the normalization
\begin{align}
	\label{eq:w_normalisation}
	w^{(m)}_{\underline{j}i}=\frac{1}{ \bar{d}^{(m)}}a^{(m)}_{\underline{j}i}.
\end{align}
For instance in the $M=1$ case (i.e., if we only have self-interactions and pairwise interactions) this leads to 
\begin{equation}\label{simple_unweighted_graphs_R}
R\stackrel{\eqref{def_R_matrix}}{=}\frac{\bar{q}}{\bar{d}}A    \end{equation}
where $A$ is the adjacency matrix of the graph and we use the shorthand $\bar{d}=\bar{d}^{(1)}$.

Furthermore,
\begin{align*}
	\delta_i^{(m)} \stackrel{\eqref{eq:delta^m}}{=}\frac{\bar{q}^{(m)}}{\bar{d}^{(m)}} d_i^{(m)}.
\end{align*}

Assuming $\delta_{max} \leq K$ (cf.\ \eqref{eq:delta_max}) in the setting of weighted hypergraphs is the same as requiring that
\begin{align}
	\label{eq:upper_regualr}
	\forall i \in V, 1 \leq m \leq M \  d^{(m)}_i \leq \frac{K}{\bar{q}^{(m)}} \bar{d}^{(m)},
\end{align}
or in other words, requiring that the hypergraph is upper regular.

Note that for a (possibly oriented) unweighted graph ($M=1$) the maximal influence (cf.\ \eqref{eq:r_max}) satisfies \begin{equation*} \tilde{r}_{max}=\frac{1}{\bar{d}}.\end{equation*}

\subsection{The N-intertwined Mean Field Approximation}\label{subsection_NIMFA_intro}

The models described in Section \ref{s:transition_rates} are Markov chains with state space $\mathcal{S}^V.$ One could theoretically calculate the occupation probabilities via Kolmogorov's forward equation. However, that would require the numerical integration of a system of $\mathcal{S}^N$ ODEs, an infeasible task even for moderate sized graphs.

To mitigate the problem, we follow an individual-based approach where we aim to approximate the marginal probabilities
\begin{align}
\label{eq:y}
y_{i,s}(t):= \pr \left(\xi_{i,s}(t)=1 \right)= \E \left(\xi_{i,s}(t) \right)
\end{align}
instead of the joint probabilities $\pr \left(\xi_{i_1,s_1}(t)=1, \dots, \xi_{i_N,s_N}(t)=1 \right)$
as it only requires tracking $\left| \mathcal{S} \right| \cdot N$ variables, a drastic reduction of complexity. In plain words, $y_{i,s}(t)$ is the probability that vertex $i$ is in state $s$ at time $t$.

In order to write down the dynamics for $y_{i,s}(t)$, one should apply the infinitesimal generator $\mathcal{A}.$
\begin{align*}
\mathcal{A}\xi_{i,s}(t)=&\lim_{\Delta t \to 0_{+}} \frac{1}{\Delta t} \E \left( \left.\xi_{i,s}(t+\Delta t)-\xi_{i,s}(t) \right| \mathcal{F}_t \right)\\
=&\sum_{m=0}^{M} \frac{1}{m!} \sum_{\substack{\underline{j} \in V^m \\ \underline{\tilde{s}} \in \mathcal{S}^{m} \\ s' \in \mathcal{S} }} r^{(m)}_{\underline{j}i; (\underline{\tilde{s}},s') \to (\underline{\tilde{s}},s)}\xi_{i,s'}(t)\prod_{l=1}^{m} \xi_{j_l,\tilde{s}_l}(t)
\end{align*}
where $\mathcal{F}_t$ is the filtration of the process. Note that the negative terms corresponding to vertex $i$ leaving state $s$ are included via the convention \eqref{eq:convenction}. 

After taking the expectation of both sides, we get
\begin{align}
\label{eq:y_ODE}
\frac{\d}{\d t}y_{i,s}(t)=\sum_{m=0}^{M} \frac{1}{m!} \sum_{\substack{\underline{j} \in V^m \\ \underline{\tilde{s}} \in \mathcal{S}^{m} \\ s' \in \mathcal{S} }} r^{(m)}_{\underline{j}i; (\underline{\tilde{s}},s') \to (\underline{\tilde{s}},s)} \E \left(\xi_{i,s'}(t)\prod_{l=1}^{m} \xi_{j_l,\tilde{s}_l}(t) \right).	
\end{align}

The problem with \eqref{eq:y_ODE} is that it does not form a closed system of ODEs, it is not enough for determining the values of $y_{i,s}(t)$. One way to achieve closure of \eqref{eq:y_ODE} is to assume that $\xi_{i,s'}(t),\xi_{j_1,\tilde{s}_1}, \dots, \xi_{j_m,\tilde{s}_m}$ are approximately independent from each other:
\begin{equation}\label{decorrrelation}
\E \left(\xi_{i,s'}(t)\prod_{l=1}^{m} \xi_{j_l,\tilde{s}_l}(t) \right) \approx \E \left(\xi_{i,s'}(t) \right) \prod_{l=1}^{m} \E \left( \xi_{j_l,\tilde{s}_l}(t) \right)=y_{i,s'}(t) \prod_{l=1}^{m} y_{j_l,\tilde{s}_l}(t).
\end{equation}

Inserting this approximate identity into \eqref{eq:y_ODE}
 gives the approximate, but closed ODE system called quenched mean field approximation or N-Intertwined Mean Field Approximation (NIMFA) \cite{NIMFA2011}:
\begin{align}
	\label{eq:NIMFA}
	\frac{\d}{\d t} z_{i,s}(t)=&\sum_{m=0}^{M} \frac{1}{m!} \sum_{\substack{\underline{j} \in V^m \\ \underline{\tilde{s}} \in \mathcal{S}^{m} \\ s' \in \mathcal{S} }} r^{(m)}_{\underline{j}i; (\underline{\tilde{s}},s') \to (\underline{\tilde{s}},s)}z_{i,s'}(t)\prod_{l=1}^{m} z_{j_l,\tilde{s}_l}(t) \\
	\nonumber
	z_{i,s}(0)=&y_{i,s}(0).
\end{align}
The heuristic suggests that if the states of the vertices are approximately independent then we must have $y_{i,s}(t) \approx z_{i,s}(t)$.

Since the right hand side of \eqref{eq:NIMFA} is locally Lipschitz, it  exhibits a unique maximal solution. Note that for each $i \in V$ $y_{i}(t):= \left(y_{i,s}(t) \right)_{s \in \mathcal{S}}$ is a probability vector. This property is inherited by $z_{i}(t):= \left( z_{i,s}(t) \right)_{s \in \mathcal{S}}$ for all $t\geq 0$ when the initial conditions are $z_{i}(0)=y_{i}(0).$ This is due to the fact that \eqref{eq:NIMFA} also exhibits a stochastic representation. In Section \ref{section_coupling} we construct an an auxiliary branching structure which produces an $\mathcal{S}$-valued random variable $\widetilde{\sigma}_{i}(t)$ such that $z_{i,s}(t)=  \mathbb{P}(\widetilde{\sigma}_{i}(t)=s) $. This also means that, $z_{i,s}(t)$ cannot blow up for $t \geq 0$, making the solutions well-defined on $\mathbb{R}^{+}.$

There are several advantages of \eqref{eq:NIMFA}.

Firstly, it is only composed of $|\mathcal{S} |N$ ODEs out of which lne only has to treat $(|\mathcal{S} |-1)N$ of them due to the conservation of probabilities giving $\sum_{s \in \mathcal{S}} z_{i,s}(t)=1.$ In contrast, the exact Kolmogorov forward equations corresponding to our Markov process form a system of $|S|^N-1$ ordinary differential equations  which are numerically infeasible even for moderate sized $N$.

Secondly, for concrete models (see \eqref{eq:SIS_NIMFA} for the SIS process) \eqref{eq:NIMFA} might have a relatively simple form enabling even closed form analytical results (such as the critical value of the infection parameter \cite{NIMFA2011}).

Lastly, one can use \eqref{eq:NIMFA} as a stepping stone for deriving a limit object which usually takes the form of an integro-differential equation, cf.\ \eqref{eq:SIS_PDE} for the SIS process.

\subsection{The simplicial SIS process}\label{subsection_SIS}

The most simple example of a complex contagion model is the simplicial SIS model \cite{simplicial_SIS}. Vertices can be in two states: susceptible ($S$) and infected ($I$), making the state space $\mathcal{S}=\{S,I\}$. For $1 \leq m \leq M$ if vertices $\underline{j}=(j_1, \dots, j_m) \in V^m$ are infected and $i \in V$ is susceptible, then $i$ becomes infected at rate 
$$r^{(m)}_{\underline{j}i; ((I,\dots,I),S) \to ((I,\dots,I),I)}= \bar{r}_{\underline{j}i}^{(m)}.$$

Recovery (i.e., curing) happens via the following self-interaction: an infected vertex $i \in V$ recovers at rate $r_{i,I \to S}=\bar{r}_{i}.$

When it comes to \eqref{eq:NIMFA}, it is enough to consider the  variables $z_{i,I}(t), i \in V$ since we have $z_{i,S}(t)=1-z_{i,I}(t)$, resulting in
\begin{align}
\label{eq:simplicial_SIS_NIMFA}
\frac{\d}{\d t}z_{i,I}(t)=-\bar{r}_{i}z_{i,I}(t)+(1-z_{i,I}(t)) \sum_{m=1}^{M} \frac{1}{m!} \sum_{\underline{j} \in V^m} \bar{r}_{\underline{j}i}^{(m)} \prod_{l=1}^{m} z_{j_l,I}(t).
\end{align}

In the $M=1$ case we get back the classical version of the SIS model \cite{NIMFA2011} on the weighted graph $R$ with recovery rates $(\bar{r}_{i})_{i \in V}$. Here, \eqref{eq:simplicial_SIS_NIMFA} reduces to
\begin{align}
	\label{eq:SIS_NIMFA}
	\frac{\d}{\d t}z_{i,I}(t)=-\bar{r}_{i}z_{i,I}(t)+(1-z_{i,I}(t))  \sum_{j \in V} \bar{r}_{ji} z_{j,I}(t).
\end{align}

When the recovery rates $\bar{r}_{i}$ are zero, we will refer to the model as the simplicial SI process.

\begin{remark}[Integro-differential equation limit of SIS process as $N \to \infty$]
When there are some functions $\gamma, W$ such that $\bar{r}_{i} \approx \gamma \left( \frac{i}{N} \right)$ and $\bar{r}_{ji} \approx \frac{1}{N}W \left( \frac{j}{N},\frac{i}{N} \right)$, then it is possible to approximate $z_i(t)$ as $z_{i}(t) \approx u \left(t, \frac{i}{N} \right)$ where $u$ satisfies
\begin{align}
\label{eq:SIS_PDE}
\partial_t u(t,x)=-\gamma(x)u(t,x)+(1-u(t,x)) \int_{0}^{1} W(y,x)u(t,y) \d y.
\end{align}

For the precise statement and conditions, see \cite{Kuehn2023}.
\end{remark}

\section{Results}
\label{s:results}

 In Section \ref{s:concentration} we state our results about the concentration
 of empirical averages of form $\frac{1}{\left|\mathcal{M} \right|}\sum_{i \in \mathcal{M}} \xi_{i,s}(t)$, where $\mathcal{M} \subseteq V$.
In  Section  \ref{s:l1_error} we give upper and lower bounds on the individual error 
$\left|y_{i,s}(t)-z_{i,s}(t)\right|$ in terms of $\tilde{r}_{max}$ (cf.\ \eqref{eq:r_max}).
In Section \ref{s:average_error} we give upper and lower bounds on the average error $\frac{1}{N}\sum_{i \in V} \left |y_{i,s}(t)-z_{i,s}(t) \right|$ in the $M=1$ case in terms of $\frac{1}{N}      \sum_{i,j \in V}\bar{r}_{ij }^2$.

\begin{remark}[Average degree assumptions]
Note that in the case of unweighted graphs  (cf.\ \eqref{simple_unweighted_graphs_R}) we will need the average degree $\bar{d}$ to  be  big for the error of the approximation \eqref{decorrrelation}  (and therefore the error of the NIMFA approximation results of Sections 
\ref{s:l1_error} and \ref{s:average_error}) to be small, but the concentration results of Section \ref{s:concentration} do not require $\bar{d}$ to be big.
More specifically, the average error of the NIMFA approximation does not vanish as $N \to \infty$ if the underlying graph is  $d$-regular with fixed $d$ (cf.\ Theorem \ref{t:l1_charachterisation}), but the concentration bound \eqref{eq:concentration_bound}  is essentially optimal in this case, cf.\ Remark \ref{remark_tight_concentration}.
\end{remark}

\subsection{Concentration results}
\label{s:concentration}
Here we present our results concerning the concentration of macroscopic or mesoscopic averages around their expectation. We show that under mild conditions, the statistics of large  subpopulations (say, a city or an age group) will have small fluctuations around a deterministic value.

Some of our results are formulated in terms of a random set $\mathcal{H}_i(t)$ corresponding to those vertices which could potentially influence the state of vertex $i$ at time $t$. We call $\mathcal{H}_i(t)$ the \emph{information set} of vertex $i$ at time $t$. In our proofs the set
$\mathcal{H}_{i}(t)$ naturally appears in the graphical construction of the interacting particle systems that we study (cf.\ Section  \ref{section_coupling}), but in order
to state our theorems, we merely need to define the  evolution of the set $\mathcal{H}_{i}(t)$:

\begin{definition}[Information set]\label{def_info_set}
Let $\mathcal{H}_{i}(t)$ be a time-dependent random subset of $V$ with initial condition $\mathcal{H}_{i}(0)=\{ i\}$ which evolves according to the following Markovian dynamics:
If the set is in state $\mathcal{H}_{i}(t)=H$ then for every $k \in H$, every $m \in \{1,\dots,M \}$ and every $\underline{j}=(j_1, \dots, j_m) \in V^m$ that satisfies $j_1<j_2<\dots<j_m$ it jumps to state $H \cup \{j_1, \dots, j_m \}$ at rate $\bar{r}_{\underline{j}k}^{(m)}$ (cf.\ \eqref{eq:r_bar}). 
\end{definition}
Observe that information sets increase, i.e., if $t' \leq t$ then $\mathcal{H}_{i}(t') \subseteq \mathcal{H}_{i}(t)$.

Our first lemma quantifies the error of the approximation \eqref{decorrrelation} using information sets. Loosely speaking, if the information sets don't overlap then independence is guaranteed.

\begin{lemma}[Propagation of chaos]
\label{t:chaos}
Let $i_1, \dots, i_m \in V$ be distinct vertices and $s_{1}, \dots, s_{m} \in S$ arbitrary. Then we have
\begin{align}
\label{eq:chaos}
\left| \E \left( \prod_{l=1}^{m} \xi_{i_l,s_l}(t) \right)-\prod_{l=1}^{m} \E \left( \xi_{i_l,s_l}(t) \right) \right| \leq \pr \left( \bigcup_{\substack{i,j \in \{i_1, \dots, i_m\} \\ i \neq j}} \{\mathcal{H}_{i}(t) \cap \mathcal{H}_j(t) \neq \emptyset\}  \right)
\end{align}	
\end{lemma}
We will prove Lemma \ref{t:chaos} in Section \ref{subsection_proofs_concentration}.

We are ready to state the first main result of Section \ref{s:concentration}.

\begin{theorem}[Concentration bounds]
\label{t:concentration_bound}
Let $ \emptyset \neq \mathcal{M} \subseteq V$ be a set of vertices. Let us denote the fraction of vertices from $\mathcal{M}$ being in state $s$ at time $t$ by
\begin{equation}\label{emp_average}
	\bar{\xi}^{\mathcal{M}}_s(t):=\frac{1}{\left|\mathcal{M} \right|}\sum_{i \in \mathcal{M}} \xi_{i,s}(t).
\end{equation}
Then for any $s \in \mathcal{S}$ we have
\begin{align}
\label{eq:concentration}
 \Var \left( \bar{\xi}^{\mathcal{M}}_s(t)\right)& \leq \frac{1}{\left| \mathcal{M} \right|^2} \sum_{i, j \in \mathcal{M}} \pr \left( \mathcal{H}_i(t) \cap \mathcal{H}_j(t) \neq \emptyset \right).
\end{align}

Recall \eqref{def_R_matrix}. In the $M=1$ case we have
\begin{align}
\label{eq:concentration_bound}
 \Var \left( \bar{\xi}^{\mathcal{M}}_s(t)\right) &\leq   \frac{1}{\left|\mathcal{M} \right|}e^{2 \| R \|_2 t}.
\end{align}
\end{theorem}

We will prove Theorem \ref{t:concentration_bound}  in Section \ref{subsection_proofs_concentration}.

\begin{remark}[Sharpness]\label{remark_tight_concentration}
The bound \eqref{eq:concentration} is tight in the following sense: since $\mathcal{H}_i(0)=\{i\}$,  at $t=0$ the r.h.s.\ of \eqref{eq:concentration} is equal to $\frac{1}{\left| \mathcal{M} \right|}$.
Now if we consider the SI process with i.i.d.\ initial conditions $\pr \left( \xi_{i,I}(0)=1 \right)=\frac{1}{2}$ then at $t=0$ the l.h.s.\ of \eqref{eq:concentration} is equal to
$\frac{1}{4 \left| \mathcal{M} \right|}$.
\end{remark}

\begin{corollary}
\label{c:concentration}
It is  natural to assume $\|R \|_2=O(1)$ since this assumption holds e.g.  in the case of upper regular graphs (cf.\ \eqref{eq:upper_regualr}) or $W$-random graphs with density $ \gg \frac{\log N}{N}  $, cf.\ \cite{randomgraphspectra}. In this case the r.h.s.\ of \eqref{eq:concentration_bound} simply reduces to $O \left(\frac{1}{\left|\mathcal{M} \right|}\right)$ where the implied constant on $O(\cdot)$ depends on $t$.
\end{corollary}

\begin{corollary}\label{corollary_log}
Assume  $\|R \|_2=o\left( \log \left|\mathcal{M} \right| \right)$ and $|\mathcal{M} |=\Omega\left(N^{\alpha}\right)$ for some $0<\alpha \leq 1$. Under this assumption, for any $0<\varepsilon<1$, the r.h.s.\ of  \eqref{eq:concentration_bound} is
$O \left(\left|\mathcal{M} \right|^{-(1-\varepsilon)}\right)$, where the implied constant
in $O(\cdot)$ depends on $t$ and $\varepsilon$.
\end{corollary}

The assumption  $\|R \|_2=o\left( \log (N) \right)$   holds with high probability for the Erdős-Rényi graph $\mathcal{G} \left(N, \frac{\lambda}{N} \right)$, cf.\ \cite{ER_spectrum}, thus Corollary \ref{corollary_log} can be applied.
However, as we will see in Proposition \ref{propp:Erdos_Renyi}, the tight $O \left( \frac{1}{\left| \mathcal{M} \right|} \right)$ bound can also be achieved, despite $\|R \|_2 \to \infty$ as $n \to \infty$.

The scaling that we employ in 
 \eqref{ER_rates} below is formally different from our usual convention \eqref{eq:w_normalisation}, but  the two normalizations have the same order of magnitude  since in $\mathcal{G}\left(N, \frac{\lambda}{N} \right)$ the average degree satisfies $\bar{d} \approx \lambda$ and in Proposition \ref{propp:Erdos_Renyi} we assume that $\lambda$ is a fixed positive constant that does not depend on $N$.
 
\begin{proposition}[Erdős-Rényi concentration bound]
\label{propp:Erdos_Renyi}
Take a process with self interactions and pair interactions with rates 
\begin{equation}\label{ER_rates}
r_{ji;(\tilde{s},s') \to (\tilde{s},s)}=a_{ji}q_{(\tilde{s},s') \to (\tilde{s},s)},
\end{equation}
similarly to the convention in \eqref{eq:r_factorized}, where $\left(a_{ji}\right)_{j,i \in V}$  is the adjacency matrix  of an  Erdős-Rényi graph $\mathcal{G}\left(N, \frac{\lambda}{N} \right)$.  Take any (deterministic) subset $\emptyset \neq \mathcal{M} \subseteq V$. Then using the notation $\bar{q}:=\bar{q}^{(1)}$ from \eqref{eq:bar_q} we have 
\begin{equation}\label{ER_var_bound}
\mathbb{P}\left[ 
\Var \left(  \bar{\xi}^{\mathcal{M}}_{s}(t) \, \left| \, \mathcal{G}\left(N, \frac{\lambda}{N} \right) \right. \right) \leq  \frac{K}{ |\mathcal{M}|}e^{2 \lambda \bar{q}  t} \right] \geq 1-\frac{1}{K}.
\end{equation}
\end{proposition}

We will prove Proposition \ref{propp:Erdos_Renyi}  in Section \ref{subsection_proofs_concentration}.

Recall the notion of the maximal incoming rate $\delta_{max}$ from \eqref{eq:delta_max}.
\begin{remark}[Symmetry matters]\label{remark_nonsymm}
Note that by \eqref{eq:concentration_bound} it is enough to assume
$\delta_{max}=O(1) $  for the variance to be $O \left(\frac{1}{\left| \mathcal{M} \right|} \right)$ if $R$ is symmetric (since  $\| R\|_2 \leq \|R \|_{\infty}=\delta_{max} $ holds if $R^\intercal=R$), but $\delta_{max}=O(1) $ does not imply the $O \left(\frac{1}{\left| \mathcal{M} \right|} \right)$ variance bound in the non-symmetric case, as we now explain. 
 Take an SI process on the directed star graph where the edges are pointing outwards the center. That is, $r_{1j; (I,S) \to (I,I)}=1$ for $j>1$ and $\bar{r}_{ij}=0$ otherwise (i.e., vertex $1$ can infect others but not the other way around). If initially vertex $1$ is infected with probability $\frac{1}{2}$ while the others are initially susceptible then the ratio of infected vertices at $t=1$  fluctuates a lot based on whether the center is infected or not, despite  the fact that the maximal in-degree $\delta_{max}$ is equal to  $1$.
\end{remark}

To take care of the problem mentioned in Remark \ref{remark_nonsymm}, we provide an upper bound that works for directed hypergraphs as well. Recall the notion of the maximal influence $\tilde{r}_{max}$   from \eqref{eq:r_max}.

\begin{theorem}[Concentration bounds using $\delta_{max}$ and $\tilde{r}_{max}$]
\label{t:concentration_delta_max}
 For any $i \neq j \in V$ we have 
\begin{equation}\label{intersect_Hijt_bound}
\pr \left( \mathcal{H}_{i}(t) \cap \mathcal{H}_{j}(t) \neq \emptyset \right) \leq t \cdot e^{\delta_{max}Mt}(1+e^{\delta_{max}Mt})\cdot \tilde{r}_{max}.
\end{equation}
\end{theorem}
We  prove Theorem \ref{t:concentration_delta_max} in Section  \ref{subsection_proofs_concentration}.
To put it more simply, the r.h.s.\ of \eqref{intersect_Hijt_bound} is $O(\tilde{r}_{max})$, 
where the implied constant in $O( \cdot)$ depends on $t$, $\delta_{max}$ and $M$.

So far we only provided an upper bound for the variance of $\bar{\xi}^{\mathcal{M}}_s(t)$. In the case of $M=1$ and with symmetric rates, however, more can be said. Namely, it is possible to characterize a sequence of rate matrices $R$ and subsets of vertices $\mathcal{M} \subseteq V$ for which the variance of empirical averages goes to zero, regardless the choice of (independent) initial conditions or the underlying process.

\begin{theorem}[Characterization of concentration using  information sets]
	\label{t:concentration_characterisation}
	Let $M=1$. Given a sequence of symmetric rate matrices $R(n), n \in \mathbb{N}$ and sample sets $\emptyset \neq \mathcal{M}(n) \subseteq V(n)$, 
		the following two statements are equivalent:
	
	\begin{itemize}
		\item For all Markov processes satisfying \eqref{eq:r_bar}  and for all sequences of independent initial conditions $\left(\xi_{i}(0) \right)_{i \in V(n)}$ we have
		$$ \forall \, t \geq 0, \ s \in \mathcal{S} \ \lim_{n \to \infty} \Var \left(\bar{\xi}^{\mathcal{M}(n)}_{s}(t) \right)=0. $$
		\item
		\begin{equation}\label{nonintersect_lim}
  \forall \, t \geq 0 \ \lim_{n \to \infty} \frac{1}{\left|\mathcal{M}(n) \right|^2}\sum_{i,j \in \mathcal{M}(n)} \pr \left(\mathcal{H}_i(t) \cap \mathcal{H}_i(t) \neq \emptyset \right)=0. 
  \end{equation}
	\end{itemize}
\end{theorem}
We will prove Theorem \ref{t:concentration_characterisation} in Section  \ref{subsection_proofs_concentration}.

\begin{conjecture}
		We believe that Theorem \ref{t:concentration_characterisation} is true even without assuming symmetry or $M=1$.
	
	Heuristically, if $ \frac{1}{\left|\mathcal{M}(n) \right|^2}\sum_{i,j \in \mathcal{M}(n)} \pr \left(\mathcal{H}_i(t) \cap \mathcal{H}_i(t) \neq \emptyset \right) \geq c>0$  then $  \mathcal{H}_i(t)$ and $\mathcal{H}_j(t)$ has a non-empty intersection for two independently and uniformly chosen $i,j \in \mathcal{M}(n)$ with non-vanishing probability, making it possible for information to flow to both $i$ and $j$ from the same source, hence making them correlated (cf.\ Section \ref{section_coupling} for the interpretation of information sets in terms of the graphical construction of the interacting particle system).
\end{conjecture}

\begin{remark} We will later see in Claim \ref{c:H}  that in the case of  symmetric pair interactions the information set $\mathcal{H}_i(t)$ evolves as the infected set of an SI process on the weighted graph $R$ started from a single infected individual at vertex $i$.
Moreover, by \eqref{sum_of_nonintersect}, if $R(n)=R^{\intercal}(n)$ then
the condition \eqref{nonintersect_lim} is equivalent to
\begin{equation*}
 \forall \, t \geq 0 \ \lim_{n \to \infty}     \frac{1}{\left| \mathcal{M}(n) \right|}\sum_{i \in \mathcal{M}(n)} \E \left( \frac{\left| \mathcal{H}_i(2t) \cap \mathcal{M}(n) \right|}{\left|\mathcal{M}(n) \right|} \right)=0,
\end{equation*}
which simply means that the SI process on the weighted graph $R$ started from a uniform vertex of $\mathcal{M}(n)$ cannot blow up in finite time. 
    
\end{remark}

\subsection{Mean field approximation - $\ell^{\infty}$ error}
\label{s:l1_error}

In this chapter we discuss the maximal error  between the occupation probabilities $y_{i,s}(t)$ (cf.\ \eqref{eq:y}) and their NIMFA approximation $z_{i,s}(t)$ (cf.\ \eqref{eq:NIMFA}).

The main heuristic is that if any vertex has a small influence on any other vertex
 then the 
states of vertices will be almost independent, making the error of NIMFA small.

Firstly, we show that if the total incoming interaction rates are uniformly bounded in the set of vertices - i.e., there are no fast changing nodes - then  the error of NIMFA is $O(\tilde{r}_{max})$ (where the maximal influence $\tilde{r}_{max}$ is defined in \eqref{eq:r_tilde}) \emph{uniformly} in $i \in V$. Therefore, it is possible to given an accurate approximation at the individual level. 

\begin{theorem}[NIMFA $\ell^{\infty}$ upper bound in terms of $\delta_{max}$ and $\tilde{r}_{max}$]
\label{t:mean_field_uniform} We have
\begin{align}
\label{eq:r_max_bound}
\max_{\substack{i \in V \\ s \in \mathcal{S}}} \left|y_{i,s}(t)-z_{i,s}(t)  \right| \leq 2Mt \cdot e^{2M \delta_{max}t} \cdot \tilde{r}_{max} 
\end{align}
\end{theorem}
\noindent We prove Theorem \ref{t:mean_field_uniform} in Section \ref{subsection_proofs_l1_error}.
To put it more simply, the r.h.s.\ of \eqref{eq:r_max_bound} is
$O \left(\tilde{r}_{max} \right)$,
where the implied constant in $O(\cdot)$ depends on  $\delta_{max}$, $M$ and $t$.

The following theorem contains lower bounds that complement the upper bounds of Theorem \ref{t:mean_field_uniform}.

\begin{theorem}[NIMFA $\ell^{\infty}$ lower bound in terms of  $\tilde{r}_{max}$]
\label{t:mean_field_lower_bound} 

$ $

\begin{enumerate}[(I)]
\item\label{lilb1}
For any $\left( \bar{r}_{\underline{j}i}^{(m)}\right)_{\substack{0 \leq m \leq M \\ j_1, \dots, j_m, i \in V}}$ there is a process satisfying \eqref{eq:r_bar}  and an appropriately chosen set of initial conditions such that
\begin{align}
\label{eq:mean_field_lower_bound_general}
 \exists \, i \in V, \ s \in \mathcal{S} \ \left|y_{i,s}(1)-z_{i,I}(1) \right| = \frac{1}{2} \left(1-e^{-\frac{1}{2} \tilde{r}_{max}} \right)^2.
\end{align}

\item\label{lilb2}
Assume $\bar{r}_{i  j}=\frac{a_{ij}}{\bar{d}}$ for some unweighted graph with (symmetric) adjacency matrix $\left(a_{ij} \right)_{i,j \in V}$ with average degree $\bar{d} \geq 1.$ Pick any vertex  $i$ with degree $d_i \geq \bar{d}$. Take the SI process with initial condition $\pr \left( \xi_{i,I}(0)=1 \right)=\frac{1}{2}$ and  $\xi_{j,S}(0)=1$ ($j \neq i$). Then
\begin{align}
\label{eq:mean_field_lower_bound_tight}
 \left|y_{i,I}(1)-z_{i,I}(1) \right| \geq\frac{1}{32} \frac{1}{\bar{d}}= \frac{1}{32}\tilde{r}_{max}.
\end{align}
\end{enumerate}
\end{theorem}
We prove Theorem \ref{t:mean_field_lower_bound} in Section \ref{subsection_proofs_l1_error}.

The identity \eqref{eq:mean_field_lower_bound_general} together with Theorem \ref{t:mean_field_uniform} implies that
if we consider a sequence of models indexed by $n$ and
 assume $\delta_{max}(n)=O(1)$ then $\tilde{r}_{max}(n) \to 0$ characterizes whether the $\ell^{\infty}$ error vanishes or not as $n \to \infty$.

The lower bound \eqref{eq:mean_field_lower_bound_tight} shows that the bound \eqref{eq:r_max_bound} is tight in the $M=1$ unweighted (not necessarily undirected) graph case with average in-degree $\bar{d} \geq 1$.

\subsection{Mean field approximation - $\ell^1$ error}
\label{s:average_error}

In Section \ref{s:l1_error}  we focused on error bounds at the individual level, i.e., $\ell^{\infty}$ bounds. However, as we have already seen in \eqref{eq:mean_field_lower_bound_general}, it is a necessary condition for $\tilde{r}_{max}$ to vanish, otherwise the $\ell^{\infty}$ bound becomes separated from $0$. For this reason we study the average error (also known as $\ell^{1}$ error), which we can control better. However, we restrict our focus to symmetric pair interactions.

\begin{theorem}[NIMFA $\ell^{1}$ upper bounds using Frobenius norm of $R$]
\label{t:l1}
Assume $M=1$ and $R^\intercal =R$. Then the following bounds hold:
\begin{align}
\label{eq:l1_delta_max}
\frac{1}{N}\sum_{i \in V}  \max_{s \in \mathcal{S}} \left |y_{i,s}(t)-z_{i,s}(t) \right| \leq& 4 t^2e^{3 \delta_{max}t} \frac{1}{N} \sum_{m \in V} \left(R^2 \right)_{mm}, \\
\label{eq:l1_sigma}
\frac{1}{N}\sum_{i \in V}  \max_{s \in \mathcal{S}} \left |y_{i,s}(t)-z_{i,s}(t) \right| \leq& 4  t^2e^{3 \|R \|_2 t} \sqrt{\frac{1}{N} \sum_{m \in V} \left(R^2 \right)_{mm}^2}. 
\end{align}
\end{theorem}
We will prove Theorem \ref{t:l1} in Section \ref{subsection_proofs_average_error}.

\begin{remark}
Note that due to symmetry we have
\begin{align}
\label{eq:Frobenius}
\frac{1}{N} \sum_{m \in V} \left(R^2 \right)_{mm}=\frac{1}{N}\sum_{l,m \in V} \bar{r}_{ml}\bar{r}_{lm} =\frac{1}{N}\sum_{m,l \in V}\bar{r}_{ml }^2=\frac{1}{N}\|R \|_F^2,
\end{align}
where $\| \cdot \|_F$ stands for the Frobenius norm.

In the previous works \cite{Sridhar_Kar,NIMFA_Illes} only the weaker $\ell^{1}$ upper bound $O \left( \sqrt{\frac{1}{N}\|R \|_F^2} \right)$ was proved under the assumption $\delta_{max}=O(1)$, albeit without the need of the symmetry assumption.
Our bound \eqref{eq:l1_delta_max} improves this to $O \left( \frac{1}{N}\|R \|_F^2 \right)$  where the implied constant in $O(\cdot)$ depends on  $\delta_{max}$ and $t$.
\end{remark}

The bound \eqref{eq:l1_delta_max} is only useful if $\delta_{max}=O(1)$ while the bound \eqref{eq:l1_sigma} is only useful if $\|R \|_2=O(1)$. If $R=R^{\intercal}$ then we have
\begin{equation}\label{r:R_graph} \| R\|_2 \leq \|R \|_{\infty}=\delta_{max},
\end{equation}
and in general $\|R \|_2=O(1)$ is a much weaker assumption than $\delta_{max}=O(1)$ as one can see from the example in Proposition \ref{p:chung-lu} bellow.

\begin{proposition}[Chung-Lu: an inhom.\ random graph with $\|R \|_2=O(1)$]
\label{p:chung-lu}
Take the rates $\bar{r}_{ji}=\frac{1}{N^{\alpha}}a_{ji}$ where $\left(a_{ji} \right)_{j,i \in V}=A=A(N)$ is the adjacency matrix of a Chung-Lu graph  with connection probabilities
\begin{equation}\label{chung_lu_a_ij_def}
\pr \left(a_{ij}=1 \right)=N^{\alpha} \frac{\left( \frac{N}{i} \right)^{\gamma} \left( \frac{N}{j} \right)^{\gamma}}{\sum_{k=1}^N \left( \frac{N}{k} \right)^{\gamma}}, \qquad 1 \leq i < j \leq N,
\end{equation}
where the exponents $\alpha$ and $\gamma$ satisfy $0<\gamma<\frac{1}{3}, \ \gamma<\alpha<1-2 \gamma$.

Then the sequence $R=R(N)$ satisfies  $\delta_{max}=\Theta \left( N^{\gamma} \right)$ and $\|R\|_2=O(1)$ with high probability.
\end{proposition}
We will prove Proposition \ref{p:chung-lu} in Section \ref{subsection_proofs_average_error}.

Since $\delta_{max}=\Theta \left( N^{\gamma} \right)$,  the bound \eqref{eq:l1_delta_max} is useless for a Markov process with rate matrix $R(N)$ introduced in Proposition \ref{p:chung-lu} if $N$ is large. On the other hand, \eqref{eq:l1_sigma} is very much useful: in fact it gives an error bound which has the optimal order of magnitude in this case, as explained in Remark \ref{rem_chung_lu_optimal} below.

\begin{lemma}[Frobenius norm and average degree of simple graphs]
\label{l:former_remark}

Assume $\bar{r}_{ij}=\frac{a_{ij}}{\bar{d}}$ for a an unweigthed undirected graph. Then we have 
\begin{align}\label{frob_simple_graph}
\frac{1}{N}\sum_{m \in V}\left(R^2 \right)_{mm}=& \frac{1}{\bar{d}}, \\
\label{ugly_nice_bound}
\sqrt{\frac{1}{N}\sum_{m \in V} \left(R^2 \right)_{mm}^2} \leq &\|R \|_2 \frac{1}{\bar{d}}.
\end{align}
\end{lemma}
We will prove Lemma \ref{l:former_remark} in in Section \ref{subsection_proofs_average_error}.

Lemma \ref{l:former_remark} and \eqref{r:R_graph}
imply that \eqref{eq:l1_delta_max} and \eqref{eq:l1_sigma} give error bounds of the same order of magnitude for unweighted undirected graphs satisfying $\delta_{max}=O(1).$

Now we state our lower bounds on the $\ell^1$ error of NIMFA that complement the upper bounds obtained in Theorem
\ref{t:l1}.

\begin{theorem}[NIMFA $\ell^{1}$ lower bound using (almost) the Frobenius norm]
\label{t:l1_lower_bound}
Let us define an SI process on the (not necessarily symmetric) weighted graph $R$ where initially every vertex is infected with probability $\frac{1}{2}$. Then
\begin{equation}\label{nimfa_l1_lower_exp_delta}
\frac{1}{N}\sum_{i \in V} \left |y_{i,I}(1)-z_{i,I}(1) \right| \geq \frac{1}{16} \frac{1}{N}\sum_{i,j \in V}e^{-(\delta_i+\delta_j)}\bar{r}_{ji}^2.
\end{equation}
\end{theorem}
We will prove Theorem \ref{t:l1_lower_bound} in Section \ref{subsection_proofs_average_error}.

\begin{corollary}
\label{c:delta_max}
If $\delta_{max}=O(1)$ and $R^{\intercal}=R$ then 
$\frac{1}{N}\sum_{i,j \in V}e^{-(\delta_i+\delta_j)}\bar{r}_{ij}^2=\Omega \left( \frac{1}{N} \|R\|_F^2 \right)$,
which (together with \eqref{eq:l1_delta_max}) gives that the tight bound for the average error is $\Theta \left(\frac{1}{N} \|R\|_F^2 \right).$
\end{corollary}

A consequence of \eqref{frob_simple_graph} and Corollary \ref{c:delta_max} is that for undirected, unweighted graphs the $\ell^1$ error is $\Theta \left( \frac{1}{\bar{d}} \right)$ once we assume $\delta_{max}=O(1)$. However, under said assumption even the $\ell^\infty$ error is $\Theta\left(\frac{1}{\bar{d}}\right)$ according to Theorems \ref{t:mean_field_uniform} and \ref{t:mean_field_lower_bound}, making this statement trivial. As we will see in the next theorem, the condition $\delta_{max}=O(1)$ can be relaxed to $\|R \|_2=O(1)$ in the case of the $\ell^1$ error.
\begin{theorem}[NIMFA $\ell^{1}$ lower bound on simple graphs]
\label{t:l1_graph}
For an  unweighted undirected graph with $\bar{r}_{ij}=\frac{a_{ij}}{\bar{d}}$, if we take an SI process with i.i.d.\ $\operatorname{Ber}\left(\frac{1}{2} \right)$ initial infections then  the following lower bound holds:
\begin{equation}\label{lower_recipr_bar_d}
\frac{1}{N}\sum_{i \in V} \left |y_{i,I}(1)-z_{i,I}(1) \right| \geq \frac{1}{32}e^{ -8 \| R\|_2^2} \frac{1}{\bar{d}}.
\end{equation}
\end{theorem}
We will prove Theorem \ref{t:l1_graph} in Section \ref{subsection_proofs_average_error}.

Putting together \eqref{eq:l1_sigma}, \eqref{ugly_nice_bound} and  \eqref{lower_recipr_bar_d}, we obtain that for unweighted undirected graphs  it is enough to assume $\|R \|_2=O(1)$  to obtain that the average error  of NIMFA is  $\Theta\left(\frac{1}{\bar{d}}\right)$.

\begin{remark}[$\ell^1$ error bounds \eqref{eq:l1_sigma} are sharp for Chung-Lu example]\label{rem_chung_lu_optimal}
One can check that the average degree of the random graph with adjacency matrix $A(N)$ from Proposition \ref{p:chung-lu} is well concentrated around $\frac{N^{\alpha}}{1-\gamma}$, thus the normalization $R(N)=\frac{1}{N^\alpha}A(N)$
used in Proposition \ref{p:chung-lu} only differs from the normalization $R=\frac{1}{\bar{d}} A$ used in Lemma \ref{l:former_remark} by a random, but well-concentrated  multiplier. Thus 
\eqref{eq:l1_sigma} and \eqref{ugly_nice_bound} together give an upper bound of order $O(\frac{1}{\bar{d}})$ on the $\ell^1$ error of NIMFA in the case of the Chung-Lu graphs of Proposition \ref{p:chung-lu}. The order of magnitude of this bound is optimal 
by Theorem \ref{t:l1_graph}.
 \end{remark}

Lastly, we return to the case of general weighted rate matrices $R$ (as defined in \eqref{def_R_matrix}) and 
characterize when the  average error vanishes using $\frac{1}{N} \|R \|_F^2$, assuming only $R^{\intercal}=R$ and $\| R\|_2=O(1)$. The first step is to show that the key term on the r.h.s.\ of  \eqref{eq:l1_sigma}
is small if and only if $\frac{1}{N} \|R \|_F^2$ is small.

\begin{claim}[Frobenius bounds for weighted graphs]
\label{r:theta}
Let $e_i$ denote the $i$th unit vector. Then for all $i$ we have
\begin{equation}\label{R_squared_and_norm}
\left(R^2 \right)_{ii}=e_i^{\intercal}R^2 e_i \leq \|R^2 \|_2 \|e_i \|_2^2 \leq \| R\|_2^2,
\end{equation}
therefore,
\begin{align*}
\frac{1}{N}\sum_{m \in V} \left(R^2 \right)_{mm} \stackrel{\text{CSB}}{\leq} \sqrt{\frac{1}{N}\sum_{m \in V} \left(R^2 \right)_{mm}^2} \leq \|R \|_2 \sqrt{\frac{1}{N}\sum_{m \in V} \left(R^2 \right)_{mm}},
\end{align*}
hence, $\frac{1}{N}\sum_{m \in V} \left(R^2 \right)_{mm}$ and $\sqrt{\frac{1}{N}\sum_{m \in V} \left(R^2 \right)_{mm}^2}$ are equivalent under the assumption $\|R \|_2=O(1)$ in the sense that they go to $0$ for the same sequences of weighted graphs $R=R(n)$, $n=1,2,\dots$
\end{claim}
The statements of Claim \ref{r:theta} require no further validation.

\begin{theorem}[NIMFA $\ell^{1}$ lower bound on weighted graphs in terms of Frobenius norm]
\label{t:l1_charachterisation}
Define $0<\theta:=\frac{1}{N} \|R \|_F^2$. Take an SI process with symmetric rates $R=R^{\intercal}$ where initially every vertex is infected with probability $\frac{1}{2}.$ Then we have 
\begin{align*}
\frac{1}{N}\sum_{i \in V} \left|y_{i,I}(1)-z_{i,I}(1) \right| \geq \frac{1}{32}e^{-\frac{8 \|R \|_2^3}{\theta}}\theta.
\end{align*}
\end{theorem}
We prove Theorem \ref{t:l1_charachterisation} in Section Section \ref{subsection_proofs_average_error}.

\begin{corollary}[Characterization of vanishing NIMFA $\ell^{1}$ error for weighted graphs using Frobenius]
Together, Theorems \ref{t:l1}, \ref{t:l1_charachterisation}   and Claim \ref{r:theta} imply that for a sequence of symmetric weighted graphs $R=R(n)$ satisfying $\|R(n) \|_2=O(1)$, the average error (i.e., the left-hand side of \eqref{eq:l1_delta_max}) vanishes for all processes with rate matrix $R(n)$ (cf.\ \eqref{def_R_matrix})  for any fixed $t \geq 0$ if and only if  $\theta(n) \to 0$.
\end{corollary}

\section{Applications}
\label{s:applications}

Now we sketch possible applications of the results stated in Section \ref{s:results} with the aim of demonstrating the flexibility of our setup described in Section \ref{s:setup}.

In Section \ref{subsection_multilayered} we discuss the SAIS processes with the aim of illustrating the difference between interacting particle systems (IPSs) on multilayered networks and IPSs on weighted graphs that we earlier described in Section \ref{s:hypergraphs}. 

In Section \ref{s:graph_dynamics} we demonstrate that our setup can handle some models for which the edge set of a graph changes according to Markovian dynamics
(e.g.\ a certain modification of the triangle removal process, a central example in \cite{flip_process}).

In Section \ref{s:simultaneous} we go one step further and discuss a model where \emph{dynamics of networks} is combined with \emph{dynamics on networks} (but we can only treat cases when the latter does not affect the former).

\subsection{Multilayered networks}\label{subsection_multilayered}

In Section \ref{s:hypergraphs} we discussed how the transition rates  factorizes to \eqref{eq:r_factorized} when studying Markov processes on a single hypergraph. However, some models require more than one network to operate on, such as in the case of epidemic processes on multilayered graphs where infection spreads on physical contact networks while information diffuses on a virtual social network.

An example of this is the so called SAIS model \cite{SAIS} where, besides susceptible (S) and infected (I), vertices can be in the alert (A) state as well. There are two (potentially directed) weighted networks $\left(w_{1;ij} \right)_{i,j \in V},\left(w_{2;ij} \right)_{i,j \in V}.$

An infected vertex $i$ recovers at rate $r_{i; I \to S}=\gamma_i$. An infected vertex  can infect another susceptible
vertex $j$ at rate $r_{ij;(I,S) \to (I,I)}=\beta_{S} w_{1;ij}$. An infected vertex can infect an
alert vertex $j$ at rate $r_{ij;(I,A) \to (I,I)}=\beta_{A} w_{1;ij} $. We usually set $\beta_{A}<\beta_{S}$ to ensure that alert vertices are less likely to get infected. Lastly, on the second network, an infected vertex $i$ can make a susceptible vertex $j$ alert at rate $r_{ij;(I,S) \to (I,A)}= \kappa w_{2;ij}$.

 It is straightforward to generalize these sort of models to the case when we have a hypergraph with $L$ layers. In this case the total potential interaction rate could be upper bounded by
 \begin{align*}
 \bar{r}_{\underline{j}i}^{(m)} \leq C \sum_{l=1}^{L}w_{l,\underline{j}i}^{(m)},
 \end{align*}
 therefore all the results stated in Section \ref{s:results} hold.

\subsection{Dynamics of networks}
\label{s:graph_dynamics}

So far we only discussed the dynamics of the states of vertices. Now we wish to extend the approach to include edge dynamics as well. This can be achieved by thinking of the edges of an evolving graph as vertices of a larger hypergraph. To avoid confusion, in this and the next section (Section \ref{s:simultaneous}) we will call the elements of $V$ as \emph{agents}, where agents can be either vertices or edges of an evolving graph.

In this section we wish to describe a model of graph processes resembling the one in proposed in \cite{flip_process}. The discrete-time graph processes in \cite{flip_process} are described as follows: at each time step we sample  $n$ i.i.d.\ uniformly distributed vertices which induce a subgraph $H$ of the current graph. This subgraph is then changed to $H'$ with probability $R_{H,H'}$ given by a stochastic matrix $R$ (not related to the matrix $\left(\bar{r}_{ij}\right)_{i,j \in V}$ in our notation). The authors of \cite{flip_process} show that if the sequence $G_N(0)$ of initial graphs converges to a graphon as $N \to \infty$ then for any $t \in \mathbb{R}_+$ the graphs $G_N( \lfloor N^2 t \rfloor )$ also converge to a graphon and the dynamics on the space of graphons is described by a differential equation. Note that, by the definition of convergence of dense graph sequences, these results can be recast as laws of large numbers for hompomorphism densities.

The main differences between the setting of \cite{flip_process} and ours  are the following:
(i) in our work, the graph evolves in continuous time (although, this difference does not matter when it comes to the law of large numbers); (ii) our main results are formulated in a way that we fix the number of vertices (instead of studying  graph sequences with $N \to \infty$); (iii) we only allow \emph{one} edge to change her state at a time, while in \cite{flip_process} more edges can change their states in one step; (iv)
the error bounds that we prove in Section \ref{s:l1_error}  pertain to the NIMFA approximation of the probability that an individual agent (in this case, an edge) is in a certain state (in this case, present or absent), and the law of large numbers for the homomorphism densities can be derived from these results.

Informally, we look at random subgraphs with a fixed number of $n \geq 2$ vertices and change one of the edges based on the state of the other edges. This means there are $M+1:= \binom{n}{ 2}$ edges that take part in an interaction.

The set of agents  corresponds to the set of edges $V= \binom{[N]}{2}$. The state space is $\mathcal{S}= \{0,1 \}$ where the state $1$ ($0$) is interpreted as an open (closed) edge. The set $V$ is the vertex set of an $(M+1)-$uniform hypergraph where the agents $f_1,\dots, f_{M},e \in \binom{[N]}{2}$ are in a hyperedge if and only if they form a complete subgraph $K_n$ as edges. The transition rates are

\begin{align*}
r_{\underline{f}e;(\underline{\tilde{s}},s) \to (\underline{\tilde{s}},s') }^{(M)}=\frac{1}{ \binom{N-2}{n-2} }q_{(\underline{\tilde{s}},s) \to (\underline{\tilde{s}},s')}^{(M)}  \1{\textit{$f_1,\dots, f_M, e$ forms a $K_n$}}.
\end{align*}

Recalling \eqref{eq:bar_q} it is easy to see that
\begin{align*}
\bar{r}^{(M)}_{\underline{f}e}=&\frac{1}{ \binom{N-2}{n-2}}\bar{q}^{(M)} \1{\textit{$f_1,\dots, f_M, e$ forms a $K_n$}} \\
\delta_{e}^{(M)}=&\bar{q}^{(M)}
\end{align*}
where we used the fact that an agent is a member of exactly $\binom{N-2}{ n-2}$ hyper edges. When it comes to influences, there are two cases. When the edges $f\neq e$ share a common vertex then there are $3$ fixed vertices and the other $n-3$ are free to be chosen to form a $K_n$. If the edges $f$ $g$ do not share a common vertex then $4$ vertices are fixed. This means
\begin{align*}
\tilde{r}_{fe}=&\begin{cases}
	\frac{ \binom{N-3}{n-3} }{ \binom{N-2}{n-2} }\bar{q}^{(M)} \text{ if $f,e$ share a common vertex} \\
	\frac{ \binom{N-4}{n-4} }{ \binom{N-2}{n-2} }\bar{q}^{(M)} \text{ otherwise},
\end{cases} \\
\tilde{r}_{max}=& \, O \left( \frac{1}{N} \right).
\end{align*}

The main example in \cite{flip_process} is the triangle removal process where we independently and uniformly pick three vertices of the graph and delete their edges in case they form a triangle. Our setting does not allow this process as it requires changing the state of three edges at a time. Hence, we propose an modified model where we only remove one edge of the triangle ($q_{(1,1,1) \to (1,1,0)}=1$). For this model NIMFA takes the form
\begin{align}
\label{eq:NIMFA_triangle}
\frac{\d}{\d t}z_{ij}(t)=-z_{ij}(t)\frac{1}{N-2}\sum_{k \in V \setminus \{i,j\}}z_{ik}(t)z_{kj}(t)
\end{align} 
where we made use of the convention $z_{e,1}(t)=z_{ij}(t)=z_{ji}(t)$ for $e=ij$.

Let the evolving graph be denoted by $\mathcal{G}(t)$ and let $H$ be a graph with $k$ vertices. One can show that the homomorphism density of $H$ in $\mathcal{G}(t)$ satisfies 
\begin{align*}
\E \left(t(H,\mathcal{G}(t)) \right) 
\overset{\textit{Lemma.\ \ref{t:chaos}, Thm.\ \ref{t:concentration_delta_max}}}{=}& \frac{1}{N^k}\sum_{e_1,\dots,e_k \in E(H)} \prod_{l=1}^{k}y_{e_k,1}(t)+O \left(\frac{k^2}{N} \right) \\
\overset{\textit{Theorem \ref{t:mean_field_uniform} }}{=}&\frac{1}{N^k}\sum_{e_1,\dots,e_k \in E(H)} \prod_{l=1}^{k}z_{e_k,1}(t)+O \left(\frac{k^2}{N} \right).
\end{align*} 

With the usual trick of taking two copies of $H$ it is possible to show that 
$$\operatorname{Var}\left(t \left(H, \mathcal{G}(t) \right) \right)=O \left( \frac{k^2}{N}\right).$$

Deriving that $\mathcal{G}(t)$ converges to a graphon (assuming that $\mathcal{G}(0)$ converges to a graphon) amounts to showing that - roughly speaking - $z_{ij}(t) \approx W_{t} \left( \frac{i}{N}, \frac{j}{N} \right)$
where $W_t(x,y)$ satisfies the integro-differential equation
\begin{align}
\label{eq:ODE_W_t}
\partial_t W_t(x,y)=-W_t(x,y) \int_{0}^{1}W_t(x,z)W_{t}(z,y) \d z,
\end{align} 
which is a continuous analogue of \eqref{eq:NIMFA_triangle}.

Proving this is outside scope of this paper and $\mathcal{G}^N(t) \to W_t$ in the graphon sense has already been shown in \cite{flip_process} (in the case of the discrete-time dynamics). What we would like to emphasise is that our setup is general enough to cover dynamics of graphs for which propagation of chaos and mean field approximation results holds. Given these results, the only thing left to study is the integro-differential equation that arises as the limit object of NIMFA.

\subsection{Simultaneous vertex and edge dynamics}
\label{s:simultaneous}

Previously we discussed how to treat vertex and edge dynamics separately. Here we show how can they be combined in certain cases in a way that fits our general setup. For further works on processes on dynamic graphs, see \cite{Piankoranee_2018, jacob2022, schapira2023}.

The two main restrictions are:
	(I) as before, we only allow one agent to change her state at a time;
	(II) the edge dynamics can affect the state of the vertices but not the other way around.
The reason behind condition (II) will be explained later.

Instead of giving a general framework, we give an example as a proof of concept which is easy to generalize. More specifically, we study an SI process on a graph governed by the (modified) triangle removal process.

The set of agents are made of two disjoint sets $V=[N] \sqcup
\binom{ [N]}{2}$ representing vertex and edge type agents. The state space is $\mathcal{S}=\{S,I,1,0\}$ where vertex agents can only be at state $S$ or $I$ and edge agents can only be in state $0$ or $1$.

The transition rates are given by:
\begin{align*}
\bar{r}_{j(ji)i}=&r_{j(ji)i;(I,1;S) \to (I,1,I)}\\
=&\frac{1}{N-1} \ \ (\textit{vertex $j$ infects vertex $i$ through edge $ji$}), \\
\bar{r}_{(ik)(jk)(ij)}=&r_{(ik)(jk)(ij);(1,1,1) \to (1,1,0)}\\
=& \frac{1}{N-2} (\textit{edge $ik$ and $jk$ deletes edge $ij$}).
\end{align*}

Using the simplified notation
\begin{align*}
	w_{ij}(t):=z_{ij,1}(t), \qquad
	u_{i}(t):=z_{i,I}(t),
\end{align*}
NIMFA takes the form
\begin{align*}
\frac{\d}{\d t}u_i(t)=&(1-u_i(t)) \frac{1}{N-1} \sum_{k \in V}w_{ik}(t)u_k(t), \\
\frac{\d}{\d t}w_{ij}(t)=&-w_{ij}(t) \frac{1}{N-2}\sum_{k \in V \setminus \{i,j\}}w_{ik}(t)w_{kj}(t).
\end{align*}

Now let us check the conditions for concentration and mean field accuracy. One can check that
\begin{align*}
\delta_{i}^{(2)}=1 \; \text{ and } \;
\delta_{(ij)}^{(2)}=1, \; \text{ therefore } \; \delta_{max}=1,
\end{align*}
and the non-zero terms of the influences are
\begin{align*}
\tilde{r}_{ji}=\tilde{r}_{(ji)i}= \frac{1}{N-1} \; \text{ and } \;
\tilde{r}_{(ik)(ij)}= \frac{1}{N-2}, \; \text{ therefore } \; \tilde{r}_{max}=\frac{1}{N-2}.
\end{align*}

This means that the $\ell^\infty$ error for NIMFA is $O \left( \frac{1}{N} \right)$ and the law of large numbers results hold.

 In general under assumption (II) vertex agents do not influence edge agents, hence, $\tilde{r}_{fe}$ remains the same as in Section \ref{s:graph_dynamics}. The influence on a vertex is virtually the same as in the classical case of vertex dynamics described in  Section \ref{s:hypergraphs}  with the only caveat that we have to ``activate'' any potential edge to calculate the influence. In general, a vertex $j$ will have the same influence on vertex $i$ as the edge $ji$.
 
 However, $\bar{r}_{max}$ may not vanish if we do not assume (II). E.g., with an SI dynamic on a graph where susceptible vertices can delete their edge to an infected neighbor at rate $\alpha$ (to prevent infection). This, however would mean that the edge $(ij)$ is strongly influenced by both vertex $i$ and $j$ as $\tilde{r}_{i(ij)}=\tilde{r}_{j(ij)}=\alpha,$ hence the usual heuristic that the direct effect between two individual agents is low fails.

\section{Graphical construction}
\label{section_coupling}

In Section \ref{subsection_ppp_notation} we set up the notation that we use to construct our interacting particle systems in terms of labeled Poisson point processes.
In Section \ref{subsection_forward_ips} we put things together and perform the construction in chronological order.
In Section \ref{subsection_backward_particle_sys} we perform the graphical construction in reverse chronological order: this allows us to trace the set of vertices whose initial state can potentially affect the state of a fixed vertex at time $t$, giving rise to the construction of information sets of Definition \ref{def_info_set}.
In Section \ref{s:joint_construction} we augment our backward construction in a way so that we can couple it to an auxiliary mean field backward branching structure that serves as a stochastic representation of the NIMFA equation.

\medskip 

Now let us give a brief survey of the literature of the methods that we use in Section \ref{section_coupling}.
The idea of looking at the graphical construction backward in time has been widely used in various concrete cases. In general, two processes 
are called pathwise dual if they can be coupled using the same source of randomness in a way that one process is running forward, the other is running backward in time, see \cite{jk14} for a survey on the notion of duality.
In particular, methods similar to ours using the backward graphical construction and the notion of so-called ghost vertices (cf.\ e.g.\ \eqref{def_ghost_event}) have been used in \cite{Barbour, Juli2014, Juli2015}.
The idea of the stochastic representation of the solution of mean field differential equations like NIMFA in terms of a recursive tree structure already appeared in \cite[Section 3.3]{msss20}, moreover the idea of using a coupling of the interacting particle system with an auxiliary recursive tree structure to control the error of mean field approximations already appeared in
\cite[Section 4]{msss20}, although only in the case when the transition rates are invariant under all permutations of the vertex set.

\subsection{Labeled Poisson point processes (PPPs)}\label{subsection_ppp_notation}

For any $m$, $\underline{j} \in V^m $ such that $j_1<\dots<j_m$ (where $V$ is identified with $[N]$) and $i \in V$, let
$\underline{\mathcal{T}}_{\underline{j}i}^{(m)}$ denote a time-homogeneous Poisson point process (PPP) on $\mathbb{R}_+$ with
with intensity $\bar{r}_{\underline{j}i}^{(m)}$. Let us assume that these PPPs are independent. We have $\underline{\mathcal{T}}_{\underline{j}i}^{(m)}=(\mathcal{T}_{\underline{j}i}^{(m)}(k))_{k=1}^{\infty}$, where $\mathcal{T}_{\underline{j}i}^{(m)}(k)$ is the arrival time of the $k$'th point of the PPP. We visualize a point of $\underline{\mathcal{T}}_{\underline{j}i}^{(m)}$ as an \emph{instruction arrow} that has the $m$-tuple $\underline{j}=(j_1,\dots,j_m)$ of vertices at its \emph{base} and the vertex $i$ is its \emph{target}.

Let us
denote by $\underline{\mathcal{L}}_{\underline{j}i}^{(m)}=(\mathcal{L}_{\underline{j}i}^{(m)}(k))_{k=1}^{\infty}$ the family of corresponding labels, i.e.,
$\mathcal{L}_{\underline{j}i}^{(m)}(k)$ is the instruction of form $(\underline{\tilde{s}},s') \to (\underline{\tilde{s}},s)$ corresponding to the $k$'th point $\mathcal{T}_{\underline{j}i}^{(m)}(k)$ of $\underline{\mathcal{T}}_{\underline{j}i}^{(m)}$. We choose i.i.d.\ labels for the points of $\underline{\mathcal{T}}_{\underline{j}i}^{(m)}$ and we set the probability of label $(\underline{\tilde{s}},s') \to (\underline{\tilde{s}},s)$ (where $s' \neq s$) to be equal to  $r^{(m)}_{\underline{j}i;(\underline{\tilde{s}},s') \to (\underline{\tilde{s}},s)} / \bar{r}_{\underline{j}i}^{(m)}$, noting that by \eqref{eq:r_bar} these probabilities indeed sum up to one. Thus, by the coloring property of PPPs,  instruction arrows with label $(\underline{\tilde{s}},s') \to (\underline{\tilde{s}},s)$, base $\underline{j}$ and target $i$ arrive at rate $r^{(m)}_{\underline{j}i;(\underline{\tilde{s}},s') \to (\underline{\tilde{s}},s)}$, independently of each other, as they should.

For any $A \subseteq V$ let $\underline{\mathcal{T}}_A$ denote the family $\left( \underline{\mathcal{T}}_{\underline{j}i}^{(m)} \right)_{ m \in \{0,\dots,M \}, \underline{j} \in V^m, i \in A } $ of PPPs whose instructions can potentially have an effect on the state of a vertex in $A$. With a slight abuse of notation, we will sometimes think about $\underline{\mathcal{\mathcal{T}}}_A$ as a point process on $\mathbb{R}_+$ which consists of all of the points of the point processes $\underline{\mathcal{T}}_{\underline{j}i}^{(m)}, m \in \{0,\dots,M \}, \underline{j} \in V^m, i \in A$.

Let us define the function $f^{(m)}_{(\underline{\tilde{s}},s') \to (\underline{\tilde{s}},s)} : \mathcal{S}^m \times \mathcal{S} \to \mathcal{S}  $ by
\begin{equation}\label{function_state_change_def}
	f^{(m)}_{(\underline{\tilde{s}},s') \to (\underline{\tilde{s}},s)}(\underline{s}'',s''  ):= \begin{cases} s & \text{ if } (\underline{s}'',s'')=(\underline{\tilde{s}},s'), \\
		s'' & \text{otherwise.} \end{cases}
\end{equation}

If $\underline{\sigma}=( \sigma_i )_{i \in V} \in \mathcal{S}^V$ and $\underline{j}=(j_1,\dots,j_m) \in V^m $, let $\underline{\sigma}_{\underline{j}}$ denote the element of $\mathcal{S}^m$ defined by $\underline{\sigma}_{\underline{j}}=(\sigma_{j_1},\dots,\sigma_{j_m} )$.

Let $f^{(m)}_{\underline{j}i;(\underline{\tilde{s}},s') \to (\underline{\tilde{s}},s)}: \mathcal{S}^V \to \mathcal{S}^V  $ denote the  effect of the instruction
$(\underline{\tilde{s}},s') \to (\underline{\tilde{s}},s)$ assigned to the $m$-tuple $\underline{j}$ of vertices and vertex $i$:
\begin{equation}\label{function_configuration_change_def}
	f^{(m)}_{\underline{j}i;(\underline{\tilde{s}},s') \to (\underline{\tilde{s}},s)}(\underline{\sigma})=\underline{\sigma}', \quad \text{where} \quad
	\sigma'_{i'} =\begin{cases} f^{(m)}_{(\underline{\tilde{s}},s') \to (\underline{\tilde{s}},s)} \left( \underline{\sigma}_{\underline{j}}, \sigma_i \right) & \text{ if } i'=i, \\
		\sigma_{i'} & \text{otherwise}. \end{cases}
\end{equation}

\subsection{Forward graphical construction of interacting particle system}\label{subsection_forward_ips}

We will denote the state of vertex $i$ at time $t$ by $\sigma_i(t)$, i.e., $\sigma_i(t)=s$ if and only if $\xi_{i,s}(t)=1$ (cf.\ Section \ref{s:transition_rates}). Let $\underline{\sigma}(t)=\left( \sigma_i(t) \right)_{i \in V}$ denote the configuration of states at time $t$.
Recall that we assume that the random variables $\sigma_i(0), i \in V$ are independent.
Given the initial configuration $\underline{\sigma}(0)$, the family of PPPs $\underline{\mathcal{\mathcal{T}}}_V$ and the family of corresponding labels $\underline{\mathcal{\mathcal{L}}}_V$, we can define $\underline{\sigma}(t)$ for any $t \geq 0$ by induction, following the instructions in chronological order: let us assume that we have already defined
$\underline{\sigma}(t)$ for some $t$. Let $t^*$ denote the time of the next point of  $\underline{\mathcal{\mathcal{T}}}_V$ after $t$, and let us assume that $t^*=\mathcal{T}_{\underline{j}i}^{(m)}(k)$ for some $m, \underline{j}, i$ and $k$, moreover let us assume that the corresponding instruction $\mathcal{L}_{\underline{j}i}^{(m)}(k)$ is  $(\underline{\tilde{s}},s') \to (\underline{\tilde{s}},s)$. Then
we have $\underline{\sigma}(t')=\underline{\sigma}(t)$ for any $t' \in [t, t^*)$ and
\begin{equation}
	\underline{\sigma}(t^*)\stackrel{\eqref{function_configuration_change_def}}{=} f^{(m)}_{\underline{j}i;(\underline{\tilde{s}},s') \to (\underline{\tilde{s}},s)}(\underline{\sigma}(t)).
\end{equation}
It follows from \eqref{eq:y} that we have
\begin{equation}\label{y_again}
    \mathbb{P}(\sigma_i(t)=s)=y_{i,s}, \qquad i \in V, \; s \in \mathcal{S}.
\end{equation}

\subsection{Backward graphical construction of interacting particle system}\label{subsection_backward_particle_sys}

Let us fix some $i_0\in V$ and $t_0 \in \mathbb{R}_+$. We want to determine the value of $\sigma_{i_0}(t_0)$, but we want to reveal as little information about  $\underline{\mathcal{\mathcal{T}}}_V$, $\underline{\mathcal{\mathcal{L}}}_V$ and $\underline{\sigma}(0)$ as possible. We will construct a function
$\Psi_{i_0,[0,t_0]}: \mathcal{S}^V \to \mathcal{S} $ using $\underline{\mathcal{\mathcal{T}}}_V$, $\underline{\mathcal{\mathcal{L}}}_V$
such that for any initial condition $\underline{\sigma}(0)$ we have $\sigma_{i_0}(t_0)=\Psi_{i_0,[0,t_0]}(\underline{\sigma}(0))$.

We will in fact construct a function $\Psi_{i_0,[t,t_0]}: \mathcal{S}^V \to \mathcal{S} $ for each $t \in [0,t_0]$ such that
$\sigma_{i_0}(t_0)=\Psi_{i_0,[t,t_0]}(\underline{\sigma}(t))$ holds. Moreover, we will define a set-valued process $\mathcal{H}_{i_0}(\tau)$, $0 \leq \tau \leq t_0$
such that
\begin{equation}
	\{ i_0\} = \mathcal{H}_{i_0}(0) \subseteq \mathcal{H}_{i_0}(\tau_1) \subseteq \mathcal{H}_{i_0}(\tau_2) \subseteq V \qquad 0 \leq \tau_1 \leq \tau_2 \leq t_0
\end{equation}
and the output of the function $\Psi_{i_0,[t,t_0]}: \mathcal{S}^V \to \mathcal{S}  $ only depends on the variables $\sigma_i, i \in \mathcal{H}_{i_0}(t_0-t)  $.

We define $\mathcal{H}_{i_0}(0)=\{i_0\}$ and $\Psi_{i_0,[t_0,t_0]}(\underline{\sigma})=\sigma_{i_0}$.

Let us assume that we already know $\Psi_{i_0,[t_-,t_0]}$ and $\mathcal{H}_{i_0}(t_0-t)$ for some $0< t \leq t_0$. Let us decrease the value of time, i.e., let us consider the last point $t^*$
of $\underline{\mathcal{\mathcal{T}}}_{\mathcal{H}_{i_0}(t_0-t)}$ before $t$. Let  us assume that $t^*=\mathcal{T}_{\underline{j}i}^{(m)}(k)$ for some $m, \underline{j}=(j_1,\dots,j_m), i$ and $k$, moreover let us assume that the corresponding instruction $\mathcal{L}_{\underline{j}i}^{(m)}(k)$ is  $(\underline{\tilde{s}},s') \to (\underline{\tilde{s}},s)$.
Then
we have $\Psi_{i_0,[t',t_0]} =\Psi_{i_0,[t_-,t_0]} $ for any $t' \in [ t^*, t)$ and $\mathcal{H}_{i_0}(\tau)= \mathcal{H}_{i_0}(t_0-t)$ for any $\tau \in [t_0-t, t_0-t^*)$, moreover we define
\begin{align}
	\label{ips_bacward_recursive_def}
	\Psi_{i_0,[t^*_-,t_0]} & \stackrel{\eqref{function_configuration_change_def}}{:=} \Psi_{i_0,[t_-,t_0]} \circ f^{(m)}_{\underline{j}i;(\underline{\tilde{s}},s') \to (\underline{\tilde{s}},s)},\\
	\mathcal{H}_{i_0}(t_0-t^*) & :=
	\mathcal{H}_{i_0}(t_0-t) \cup \{ j_1,\dots, j_m\}.
\end{align}
This completes the recursive definition of $\Psi_{i_0,[t,t_0]}$ for all $0 \leq t \leq t_0$ and $\mathcal{H}_{i_0}(\tau)$ for all $0 \leq \tau \leq t_0$.
See Figure \ref{fig_bw_construct} for a visualization.

\begin{figure}[h]
	\includegraphics[width=0.9 \textwidth]{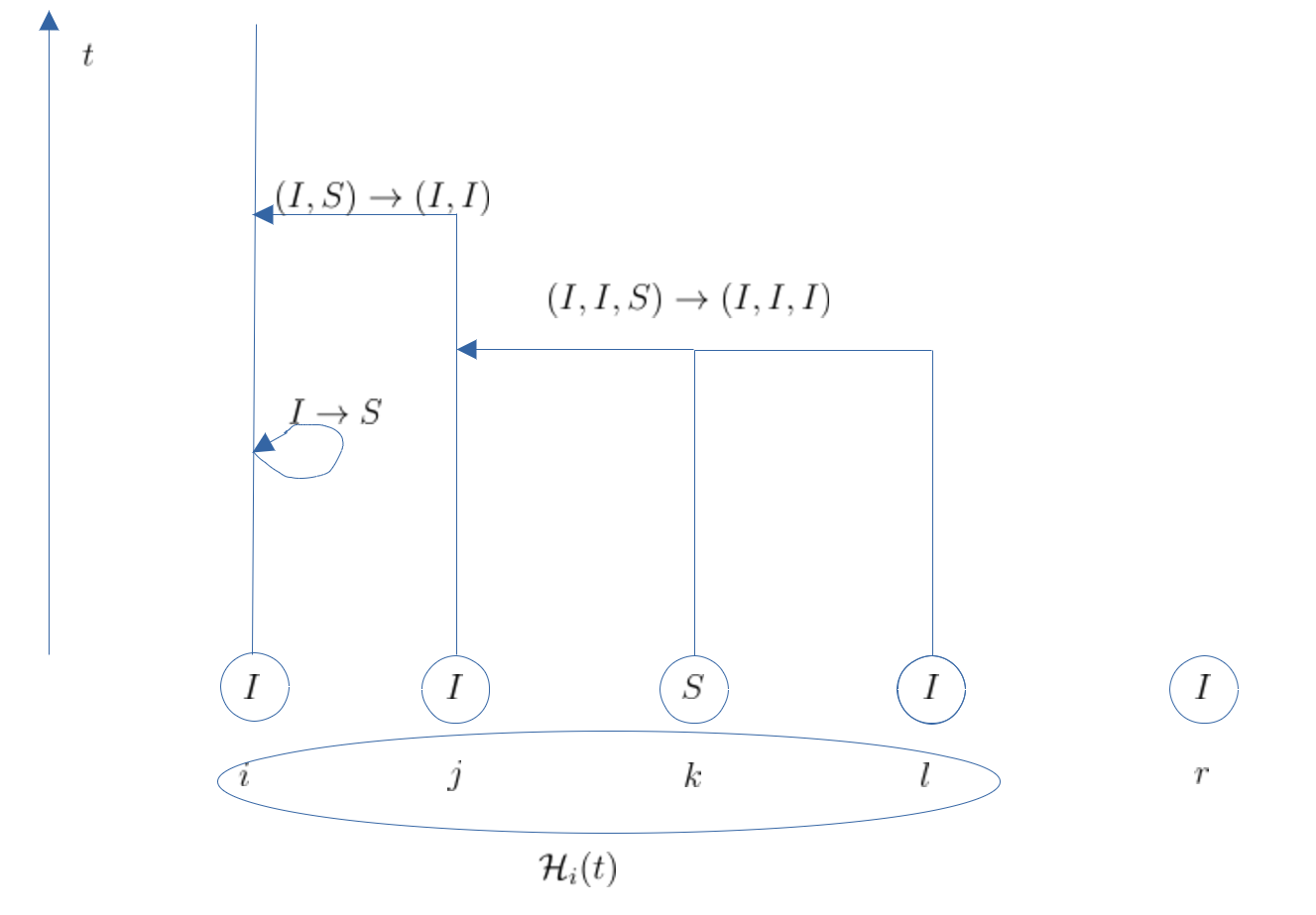} 
	\caption{\label{fig_bw_construct} Example for the backward construction of the simplicial SIS process. The arrow with label $(I,I,S) \to (I,I,I)$ does not result in a transition as vertex $j$ is not in state $S$ and vertex $k$ is not in state $I$, but said arrow makes vertices $k,l$ being included in $\mathcal{H}_i(t)$. The state of vertex $i$ at time $t$ is $I$.}
\end{figure}

\begin{claim}[Information set revisited]
	\label{c:H}	
	One can interpret $\mathcal{H}_{i_0}(\tau)$ as a variant of the $SI$ process. Initially only vertex $i_0$ is infected. If vertex $i$ is infected and $\underline{j}=(j_1,\dots, j_m)$ satisfies $j_1<j_2<\dots<j_m$ then $i$ infects the vertices of $\underline{j}$ (all of them at the same time) at rate $\bar{r}^{(m)}_{\underline{j}i}$. Observe that the rules of the evolution of $\mathcal{H}_{i_0}(\tau)$ agree with those of the information set introduced in
  Definition \ref{def_info_set}.
	
	When $M=1$, one can reduce this process to an SI process proper on the weighted graph $R^{\intercal}$, noting that the transpose of $R$ appears because of the reversion of the chronological order employed in Section \ref{subsection_backward_particle_sys}.
	
	Note that self-interactions do not play a role in the evolution of $\mathcal{H}_{i_0}(\tau).$
\end{claim}

\begin{claim}[Disjoint information sets and independence]
	\label{remark:arrow}
	For a fixed subset $H_i \subseteq V$  the event $\{ \mathcal{H}_i(t)=H_i \} $ is measurable with respect to the sigma-algebra generated by $ \mathcal{T}_{H_i}$, since $\mathcal{H}_{i}(\tau)$ evolves by adding the base of the arrow pointing to $\mathcal{H}_{i}(\tau) \subseteq H_i. $
	
	This implies that if we fix the disjoint sets $H_{i_1}, \dots, H_{i_m} \subseteq V$  then the events $\{\mathcal{H}_{i_1}(t)=H_{i_1}\}, \dots, \{\mathcal{H}_{i_m}(t)=H_{i_m}\}$
	 are independent.	
\end{claim}

\begin{claim}[Disjoint information sets and conditional independence]
	\label{remakark:xi}
	For a fixed subset $H_i \subseteq V$  the event $\{\xi_{i,s}(t)=1\} \cap \{ \mathcal{H}_{i}(t)=H_{i} \}$ is measurable with respect to the sigma-algebra generated by $\left( \left(\sigma_{k}(0) \right)_{k \in H_i}, \mathcal{T}_{H_i}, \mathcal{L}_{H_i} \right),$ therefore, for disjoint sets $H_{i_1}, \dots, H_{i_m} \subseteq V$  and arbitrary $s_{1}, \dots, s_{m}$ the random variables $\xi_{i_1,s_1}(t), \dots , \xi_{i_m,s_m}(t)$ are conditionally independent if we condition on the event $\{\mathcal{H}_{i_1}(t)=H_{i_1}, \dots , \mathcal{H}_{i_m}(t)=H_{i_m}\}.$
\end{claim}

We omit the proofs of Claims \ref{c:H}, \ref{remark:arrow} and \ref{remakark:xi} as they directly follow from our backward graphical construction and the properties of PPPs.

\subsection{Joint backward graphical construction of interacting particle system and NIMFA}
\label{s:joint_construction}

Let us fix some $i_0\in V$ and $t_0 \in \mathbb{R}_+$. We will extend the probability space of the backward construction described in Section
\ref{subsection_backward_particle_sys} and we will define an auxiliary random variable $\widetilde{\sigma}_{i_0}(t_0)$ which (i) provides a
stochastic representation of the NIMFA equations (cf.\ Lemma \ref{lemma_nimfa_stoch_representation}) and (ii) under some condition (cf.\ Lemma \ref{l:no_ghost}) we will have $\widetilde{\sigma}_{i_0}(t_0)=\sigma_{i_0}(t_0)$. As a result, we only need to bound the probability that this condition fails (i.e., a so-called ghost appears, cf.\ Corollary \ref{l:ghost}) in order to control the error of NIMFA.

\subsubsection{Augmented probability space}

For any $m$, any  $\underline{j}=(j_1,\dots,j_m) \in V^m $ satisfying $j_1<j_2<\dots<j_m$, $i \in V$ and $\ell \in \mathbb{N}=\{1,2,\dots\}$,  let
$\widetilde{\underline{\mathcal{T}}}_{\underline{j}i, \ell}^{(m)}$ denote a time-homogeneous PPP on $\mathbb{R}_+$ with
with intensity $\bar{r}_{\underline{j}i}^{(m)}$.
Let us assume that these PPPs are independent. We have $\widetilde{\underline{\mathcal{T}}}_{\underline{j}i, \ell}^{(m)}=(\widetilde{\mathcal{T}}_{\underline{j}i, \ell}^{(m)}(k))_{k=1}^{\infty}$. Let us
denote by $\widetilde{\underline{\mathcal{L}}}_{\underline{j}i, \ell}^{(m)}=(\widetilde{\mathcal{L}}_{\underline{j}i, \ell}^{(m)}(k))_{k=1}^{\infty}$ the family of corresponding labels, i.e.,
$\widetilde{\mathcal{L}}_{\underline{j}i, \ell}^{(m)}(k)$ is the instruction of form $(\underline{\tilde{s}},s') \to (\underline{\tilde{s}},s)$ corresponding to the $k$'th point $\widetilde{\mathcal{T}}_{\underline{j}i, \ell}^{(m)}(k)$ of $\widetilde{\underline{\mathcal{T}}}_{\underline{j}i, \ell}^{(m)}$.
We choose i.i.d.\ labels for the points of $\widetilde{\underline{\mathcal{T}}}_{\underline{j}i, \ell}^{(m)}$ and we set the probability of label $(\underline{\tilde{s}},s') \to (\underline{\tilde{s}},s)$ (where $s' \neq s$) to be equal to  $r^{(m)}_{\underline{j}i;(\underline{\tilde{s}},s') \to (\underline{\tilde{s}},s)} / \bar{r}_{\underline{j}i}^{(m)}$, just like in Section 
\ref{subsection_ppp_notation}.

Later (in Claim \ref{claim_brw_def})
we will  define a branching random walk on $V$ and the PPPs with $i$ and $\ell$ in the subscript will affect the $\ell$'th particle located at vertex $i$.
We call an element $(i, \ell) $ of $V \times \mathbb{N}$ an augmented vertex.

For any $\widetilde{A} \subseteq V \times \mathbb{N}$, let $\widetilde{\underline{\mathcal{T}}}_{\widetilde{A}}$ denote the family $( \widetilde{\underline{\mathcal{T}}}_{\underline{j}i, \ell}^{(m)} )_{ m \in \{0,\dots,M \}, \underline{j} \in V^m, (i, \ell) \in \widetilde{A} } $ of PPPs whose instructions can potentially have an effect on the state of an augmented vertex $(i, \ell)$ in $\widetilde{A}$.

We  generate independent initial states $\widetilde{\underline{\sigma}}(0)=(\widetilde{\sigma}_{i, \ell}(0))_{(i,\ell) \in V \times \mathbb{N} }$ where the distribution of $\widetilde{\sigma}_{i, \ell}(0)$ only depends on $i$.

We  identify $\widetilde{\underline{\mathcal{T}}}_{\underline{j}i, 1}^{(m)}$ with $\underline{\mathcal{T}}_{\underline{j}i}^{(m)}$  (defined in Section \ref{subsection_ppp_notation}). We  identify
$\widetilde{\underline{\mathcal{L}}}_{\underline{j}i, 1}^{(m)}$ with $\underline{\mathcal{L}}_{\underline{j}i}^{(m)}$, moreover we  identify $\underline{\sigma}(0)$ with $(\widetilde{\sigma}_{i, 1}(0))_{i \in V}$. Thus $\sigma_{i_0}(t_0)$ (as defined in Section \ref{subsection_backward_particle_sys}) is a random variable that in now defined on our augmented probability space. 

\subsubsection{Auxiliary random variables and branching random walk}

Given $i_0 \in V$ and $t_0 \in \mathbb{R}_+$ we will construct a random variable $\widetilde{\sigma}_{i_0}(t_0)$ that will have a special role in the proof of our NIMFA error bounds stated in Sections \ref{s:l1_error} and
\ref{s:average_error}. We have already sketched the key properties of $\widetilde{\sigma}_{i_0}(t_0)$ at the beginning of Section \ref{s:joint_construction}.
We will construct a function $\widetilde{\Psi}_{i_0,[t,t_0]}: \mathcal{S}^{V\times \mathbb{N}} \to \mathcal{S} $ for each $t \in [0,t_0]$ such that
\begin{equation}\label{sigma_i_0_t_0_def}
	\widetilde{\sigma}_{i_0}(t_0)=\widetilde{\Psi}_{i_0,[0,t_0]}(\underline{\widetilde{\sigma}}(0))
\end{equation} holds.


We will also define a branching random walk with vertex set $V$, and $N_{i_0}(\tau,i)$ will denote the number of particles at vertex $i$ at time $\tau$.

We define the augmented information set $\widetilde{\mathcal{H}}_{i_0,t_0}(\tau) \subseteq V \times \mathbb{N}$ as follows: $(i, \ell) \in \widetilde{\mathcal{H}}_{i_0,t_0}(\tau)$ if and only if
$\ell \leq N_{i_0}(\tau,i)$. The set $\widetilde{\mathcal{H}}_{i_0,t_0}(\tau)$ will consist  of those augmented vertices $(i,\ell)$ such that  the output of the function $\widetilde{\Psi}_{i_0,[t-\tau,t_0]}$ potentially depends on the variable $\widetilde{\sigma}_{i, \ell}$. Loosely speaking, as we go backwards in time, the index set $\widetilde{\mathcal{H}}_{i_0,t_0}(\tau)$ of augmented vertices that can have an effect on the state of vertex $i_0$ at time $t_0$  grows as a branching random walk.

We define $N_{i_0}(0,i)=\mathds{1}[i=i_0]$ (we start with one particle at vertex $i_0$ at time zero) and $\widetilde{\Psi}_{i_0,[t_0,t_0]}(\underline{\widetilde{\sigma}})=\widetilde{\sigma}_{i_0,1}$.

Let us assume that we already know $\widetilde{\Psi}_{i_0,[t_-,t_0]}$ and $N_{i_0}(t_0-t,i)$ for some $0< t \leq t_0$. Let us decrease the value of time, i.e., let us consider the last point $t^*$
of $\underline{\widetilde{\mathcal{\mathcal{T}}}}_{\widetilde{\mathcal{H}}_{i_0,t_0}(t_0-t)}$ before $t$. Let  us assume that $t^*=\widetilde{\mathcal{T}}_{\underline{j}i, \ell}^{(m)}(k)$ for some $m, \underline{j}=(j_1,\dots,j_m), i, \ell$ and $k$, moreover let us assume that the corresponding instruction $\widetilde{\mathcal{L}}_{\underline{j}i, \ell}^{(m)}(k)$ is  $(\underline{\tilde{s}},s') \to (\underline{\tilde{s}},s)$.

Then
we define $\widetilde{\Psi}_{i_0,[t',t_0]} =\widetilde{\Psi}_{i_0,[t_-,t_0]} $ for any $t' \in [ t^*, t)$ and we define $N_{i_0}(\tau,i)= N_{i_0}(t_0-t,i)$ for any $\tau \in [t_0-t, t_0-t^*)$ and $i \in V$, moreover we define
\begin{align}
	N_{i_0}(t_0-t^*, i') & :=
	N_{i_0}(t_0-t,i') + \mathds{1}[ \, i'\in \{  j_1,\dots, j_m \} \, ], \\
	\label{tree_function_recursive_def}
	\widetilde{\Psi}_{i_0,[t^*_-,t_0]} & := \widetilde{\Psi}_{i_0,[t_-,t_0]} \circ f, \text{ where } \\
\label{f_augmented} f(\widetilde{\underline{\sigma}})=\widetilde{\underline{\sigma}}', \quad \text{where} \quad
	\widetilde{\sigma}'_{i',\ell'} & :=\begin{cases} 
		f^{(m)}_{(\underline{\tilde{s}},s') \to (\underline{\tilde{s}},s)} \left( \underline{\widetilde{\sigma}}^*_{\underline{j}}, \widetilde{\sigma}_{i, \ell} \right) & \text{ if } (i', \ell')=(i, \ell) \\
		\widetilde{\sigma}_{i',\ell'} & \text{otherwise} \end{cases},   \\
  \label{f_augmented_argument}
	\text{ where }\; \; \underline{\widetilde{\sigma}}^*_{\underline{j}} & :=( \widetilde{\sigma}_{j_1, N_{i_0}(t_0-t^*, j_1)}, \dots, \widetilde{\sigma}_{j_m, N_{i_0}(t_0-t^*, j_m)}  ).
\end{align}

This completes the recursive definition of $\widetilde{\Psi}_{i_0,[t,t_0]}$ for all $0 \leq t \leq t_0$ as well as $N_{i_0}(\tau,i)$ for all $0 \leq \tau \leq t_0$.

\begin{claim}[Branching random walk]\label{claim_brw_def} $N_{i_0}(\tau,i)$ can  be interpreted as a branching random walk: $N_{i_0}(\tau,i)$ is the number of particles at vertex $i$ at time $\tau$. We start with one particle at vertex $i_0$. Each particle
at vertex $i$, independently of each other, produce children as follows:
for each $m$ and each $\underline{j}=(j_1,\dots,j_m) \in V^m $ that satisfies $j_1<j_2<\dots<j_m$, a particle at vertex $i$ produces $m$ children simultaneously, one child at each of the vertices $j_1,\dots,j_m$, at rate
$\bar{r}_{\underline{j}i}^{(m)}$.
\end{claim}
We omit the proof of Claim \ref{claim_brw_def} as it directly follows from the construction of our auxiliary backward branching structure and the properties of PPPs.

We note that the total number of particles $\sum_{i \in V} N_{i_0}(\tau,i) $ grows exponentially as $\tau$ grows, but it is almost surely finite for all $\tau \in \mathbb{R}_+$. Note that
\begin{equation}\label{mon_N}
\text{
 $N_{i_0}(\tau,i)$ is a non-decreasing function of $\tau$ for each $i \in V$.}
\end{equation}

\begin{lemma}
	\label{l:no_ghost}
	
	 If $N_{i_0}(t_0,i ) \leq 1$ for all $i \in V$ then $\widetilde{\sigma}_{i_0}(t_0)=\sigma_{i_0}(t_0)$.
\end{lemma}
\begin{proof} Recall that we defined $\widetilde{\sigma}_{i_0}(t_0)=\widetilde{\Psi}_{i_0,[0,t_0]}(\underline{\widetilde{\sigma}}(0))$ in \eqref{sigma_i_0_t_0_def}.
	Recall that $\sigma_{i_0}(t_0)=\Psi_{i_0,[0,t_0]}(\underline{\sigma}(0))$, and that we identified the PPPs and initial conditions with $\ell=1$  in their subscript with their counterparts that appeared in Section \ref{subsection_backward_particle_sys}. If we compare the recursive definition \eqref{ips_bacward_recursive_def} of $\Psi_{i_0,[0,t_0]}$   (see also \eqref{function_configuration_change_def})
	with the recursive definition \eqref{tree_function_recursive_def} of $\widetilde{\Psi}_{i_0,[0,t_0]}$, we see that these definitions coincide as long as we have
	$N_{i_0}(t_0,i ) \leq 1$ for all $i \in V$, i.e., if we only use the PPPs and initial conditions indexed by $\ell=1$.
\end{proof}

Recall that $z_{i,s}(t)$ denotes the solution of the system of differential equations \eqref{eq:NIMFA}.
Let us define
\begin{equation}\label{wt_z_i_s}
	\widetilde{z}_{i,s}(t):= \mathbb{P}(\widetilde{\sigma}_{i}(t)=s)
\end{equation}

\begin{lemma}\label{lemma_nimfa_stoch_representation} We have $z_{i,s}(t)= \widetilde{z}_{i,s}(t)$ for all $i \in V, s \in \mathcal{S}$ and $t \in \mathbb{R}_+$.
\end{lemma}
Before we prove Lemma \ref{lemma_nimfa_stoch_representation}, we need to introduce some further notation regarding the branching structure of $\widetilde{\Psi}_{i_0,[0,t_0]}$.

\subsubsection{Branching structure}

Observe that the output of the function $\widetilde{\Psi}_{i_0,[t,t_0]}: \mathcal{S}^{V \times \mathbb{N}} \to \mathcal{S}  $ only depends on the variables $\widetilde{\sigma}_{i,\ell}, \, (i,\ell) \in \widetilde{\mathcal{H}}_{i_0,t_0}(t_0-t) $.

\begin{definition}\label{sigma_alg_interval}
	For any $t'_0 \in [0,t_0]$,  let us denote by $\mathcal{F}_{[t'_0,t_0]}$ the sigma-algebra generated by the points of the point processes of form $\widetilde{\underline{\mathcal{T}}}_{\underline{j}i, \ell}^{(m)}$ that fall in the time interval $[t'_0,t_0]$, decorated with the corresponding labels from
	$\widetilde{\underline{\mathcal{L}}}_{\underline{j}i, \ell}^{(m)}$.
\end{definition}

Note that the random function $\widetilde{\Psi}_{i_0,[t'_0,t_0]}$ and the random set $\widetilde{\mathcal{H}}_{i_0,t_0}(t_0-t'_0)$  are both $\mathcal{F}_{[t'_0,t_0]}$-measurable.

Let $t'_0 \in [0,t_0] $ and let $(i_*,\ell_*) \in \widetilde{\mathcal{H}}_{i_0,t_0}(t_0-t'_0) $.
Loosely speaking, $(i_*,\ell_*)$ is an augmented vertex whose state at time $t'_0$ is (potentially) required if we want to evaluate the state of augmented vertex $(i_0,1)$ at time $t_0$.

For any $t \in [0,t'_0]$ we will define the random function
\[\widetilde{\Psi}^{i_0,t_0}_{(i_*,\ell_*),[t,t'_0]}: \mathcal{S}^{V\times \mathbb{N}} \to \mathcal{S} \] and for any $\tau \in [0,t'_0]$ we will define the random subset $\widetilde{\mathcal{H}}^{i_0,t_0}_{(i_*,\ell_*),t'_0}(\tau)$ of $V\times \mathbb{N}$. Loosely speaking, $\widetilde{\mathcal{H}}^{i_0,t_0}_{(i_*,\ell_*),t'_0}(\tau)$ contains those augmented vertices
whose state at time $t'_0-\tau$ is required to evaluate the state of augmented vertex $(i_*,\ell_*)$ at time $t'_0$. Loosely speaking, if we plug in the states of augmented vertices at time $t$ into $\widetilde{\Psi}^{i_0,t_0}_{(i_*,\ell_*),[t,t'_0]}$, we obtain the state of augmented vertex $(i_*,\ell_*)$ at time $t'_0$. Now let us see the details.

We define $\widetilde{\mathcal{H}}^{i_0,t_0}_{(i_*,\ell_*),t'_0}(0)=\{ (i_*,\ell_*) \}$ and $\widetilde{\Psi}^{i_0,t_0}_{(i_*,\ell_*),[t'_0,t'_0]}(\underline{\widetilde{\sigma}})=\widetilde{\sigma}_{i_*,\ell_*}$.

Let us assume that we already know $\widetilde{\Psi}^{i_0,t_0}_{(i_*,\ell_*),[t_-,t'_0]}$
and $\widetilde{\mathcal{H}}^{i_0,t_0}_{(i_*,\ell_*),t'_0}(t'_0-t)$  for some $0< t \leq t'_0$.

Let us decrease the value of time, i.e., let us consider the last point $t^*$
of $\underline{\widetilde{\mathcal{\mathcal{T}}}}_{ \widetilde{\mathcal{H}}^{i_0,t_0}_{(i_*,\ell_*),t'_0}(t'_0-t) }$ before $t$. Let  us assume that $t^*=\widetilde{\mathcal{T}}_{\underline{j}i, \ell}^{(m)}(k)$ for some $m, \underline{j}=(j_1,\dots,j_m), i, \ell$ and $k$, moreover let us assume that the corresponding instruction $\widetilde{\mathcal{L}}_{\underline{j}i, \ell}^{(m)}(k)$ is  $(\underline{\tilde{s}},s') \to (\underline{\tilde{s}},s)$.

Then
we define
$\widetilde{\Psi}^{i_0,t_0}_{(i_*,\ell_*),[t',t'_0]}=\widetilde{\Psi}^{i_0,t_0}_{(i_*,\ell_*),[t_-,t'_0]}  $
for any $t' \in [ t^*, t)$
and we define $\widetilde{\mathcal{H}}^{i_0,t_0}_{(i_*,\ell_*),t'_0}(\tau)=\widetilde{\mathcal{H}}^{i_0,t_0}_{(i_*,\ell_*),t'_0}(t'_0-t)$
for any $\tau \in [t'_0-t, t'_0-t^*)$, moreover we define
\begin{multline}
	\widetilde{\mathcal{H}}^{i_0,t_0}_{(i_*,\ell_*),t'_0}(t'_0-t^*) 
	:= \\ \widetilde{\mathcal{H}}^{i_0,t_0}_{(i_*,\ell_*),t'_0}(t'_0-t) \cup \{ (j_1, N_{i_0}(t_0-t^*, j_1)), \dots  , (j_m, N_{i_0}(t_0-t^*, j_m)) \}, \\
	\label{sub_tree_function_recursive_def}
	\widetilde{\Psi}^{i_0,t_0}_{(i_*,\ell_*),[t^*_-,t'_0]}
	:= \widetilde{\Psi}^{i_0,t_0}_{(i_*,\ell_*),[t_-,t'_0]}
	\circ f, 
\end{multline}
where $f$ is defined exactly as in \eqref{f_augmented} and \eqref{f_augmented_argument}.

This completes the recursive definition of the function $\widetilde{\Psi}^{i_0,t_0}_{(i_*,\ell_*),[t,t'_0]}$ for all $0 \leq t \leq t'_0$
and the set $\widetilde{\mathcal{H}}^{i_0,t_0}_{(i_*,\ell_*),t'_0}(\tau)$ for all $0 \leq \tau \leq t'_0$. Note that the function $\widetilde{\Psi}^{i_0,t_0}_{(i_*,\ell_*),[t,t'_0]}$ only depends on the variables indexed by augmented vertices that belong to the set $\widetilde{\mathcal{H}}^{i_0,t_0}_{(i_*,\ell_*),t'_0}(t'_0-t)$.

Note that for any $0 \leq t \leq t'_0 \leq t_0$ the set
\begin{multline}\label{disjoint_union}
	\widetilde{\mathcal{H}}_{i_0,t_0}(t_0-t) \text{ is the disjoint union of the sets } \\ \widetilde{\mathcal{H}}^{i_0,t_0}_{(i_*,\ell_*),t'_0}(t'_0-t) , \text{ where } (i_*,\ell_*) \in 
	\widetilde{\mathcal{H}}_{i_0,t_0}(t_0-t'_0).
\end{multline}

For any $t'_0 \in [0,t_0] $ and  $(i_*,\ell_*) \in \widetilde{\mathcal{H}}_{i_0,t_0}(t_0-t'_0) $, let us define
\begin{equation}\label{wt_sig_i0_t0}
	\widetilde{\sigma}^{i_0,t_0}_{(i_*,\ell_*)}(t'_0)=\widetilde{\Psi}^{i_0,t_0}_{(i_*,\ell_*),[0,t'_0]}(\underline{\widetilde{\sigma}}(0)).
\end{equation}
Now if we define $\underline{\widetilde{\sigma}}^{i_0,t_0}(t'_0)=\left( \widetilde{\sigma}^{i_0,t_0}_{(i_*,\ell_*)}(t'_0) \right)_{ (i_*,\ell_*) \in \mathcal{S}\times V } $ by \eqref{wt_sig_i0_t0} if we have  $(i_*,\ell_*) \in \widetilde{\mathcal{H}}_{i_0,t_0}(t_0-t'_0) $ and arbitrarily if $(i_*,\ell_*) \notin \widetilde{\mathcal{H}}_{i_0,t_0}(t_0-t'_0) $ then we have
\begin{equation}\label{semigroup_property}
	\widetilde{\sigma}_{i_0}(t_0)\stackrel{ \eqref{sigma_i_0_t_0_def} }{=}\widetilde{\Psi}_{i_0,[0,t_0]}(\underline{\widetilde{\sigma}}(0))=
	\widetilde{\Psi}_{i_0,[t'_0,t_0]}( \underline{\widetilde{\sigma}}^{i_0,t_0}(t'_0) ).
\end{equation}

Figure \ref{fig_brstr} serves as a visualization of the objects that we have just defined.

Recall the definition of $\widetilde{z}_{i,s}(t)$ from \eqref{wt_z_i_s} and the definition of $\mathcal{F}_{[t'_0,t_0]}$ from Definition \ref{sigma_alg_interval}.
\begin{claim}[Branching structure]\label{lemma_branching_structure}

$ $
\begin{enumerate}[(I)]
	\item\label{br_identical} For any $i_0, i_* \in V$, $s \in \mathcal{S}$, $\ell_* \in \mathbb{N}$ and $0 \leq t'_0 \leq t_0$  we have
	\begin{multline}
		\mathds{1}[ (i_*,\ell_*) \in \widetilde{\mathcal{H}}_{i_0,t_0}(t_0-t'_0)  ] \cdot
		\mathbb{P} \left(  \widetilde{\sigma}^{i_0,t_0}_{(i_*,\ell_*)}(t'_0)=s   \, | \, \mathcal{F}_{[t'_0,t_0]}\right)= \\
		\mathds{1}[ (i_*,\ell_*) \in \widetilde{\mathcal{H}}_{i_0,t_0}(t_0-t'_0)  ]\cdot \widetilde{z}_{i_*,s}(t'_0).
	\end{multline}
	\item\label{br_independent}  For any $i_0 \in V$, $0 \leq t'_0 \leq t_0$, 
	conditional on $\mathcal{F}_{[t'_0,t_0]}$, the random variables $\widetilde{\sigma}^{i_0,t_0}_{(i_*,\ell_*)}(t'_0),\, (i_*,\ell_*) \in \widetilde{\mathcal{H}}_{i_0,t_0}(t_0-t'_0)$ are conditionally independent.
\end{enumerate}
\end{claim} 
\begin{proof}
    The key feature of \eqref{f_augmented_argument} is the following: as we go backward in time and grow the information set by adding a new set $\{j_1,\dots,j_m\}$ of vertices, if we add a vertex that was already there then we actually add a new ``ghost'' copy of it to the information set.  The original copies of vertices are indexed by $\ell=1$, the ghosts are indexed by $\ell=2,3,\dots$. Since the PPPs and the initial states indexed by different values of $\ell$ are independent, we obtain that sources of randomness that determine
 $\widetilde{\sigma}^{i_0,t_0}_{(i_*,\ell_*)}(t'_0)$ are conditionally  independent given $\mathcal{F}_{[t'_0,t_0]}$ for different augmented vertices $(i_*,\ell_*)$ by \eqref{disjoint_union}, i.e., \ref{br_independent}.\ holds. Also note the rates of the PPPs and the distributions of the initial states do not depend on $\ell$, thus the conditional distribution of the sources of randomness that determine
 $\widetilde{\sigma}^{i_0,t_0}_{(i_*,\ell_*)}(t'_0)$ given $\mathcal{F}_{[t'_0,t_0]}$ is the same as those that determine $\widetilde{\sigma}_{i_*}(t'_0)$, i.e., 
   \ref{br_identical}.\ holds. 
\end{proof}

\begin{figure}[h]
	\includegraphics[width=0.9 \textwidth]{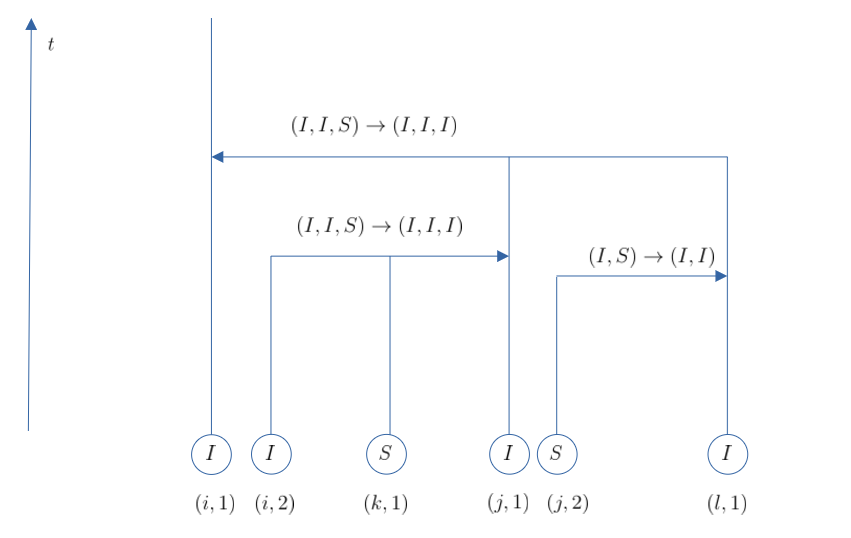} 
	\caption{\label{fig_brstr} Example for the backward construction of the simplicial SIS process for $\widetilde{\sigma}_{i}(t)$. The augmented vertices $(i,2),(k,1),(j,1)$ belong to a different branch compared to $(j,2),(l,1)$, therefore, they evolve independently.}
\end{figure}

\begin{proof}[Proof of Lemma \ref{lemma_nimfa_stoch_representation}]
	We only need to show that $\widetilde{z}_{i,s}(t), s \in \mathcal{S}, i \in V$ solve \eqref{eq:NIMFA} with the appropriate initial conditions.
	
	Assume $t'<t$, and let $h:=t-t'$.  If we assume that $h$ is tiny then the probability there are two or more  points of the PPPs in the interval $[t',t]$ is $o\left(h\right)$. 
Let $A^{(m)}_{\underline{j},i;(\underline{\tilde{s}},s') \to (\underline{\tilde{s}},s)}(t',t)$ denote the event
  that in the time interval $[t',t]$ exactly one instruction arrow with base $\underline{j}$, target $i$ and instruction label $(\underline{\tilde{s}},s') \to (\underline{\tilde{s}},s) $ occurs. We have  
	\begin{equation}\label{prob_single_point} 
 \mathbb{P}\left[ A^{(m)}_{\underline{j},i;(\underline{\tilde{s}},s') \to (\underline{\tilde{s}},s)}(t',t) \right] = r^{(m)}_{\underline{j},i;(\underline{\tilde{s}},s') \to (\underline{\tilde{s}},s)} \cdot h+o\left(h\right).
 \end{equation}
Also note that $A^{(m)}_{\underline{j},i;(\underline{\tilde{s}},s') \to (\underline{\tilde{s}},s)}(t',t)$ is $\mathcal{F}_{[t',t]}$-measurable.
If $A^{(m)}_{\underline{j},i;(\underline{\tilde{s}},s') \to (\underline{\tilde{s}},s)}(t',t)$ occurs then $$ N_i(h,i)=N_i(h,j_1)= \dots= N_i(h,j_m)=1.$$ If $A^{(m)}_{\underline{j},i;(\underline{\tilde{s}},s') \to (\underline{\tilde{s}},s)}(t',t)$ occurs 
and we further have
\begin{equation} \sigma^{i,t}_{(i,1)}(t')=s', \;\; \widetilde{\sigma}^{i,t}_{(j_1,1)}(t')=\tilde{s}_1, \;\; \dots, \;\; \widetilde{\sigma}^{i,t}_{(j_m,1)}(t')=\tilde{s}_m 
\end{equation}
then we will have
 $\sigma_{(i,1)}(t)=s$.
 	Using Lemma \ref{lemma_branching_structure},  we have
\begin{multline}
		\mathds{1}[ A^{(m)}_{\underline{j},i;(\underline{\tilde{s}},s') \to (\underline{\tilde{s}},s)}(t',t)  ] \cdot
		\mathbb{P} \left( \left.  \left \{\sigma^{i,t}_{(i,1)}(t')=s'\} \cap \bigcap_{l=1}^m \{\widetilde{\sigma}^{i,t}_{(j_l,1)}(t')=\tilde{s}_l \right\}  \, \right | \, \mathcal{F}_{[t',t]} \right)= \\
		\mathds{1}[  A^{(m)}_{\underline{j},i;(\underline{\tilde{s}},s') \to (\underline{\tilde{s}},s)}(t',t)  ]\cdot \tilde{z}_{i,s'}(t')\prod_{l=1}^{m} \tilde{z}_{j_l,\tilde{s}_l}(t').
	\end{multline}

	Accounting for the cases when there are zero or one transitions we get that
	\begin{multline*}
	\tilde{z}_{i,s}(t)= \left(1-\sum_{m=0}^{M} \frac{1}{m!} \sum_{\substack{\underline{j} \in V^m \\ \underline{\tilde{s}} \in \mathcal{S}^{m} \\ s' \in \mathcal{S} \setminus \{s\} }} r^{(m)}_{\underline{j}i; (\underline{\tilde{s}},s) \to (\underline{\tilde{s}},s')} \cdot h \cdot \prod_{l=1}^{m} z_{j_l,\tilde{s}_l}(t')  \right)\tilde{z}_{i,s}(t')\\
	+\sum_{m=0}^{M} \frac{1}{m!} \sum_{\substack{\underline{j} \in V^m \\ \underline{\tilde{s}} \in \mathcal{S}^{m} \\ s' \in \mathcal{S} \setminus \{s\} }} r^{(m)}_{\underline{j}i; (\underline{\tilde{s}},s') \to (\underline{\tilde{s}},s)} \cdot h \cdot \tilde{z}_{i,s'}(t')\prod_{l=1}^{m} z_{j_l,\tilde{s}_l}(t')  + o (h).
	\end{multline*}

Now \eqref{eq:NIMFA} follows if we rearrange this and let $h$ go to zero.
\end{proof}

If $N_{i}(t,j)>1$ then we informally say that there are $N_i(t,j)-1$  ghost particles at vertex $j$ at time $t$. Let $G_i(t)$ denote the event that a  ghost appears by time $t$ as we build our auxiliary branching structure from vertex $i$:
\begin{align}\label{def_ghost_event}
G_i(t):= \{\, \exists\, j \in V \, : \, \ N_i(t,j)>1 \, \}.
\end{align}

\begin{corollary}[Ghost probabilities control NIMFA accuracy]
\label{l:ghost}
It follows from \eqref{y_again}, \eqref{mon_N},  Lemma \ref{l:no_ghost} and Lemma \ref{lemma_nimfa_stoch_representation} that we have
\begin{align}
\label{eq:ghost}
 \max_{s \in \mathcal{S}} |y_{i,s}(t)-z_{i,s}(t)|  \leq \pr \left(G_i(t) \right).
\end{align}
\end{corollary}

\section{Proofs}
\label{s:proof}

In Section \ref{subsection_proofs_concentration} we prove the concentration results stated in Section \ref{s:concentration}.

In Section \ref{subsection_proofs_l1_error} we prove the 
$\ell^\infty$ error bounds stated in Section \ref{s:l1_error}. 

In Section \ref{subsection_proofs_average_error} we prove the $\ell^1$ error bounds stated in Section \ref{s:average_error}.

\subsection{Proofs of concentration results stated in Section \ref{s:concentration}}
\label{subsection_proofs_concentration}

Our first goal is to prove Lemma \ref{t:chaos}.

Let $\mathcal{P}(V)$ denote the set of subsets of $V$.
Let $\mathcal{P}^{m}(V)=\mathcal{P}(V)\times \dots \times \mathcal{P}(V)$ denote the $m$-fold Cartesian product of $\mathcal{P}(V)$ with itself.
Recall that $\mathcal{H}_i(t)$ denotes the information set constructed in Section \ref{subsection_backward_particle_sys}.
We will use the following shorthands:
\begin{align*}
\mathcal{H}(t):=&\left( \mathcal{H}_{i_1}(t), \dots, \mathcal{H}_{i_m}(t) \right), \\
H:=& \left(H_{i_1}, \dots, H_{i_m} \right) \in \mathcal{P}^{m}(V), \\
\Phi:=& \left \{ \, H \in \mathcal{P}^{m}(V) \left | \right. \exists \, i ,j \in \{1,\dots m\}, i \neq j : \ H_{i} \cap H_{j} \neq \emptyset \, \right\}. 
\end{align*}
Thus $\pr \left( \mathcal{H}(t) \in \Phi \right)$ is the same as the probability on the r.h.s.\ of \eqref{eq:chaos}.

Due to the construction of the information sets $ \mathcal{H}_{i_1}(t), \dots, \mathcal{H}_{i_m}(t)$  on the same probability space (cf.\ Section \ref{subsection_backward_particle_sys}), they do not evolve independently after they merge. However, somewhat surprisingly, the following lemma still holds. 

\begin{lemma}
\label{lemma:Hi}
\begin{equation}\label{indep_formula_for_Phi}
	\pr \left( \mathcal{H}(t) \in \Phi \right)=\sum_{H \in \Phi} \prod_{l=1}^{m} \pr \left(\mathcal{H}_{i_l}(t)=H_{i_l} \right)
\end{equation}
\end{lemma}

\begin{proof}
\begin{multline*}
\pr \left( \mathcal{H}(t) \in \Phi \right)
=1-\sum_{H \in \mathcal{P}^{m}(V) \setminus \Phi }\pr\left( \mathcal{H}(t)=H\right) \overset{\textit{Claim 
\ref{remark:arrow}}}{=} \\ 1-\sum_{H \in \mathcal{P}^{m}(V) \setminus \Phi } \prod_{l=1}^{m}\pr\left( \mathcal{H}_{i_l}(t)=H_{i_l}\right) 
=  \sum_{H \in \mathcal{P}^{m}(V)  } \prod_{l=1}^{m}\pr\left( \mathcal{H}_{i_l}(t)=H_{i_l}\right)
- \\ \sum_{H \in \mathcal{P}^{m}(V) \setminus \Phi } \prod_{l=1}^{m}\pr\left( \mathcal{H}_{i_l}(t)=H_{i_l}\right) 
= \sum_{H \in  \Phi } \prod_{l=1}^{m}\pr\left( \mathcal{H}_{i_l}(t)=H_{i_l}\right).
\end{multline*}	
\end{proof}

\begin{proof}[Proof of Lemma \ref{t:chaos}] Let us first look at the first term on the l.h.s.\ of \eqref{eq:chaos}.
\begin{align*}
\E\left(\prod_{l=1}^{m} \xi_{i_l,s_l}(t) \right)=&\sum_{H \in \mathcal{P}^m (V) \setminus \Phi} \E \left( \left. \prod_{l=1}^m \xi_{i_l,s_l}(t)  \right| \mathcal{H}(t)=H   \right) \pr \left(\mathcal{H}(t)=H \right) \\
&+\underbrace{\sum_{H \in  \Phi}\E \left( \left. \prod_{l=1}^m \xi_{i_l,s_l}(t)  \right| \mathcal{H}(t)=H   \right) \pr \left(\mathcal{H}(t)=H \right)}_{=:\chi_1}.
\end{align*}

We have $0 \leq \chi_1 \leq \pr \left(\mathcal{H}(t) \in \Phi \right).$
	If $H \in \mathcal{P}^m (V) \setminus V$ then Claim \ref{remark:arrow} gives
\begin{align*}
\pr \left(\mathcal{H}(t)=H \right)=\prod_{l=1}^{m} \pr \left(\mathcal{H}_{i_l}(t)=H_{i_l} \right),
\end{align*}	
while Claim \ref{remakark:xi} implies
\begin{align*}
\E \left( \left. \prod_{l=1}^m \xi_{i_l,s_l}(t)  \right| \mathcal{H}(t)=H   \right)= \prod_{l=1}^m \E \left( \left. \xi_{i_l,s_l}(t)  \right| \mathcal{H}_{i_l}(t)=H_{i_l}   \right),
\end{align*}
resulting in
\begin{align}
\label{eq:cov_lemma_proof}
\begin{split}
\E\left(\prod_{l=1}^{m} \xi_{i_l,s_l}(t) \right)=&\sum_{H \in \mathcal{P}^m (V) \setminus \Phi} \prod_{l=1}^{m} \E \left( \left. \xi_{i_l,s_l}(t) \right| \mathcal{H}_{i_l}(t)=H_{i_l} \right) \pr \left(\mathcal{H}_{i_l}(t)=H_{i_l} \right) \\
&+\chi_1.
\end{split}
\end{align}
As for the second  term on the l.h.s.\ of \eqref{eq:chaos}:
\begin{align*}
\prod_{l=1}^{m}\E \left(\xi_{i_l,s_l}(t) \right)
=&\sum_{H \in \mathcal{P}^m(V) \setminus \Phi}\prod_{l=1}^{m}\E \left(\left.\xi_{i_l,s_l}(t)\right| \mathcal{H}_{i_l}(t)=H_{i_l} \right)\pr \left(\mathcal{H}_{i_l}(t)=H_{i_l} \right)\\
&+\underbrace{\sum_{H \in \Phi}\prod_{l=1}^{m}\E \left(\left.\xi_{i_l,s_l}(t)\right| \mathcal{H}_{i_l}(t)=H_{i_l} \right)\pr \left(\mathcal{H}_{i_l}(t)=H_{i_l} \right)}_{=: \chi_2}
\end{align*}
The first term is the same as in \eqref{eq:cov_lemma_proof} while the last one can be bounded by
\begin{align*}
0 \leq \chi_2 \leq \sum_{H \in \Phi} \prod_{l=1}^{m} \pr \left(\mathcal{H}_{i_l}(t)=H_{i_l} \right) \overset{\textit{Lemma \ref{lemma:Hi}}}{=}\pr \left(\mathcal{H}(t) \in \Phi\right),
\end{align*}
thus,
\begin{align*}
 \left| \E \left(\prod_{l=1}^{m} \xi_{i_l,s_l}(t)\right)-\prod_{l=1}^{m} \E \left(\xi_{i_l,s_l}(t)\right) \right|= \left|\chi_1-\chi_2 \right| \leq \max \{\chi_1, \chi_2 \} \leq \pr \left(\mathcal{H}(t) \in \Phi\right).
\end{align*}
\end{proof}

Our next goal is to prove Theorem \ref{t:concentration_bound}.

\begin{definition}[Infection path]\label{def_infection_path_old}
	For $M=1$ and $i,j \in V$ let $ i \stackrel{[a,b]}{\rightarrow} j $ denote the event that there exist $k \in \mathbb{N}_{0}$, $a<s_1<s_2<\dots<s_k <b$
	and vertices $i_1,\dots, i_{k+1}$ such that $i_1=i$, $i_{k+1} =j$ and $i_\ell \neq i_{\ell+1}$ for any $\ell=1,\dots,k$ and  $s_\ell$ is a point of the PPP $\underline{\mathcal{T}}_{i_\ell i_{\ell+1}}$ for each $\ell=1,\dots, k$.
	
	In the case when the interval is $[0,t]$, we use the simplified notation $ i \stackrel{t}{\rightarrow} j. $
\end{definition}


Recall from Claim \ref{c:H} that
if $M=1$ then the information set $\mathcal{H}_i(t)$ evolves as an SI process on the oriented weighted graph $R^{\intercal}$.

\begin{definition}[Forward-backward infection path]
In the $M=1$ case we denote by $\overleftrightarrow{\mathcal{H}}_i(t)$  the set of infected vertices in the following modification of the SI process starting from only vertex $i \in V$ being initially infected: on $[0,t]$, $k$ infects $l$ at rate $\bar{r}_{lk}$, while on $[t,2t]$ it happens at rate $\bar{r}_{kl}.$ In other words:
\begin{align*}
\overleftrightarrow{\mathcal{H}}_i(t)=\left \{ j \in V \; \left | \; \exists \, k \in \mathcal{H}_i(t) \ : \ k \overset{[t,2t]}{\to} j \right. \right \}.
\end{align*}

We use the notation
$$\{ i \overset{t}{\leftrightarrow} j \}:=\left \{ j \in \overleftrightarrow{\mathcal{H}}_i(t)  \right \}.$$
\end{definition}

\begin{claim}
\label{r:sym}
If $R=R^{\intercal}$  then
$ 	\pr \left(i \overset{t}{ \leftrightarrow} j \right)=\pr \left( i \overset{2t}{ \to} j \right) $
holds, since the reversal of the arrows of $\underline{\mathcal{T}}_{kl}$  on $[0,t]$ leaves the distribution of the PPP unchanged.
\end{claim}
The statement of Claim \ref{r:sym} requires no further validation.

\begin{lemma}
\label{l:H_collision}
For any $i,j \in V$ we have 
\begin{align}
\label{eq:H_collision}
\pr \left( \mathcal{H}_i(t) \cap \mathcal{H}_j(t) \neq \emptyset \right)=\pr \left(i \overset{t}{\leftrightarrow} j \right).
\end{align}
\end{lemma}

\begin{proof}
Let us denote by
$\mathcal{H}_{j}^{+}(t)$  a shifted version of $\mathcal{H}_{j}(t)$ where we generate the progeny of vertex $j$ on the interval $[t,2t]$. Since $[0,t]$ and $[t,2t]$ are disjoint, $\mathcal{H}_i(t) ,\mathcal{H}_{j}^{+}(t)$ are independent, moreover $\mathcal{H}_{j}^{+}(t) , \mathcal{H}_j(t)$ have the same distribution. The following events are equivalent:
\begin{align}
i \overset{t}{\leftrightarrow} j \ &\Leftrightarrow \ \exists k \in \mathcal{H}_{i}(t) \ : \ k \overset{[t,2t]}{\to} j \ \Leftrightarrow \ 	\exists k \in \mathcal{H}_{i}(t) \ : \ k \in \mathcal{H}_{j}^{+}(t) \nonumber \\ \label{doublearrow_intersect_equiv}
 & \Leftrightarrow \ \mathcal{H}_i(t) \cap \mathcal{H}_j^{+}(t) \neq \emptyset.
\end{align}
Using Lemma \ref{lemma:Hi} in the first step we have
\begin{multline*}
	\pr\left(\mathcal{H}_{i}(t) \cap \mathcal{H}_j(t) \neq \emptyset \right) = \sum_{H_i \cap H_j \neq \emptyset} \pr \left(\mathcal{H}_{i}(t)=H_i \right)\pr \left(\mathcal{H}_{j}(t)=H_j \right)= \\
	\sum_{H_i \cap H_j \neq \emptyset} \pr \left(\mathcal{H}_{i}(t) =H_i \right)\pr \left(\mathcal{H}_{j}^{+}(t)=H_j \right)= \\
	 \sum_{H_i \cap H_j \neq \emptyset} \pr \left(\mathcal{H}_{i}(t) =H_i, \mathcal{H}_{j}^{+}(t)=H_j  \right) 
		 \stackrel{\eqref{doublearrow_intersect_equiv}}{=}\pr \left(i \overset{t}{\leftrightarrow} j \right).
\end{multline*}
\end{proof}

\begin{corollary}
	\label{c:symm}
	Based on Lemma \ref{l:H_collision} and Claim \ref{r:sym},   in the $M=1$ case with symmetric rates we have $\pr \left( \mathcal{H}_i(t) \cap \mathcal{H}_j(t) \neq \emptyset \right)= \pr \left(i \overset{2t}{ \to} j \right)$ 
	for any $i,j \in V.$ Note that for any $\mathcal{M} \subseteq V$ we have
\begin{multline}\label{sum_of_nonintersect}
		\frac{1}{\left| \mathcal{M} \right|^2}\sum_{i,j \in \mathcal{M}}\pr \left( \mathcal{H}_i(t) \cap \mathcal{H}_j(t) \neq \emptyset \right)=\frac{1}{\left| \mathcal{M} \right|^2}\sum_{i,j \in \mathcal{M}} \pr \left(i \overset{2t}{ \to} j \right)=\\
		\frac{1}{\left| \mathcal{M} \right|^2}\sum_{i \in \mathcal{M}}\E \left[ \sum_{j \in \mathcal{M}} \1{i \overset{2t}{\to} j} \right]=
  \frac{1}{\left| \mathcal{M} \right|}\sum_{i \in \mathcal{M}} \E \left( \frac{\left| \mathcal{H}_i(2t) \cap \mathcal{M} \right|}{\left|\mathcal{M} \right|} \right).
	\end{multline}
	Recalling the definition of the empirical average $\bar{\xi}^{\mathcal{M}}_{s}(t)$ from \eqref{emp_average}, we have
	\begin{multline}
		\label{eq:concentration2}
		\Var \left( \bar{\xi}^{\mathcal{M}}_{s}(t) \right) = \frac{1}{| \mathcal{M}|^2}\sum_{i,j \in \mathcal{M}} \cov \left(\xi_{i,s}(t),\xi_{j,s} \right) \stackrel{\eqref{eq:chaos}, \eqref{sum_of_nonintersect}}{\leq} \\ \frac{1}{\left| \mathcal{M} \right|}\sum_{i \in \mathcal{M}} \E \left( \frac{\left| \mathcal{H}_i(2t) \cap \mathcal{M} \right|}{\left|\mathcal{M} \right|} \right).
	\end{multline}
	
\end{corollary}

\begin{lemma} If $M=1$ then  we have 
\label{l:cov_exp_bound}
\begin{align}
\label{eq:cov_exp_bound}
\pr \left(\mathcal{H}_{i}(t) \cap \mathcal{H}_j(t) \neq \emptyset \right) \leq \left( e^{R^\intercal t} e^{R t}  \right)_{ij}
\end{align}
\end{lemma}

\begin{proof}
Recall Lemma \ref{l:H_collision}. In order to bound the probability of the event $i \overset{t}{\leftrightarrow} j$, we dominate the backwards  SI process on $[0,t]$ and the forwards SI process on $[t,2t]$ by  branching random walks (cf.\ Claim \ref{claim_brw_def}) on the weighted graph $R^\intercal$ (reversed arrows) and $R$ (original arrows), respectively.

Let $\eta(t)$ denote the occupancy configuration of a branching random walk on the weighted graph $R^\intercal$ with initial condition $\eta_{k}(0)=\mathds{1}[i=k]$ and let $\tilde{\eta}(t)$ denote another branching random walk on the weighted graph $R$ with initial condition $\tilde{\eta}(0)=\eta(t).$ We have $\{i \overset{t}{\leftrightarrow} j \} \subseteq \{\tilde{\eta}_j(t)>0 \}$ and thus
\begin{align*}
\pr \left(i \overset{t}{\leftrightarrow} j \right) & \leq \pr \left( \tilde{\eta}_j(t)>0 \right)= \pr (\tilde{\eta}_j(t) \geq 1) \leq \E \left( \tilde{\eta}_{j}(t) \right), \\
\E(\eta_{k}(t))=& \left(e^{R t} \eta(0) \right)_k=(e^{R t})_{ki}, \\
\E(\tilde{\eta}_{j}(t))=& \sum_{k} \left(e^{R^\intercal t}\right)_{jk} \left(e^{R t} \right)_{ki}= \left( e^{R^\intercal t} e^{R t} \right)_{ji}=\left( e^{R^\intercal t} e^{R t} \right)_{ij}.
\end{align*}

\end{proof}

\begin{proof}[Proof of Theorem \ref{t:concentration_bound}]
Using Lemma \ref{t:chaos} with $m=2$ we obtain 
\begin{equation}\label{varbound_nonint}
\Var \left(\bar{\xi}^{\mathcal{M}}_s(t) \right)
\leq \frac{1}{| \mathcal{M}|^2}\sum_{i,j \in \mathcal{M}} \pr \left(\mathcal{H}_i(t) \cap \mathcal{H}_j(t) \neq \emptyset \right).
\end{equation}
By our assumption $M=1$ and \eqref{eq:cov_exp_bound}, the r.h.s.\ of \eqref{varbound_nonint} is at most
\begin{multline*}
 \frac{1}{| \mathcal{M}|^2}\sum_{i,j \in \mathcal{M}} \left(e^{R^\intercal t}e^{R t} \right)_{ij}= 
\frac{1}{| \mathcal{M}|^2} \1{\mathcal{M}}^\intercal e^{R^\intercal t}e^{R t}\1{\mathcal{M}}
= \\ \frac{1}{| \mathcal{M}|^2} \left \|e^{R t} \1{\mathcal{M}} \right \|_2^2 \leq  \frac{1}{| \mathcal{M}|^2} e^{2 \|R \|_2 t} \|\1{\mathcal{M}} \|_2^2= 
 \frac{1}{| \mathcal{M}|} e^{2 \|R \|_2 t},
\end{multline*}
where we denoted by $\1{\mathcal{M}}$ the indicator vector of the set $\mathcal{M}$.
\end{proof}

\begin{proof}[Proof of Proposition \ref{propp:Erdos_Renyi}]
	Based on \eqref{eq:concentration2} we have to upper bound the mean of $\left|\mathcal{H}_{i}(2t) \cap \mathcal{M} \right| $. Now $\mathcal{H}_i(2t)$ has the same distribution as the number of infections at time $2t$ in an $SI$ process with infection rate $\bar{r}_{ji}=a_{ji} \bar{q}$ starting from a single initial infection at vertex $i$. 	
 
 The SI process and the Erdős-Rényi graph can be generated simultaneously in the following fashion: let us consider an SI process on $K_N$ with rate $\bar{q}$ and if $i$ tries to infect $j$ (for the first time), we allow the infection to happen with probability $\frac{\lambda}{N}$. By allowing multiple infection attempts,  $\left| \mathcal{H}_{i}(t) \right|$ can be dominated by a pure birth process $Y(t)$ with birth rate $ \frac{\lambda \bar{q}}{N}  Y(t) \left(N-Y(t) \right) \leq  \lambda \bar{q} Y(t) $.
	
We  can thus estimate
	\begin{multline*}
\mathbb{E} \left[\Var \left(  \bar{\xi}^{\mathcal{M}}_{s}(t) \, \left| \, \mathcal{G}\left(N, \frac{\lambda}{N} \right) \right. \right) \right] \stackrel{\eqref{eq:concentration2}}{\leq} \\
	\E \left[\frac{1}{\left| \mathcal{M} \right|^2} \sum_{i \in \mathcal{M}} \E \left( \left| \mathcal{H}_{i}(2t) \cap \mathcal{M} \right|   \,\left| \, \mathcal{G}\left(N, \frac{\lambda }{N} \right. \right)  \right) \right]  \leq  \frac{1}{\left| \mathcal{M} \right|} \E \left( Y(2t)\right) \leq \frac{1}{\left| \mathcal{M} \right|}e^{2 \lambda \bar{q} t}  
\end{multline*}
	Now the desired estimate \eqref{ER_var_bound} follows by Markov's inequality.
\end{proof}
Now we return to our most general setting (an interacting particle system with oriented higher order interactions), cf.\ Section \ref{s:transition_rates}.
Our  goal is to prove Theorem \ref{t:concentration_delta_max}. Recall the notion of  $\delta_{max}$ and $\tilde{r}_{max}$ from \eqref{eq:delta_max} and \eqref{eq:r_max}, respectively.
\begin{lemma}
\label{l:k_in_H_i}
For any $i \neq k \in V$ we have 
\begin{equation}\label{Hit_anywhere}
\pr \left(k \in \mathcal{H}_i(t) \right) \leq t \cdot e^{\delta_{max}Mt}\cdot \tilde{r}_{max}. 
\end{equation}
\end{lemma}

\begin{proof}
The size of $\mathcal{H}_i(t)$ can be dominated by a pure birth process $Y(t)$ that jumps from $n$ to $n+M$ at rate $\delta_{max}n$.  The mean of $Y(t)$ satisfies the ODE
$ \frac{\d}{\d t}\E(Y(t))=\delta_{max}M \E \left(Y(t) \right)$ 
with initial condition $ 
Y(0)=1$, thus 
\begin{equation}
\label{eq:|H_i(t)|_bound}
\E \left(\left|\mathcal{H}_i(t) \right| \right) \leq \E \left(Y(t)\right) = e^{\delta_{max}M t}. 
\end{equation} 

Recalling \eqref{eq:r_tilde}, the rate at which vertex $k$ gets included in set $\mathcal{H}_{i}(t)$ is at most 
$$\sum_{l \in \mathcal{H}_i(t)} \tilde{r}_{kl} \leq \left| \mathcal{H}_{i}(t) \right| \tilde{r}_{max}$$
as we have to look at arrows that include $k$ in their base and point towards  an element of $\mathcal{H}_i(t)$. Thus we have
$\pr \left(k \in \mathcal{H}_i(t) \right) \leq \int_0^t \E \left(\left|\mathcal{H}_i(\tau) \right| \right)\tilde{r}_{max} \, \mathrm{d}\tau$, which, together with \eqref{eq:|H_i(t)|_bound}
implies
\eqref{Hit_anywhere}.
\end{proof}

\begin{proof}[Proof of
Theorem \ref{t:concentration_delta_max}] 
Let $i \neq j \in V$.
For any $i \not \in H \subseteq V$ we  have
\begin{equation}\label{det_set_H_intersects_Hit}
\pr \left(\mathcal{H}_{i}(t) \cap H \neq \emptyset \right) \leq \sum_{k \in H} \pr \left(k \in \mathcal{H}_i(t) \right) \stackrel{\eqref{Hit_anywhere}}{\leq} |H| t e^{\delta_{max}Mt} \tilde{r}_{max}.
\end{equation}
We have
\begin{multline*}
\pr \left( \mathcal{H}_{i}(t) \cap \mathcal{H}_j(t) \neq \emptyset \right) \overset{\eqref{indep_formula_for_Phi} }{=} \sum_{H \cap H' \neq \emptyset}
\mathbb{P}(\mathcal{H}_i(t)=H' )
\mathbb{P}(\mathcal{H}_j(t)=H) \leq \\
\pr \left(i \in \mathcal{H}_j(t) \right)+  \sum_{i \not \in H } \pr \left( \mathcal{H}_{i}(t) \cap H \neq \emptyset \right) \pr \left(\mathcal{H}_j(t)=H \right)  \stackrel{\eqref{Hit_anywhere}, \eqref{det_set_H_intersects_Hit}}{\leq} \\
t \cdot e^{\delta_{max}Mt}\cdot \tilde{r}_{max} \cdot  \left( 1+ \E \left(\left|\mathcal{H}_i(t) \right| \right) \right) \stackrel{\eqref{eq:|H_i(t)|_bound}}{\leq}
t \cdot e^{\delta_{max}Mt}(1+e^{\delta_{max}Mt})\cdot \tilde{r}_{max}.
\end{multline*}	
\end{proof}

\begin{proof}[Proof of Theorem \ref{t:concentration_characterisation}]
Let us first assume that $\theta_t(n) \to 0$ for all $t \geq 0$ where
	$$  \theta_t(n):=\frac{1}{\left|\mathcal{M}(n) \right|^2}\sum_{i,j \in \mathcal{M}(n)} \pr \left(\mathcal{H}_i(t) \cap \mathcal{H}_i(t) \neq \emptyset \right). $$
Equation \eqref{eq:concentration} implies $\Var \left(\bar{\xi}^{\mathcal{M}(n)}_{s}(t) \right) \leq  \theta_t(n) \to 0$ for all $t \geq 0$ and $s \in \mathcal{S}$.

Conversely, let us assume $\theta_t(n) \not \to 0$ for some $t \geq 0$. We have
\begin{multline}\label{theta_psi}
	\theta_t(n) \stackrel{\eqref{sum_of_nonintersect}}{=}\frac{1}{\left| \mathcal{M}(n) \right|}\sum_{i \in \mathcal{M}(n)} \E \left( \frac{\left| \mathcal{H}_i(2t) \cap \mathcal{M}(n) \right|}{\left|\mathcal{M}(n) \right|} \right) \leq \\
	  \E \left( \frac{\left| \mathcal{H}_{i(n)}(2t) \cap \mathcal{M}(n) \right|}{\left|\mathcal{M}(n) \right|} \right)=:\psi_t(n)
\end{multline}
where
$$i(n):=\mathrm{argmax}_{i \in \mathcal{M}(n)} \E \left( \frac{\left| \mathcal{H}_i(2t) \cap \mathcal{M}(n) \right|}{\left|\mathcal{M}(n) \right|} \right).  $$

Let us consider an SI model with initial conditions
 $\pr\left(\xi_{i(n),I}(0)=1\right)=\frac{1}{2}$ and
	 $\xi_{j,S}(0)=1$ for all of the other vertices of $V(n)$.

If $\xi_{i(n),I}(0)=0,$ then  $\bar{\xi}^{\mathcal{M}(n)}_{I}(2t)=0$.  On the other hand, if $\xi_{i(n),I}(0)=1$ then $\bar{\xi}^{\mathcal{M}(n)}_{I}(2t)$ has the same distribution as $\frac{\left| \mathcal{H}_{i(n)}(2t) \cap \mathcal{M}(n) \right|}{\left|\mathcal{M}(n) \right|}$ by Claim \ref{c:H}.
\begin{align}
\label{expect_psi} \E \left[ \bar{\xi}^{\mathcal{M}(n)}_{I}(2t) \right]=& \frac{1}{2} \E \left[ \left. \bar{\xi}^{\mathcal{M}(n)}_{I}(2t) \right|  \xi_{i(n),I}(0)=1 \right]=\frac{1}{2} \psi_t(n), \\
\nonumber \E \left[ \left(\bar{\xi}^{\mathcal{M}(n)}_{I}(2t) \right)^2 \right]=&\frac{1}{2} \E \left[ \left.  \left(\bar{\xi}^{\mathcal{M}(n)}_{I}(2t) \right)^2 \right|  \xi_{i(n),I}(0)=1 \right] \\
 \label{var_psi} \geq & \frac{1}{2} \left( \E \left[ \left. \bar{\xi}^{\mathcal{M}(n)}_{I}(2t) \right|  \xi_{i(n),I}(0)=1 \right]\right)^2=\frac{1}{2} \psi_t^2(n), \\
 \nonumber \Var \left( \bar{\xi}^{\mathcal{M}(n)}_{I}(2t) \right) \stackrel{\eqref{expect_psi}, \eqref{var_psi} }{\geq} &   \frac{1}{2} \psi^2_t(n)  -\left( \frac{1}{2} \psi_t(n) \right)^2 = \frac{1}{4} \psi_t^2(n) \stackrel{ \eqref{theta_psi} }{\geq} \frac{1}{4} \theta_t^2(n) \not \to 0.
\end{align}

\end{proof}

\subsection{Proofs of $\ell^\infty$ error bounds stated in Section \ref{s:l1_error} }\label{subsection_proofs_l1_error}

Recall the definition of the bad event $G_i(t)$ from \eqref{def_ghost_event}.

\begin{proof}[Proof of Theorem \ref{t:mean_field_uniform}]
	Based on Corollary \ref{l:ghost}, it is enough to show a uniform bound in $i$ for the probability $G_i(t)$ of the appearance of ghosts.

The first ghost appears at tome $\tau$ if the  arrow at that time is pointing  to a vertex in $\mathcal{H}_i(\tau)$ such that the base of the arrow has a non-empty intersection with $ \mathcal{H}_i(\tau)$.

Recall from \eqref{eq:r_tilde} that $\tilde{r}_{kl}$ is the total rate at which an arrow that has vertex $k$  in its base and points towards vertex $l$ appears. Thus, the rate at which the first ghost is created at time $\tau$ is at most
\begin{equation}\label{ghost_size_squared}
\sum_{k,l \in \mathcal{H}_i(\tau)} \tilde{r}_{kl} \stackrel{\eqref{eq:r_max}}{\leq} \tilde{r}_{max} \left| \mathcal{H}_i(\tau) \right|^2.
\end{equation}

Like in the proof of  Lemma \ref{l:k_in_H_i}, we dominate $\left| \mathcal{H}_i(t) \right|$ by a pure birth process $Y(t)$ that starts from $Y(0)=1$ and jumps from $n$ to $n+M$ at rate $\delta_{max}n$. 
\begin{align}
\nonumber
\frac{\d}{\d t} \E \left(Y^2(t) \right)=&\E \left( \left[\left(Y(t)+M \right)^2-Y^2(t) \right] \delta _{max} Y(t)  \right) \\
\nonumber
\stackrel{ \eqref{eq:|H_i(t)|_bound} }{=} & 2M \delta_{max} \E \left(Y^2(t) \right)+M^2 \delta_{max} e^{\delta_{max}M t}, \\
\label{Y_second_momet_bound}
\E \left(Y^2(t) \right) =& e^{2M \delta_{max}t} \left(1+ \int_{0}^{t} M^2 \delta_{max}e^{-M\delta_{max} \tau} \d \tau \right) \leq 2M e^{2M \delta_{max}t},
\end{align}
making 
\begin{multline}\label{hoi_ghost_bound}
\pr\left( G_{i}(t) \right) \stackrel{\eqref{ghost_size_squared}}{\leq}  \int_{0}^{t} \tilde{r}_{max} \E\left( \left| \mathcal{H}_i(\tau) \right|^2 \right) \d \tau \leq \\
\int_{0}^{t} \tilde{r}_{max} \E\left( Y^2(\tau) \right) \d \tau
\stackrel{\eqref{Y_second_momet_bound}}{\leq} 
2Mt e^{2M \delta_{max}t} \cdot \tilde{r}_{max}.
\end{multline}
The proof of \eqref{eq:r_max_bound} follows from \eqref{hoi_ghost_bound} and Corollary \ref{l:ghost}.
\end{proof}

Our next goal is to prove the lower bounds stated in Theorem \ref{t:mean_field_lower_bound}. In the case of \eqref{eq:mean_field_lower_bound_general}, we have some freedom in the choice of our  Markov process, and the following trick will make our life more convenient.
\begin{claim}[No self-interactions in the proofs of lower bounds]
\label{r:wlog}
Without the loss of generality we may assume $\bar{r}^{(0)}_{i}=0$  (i.e., there are no self-interactions) when we construct counterexamples with a given set of total rates $\left(\bar{r}^{(m)}_{\underline{j}i} \right)$. Otherwise, can augment the state space to $\mathcal{S}':=\mathcal{S} \times\{\alpha, \beta\}$ and assume for any $s \in \mathcal{S}$ vertices in state $(s,\alpha)$ transition to state $(s,\beta)$ at rate $r_{i,(s,\alpha) \to (s,\beta)}^{(0)}=\frac{1}{|\mathcal{S}|} \bar{r}_i^{(0)}$, ensuring
\begin{align*}
\sum_{s' \in \mathcal{S}'}r_{i,s'}^{(0)}=\bar{r}_{i}^{(0)}.
\end{align*}
We can set the initial conditions in a way that the second coordinate is $\beta$ for all vertices, therefore self interactions do not materialize during the evolution of the process, simulating a dynamics on $\mathcal{S}$ without self interactions.
\end{claim}
The statement of Claim \ref{r:wlog} requires no further validation.

\begin{proof}[Proof of Theorem \ref{t:mean_field_lower_bound} \eqref{lilb1}]
Using Claim \ref{r:wlog} we may assume $\bar{r}_i=0$.

Choose the pair $j,i \in V$ such that $\tilde{r}_{ji}=\tilde{r}_{max}$.

We would like to create a scenario where only the state of vertex $i$ can change and only via interactions that include vertex $j$. For this reason we modify the simplicial SI process (cf.\ Section \ref{subsection_SIS}) and put  vertex $j$ is in a special state $\star$. A collection of $m$ vertices can infect another susceptible one if $m-1$ of them are infected and one if them is in state $\star$. Thus the state space is $\mathcal{S}=\{\star, S,I\} $ and the corresponding transition rates are
\begin{align*}
r_{\underline{k}i;(\star, I,\dots, I, S) \to (\star, I,\dots, I, I) }^{(m)}=\bar{r}^{(m)}_{\underline{k}i}.
\end{align*}
The initial conditions are $\xi_{i,S}(0)=1,$ $\pr \left(\xi_{j,\star}(0)=1 \right)=\pr \left( \xi_{j,I}(0)=1 \right)=\frac{1}{2}$ and $\xi_{k,I}(0)=1$ for all $k \in V \setminus \{i,j\}.$

In the real interacting particle system, vertex $i$ remains forever in state $S$ if vertex $j$ is not in state $\star$ as any infection requires exactly one vertex in state $\star$. If $\xi_{j, \star}(0)=1$ then any arrow that includes vertex $j$ in its base and points at $i$ can infect vertex $i$. According to \eqref{eq:r_tilde}, these arrows have a total rate of $\tilde{r}_{ji}=\tilde{r}_{max}$, which makes
\begin{align*}
y_{i,S}(t)=\frac{1}{2}e^{-\tilde{r}_{max}t}+\frac{1}{2}, \qquad y_{i,I}(t)=1-y_{i,S}(t)=\frac{1}{2}-\frac{1}{2}e^{-\tilde{r}_{max}t}.
\end{align*}

As for NIMFA, we have $z_{j,\star}(t) \equiv \frac{1}{2}$ since there are no transitions changing the state $\star$ or turning another state into state $\star$. Similarly, any vertex in state $I$ remains there forever, making $z_{k,I}(t) \equiv 1$ for $k \in V \setminus \{i,j\}.$  Therefore we have
\begin{align*}
\frac{\d}{\d t}z_{i,S}(t) \stackrel{\eqref{eq:NIMFA}}{=}&-z_{i,S}(t) \sum_{m=1}^{M} \sum_{k_1<\dots <k_m} \bar{r}^{(m)}_{\underline{k}i} z_{k_1,\star}(t) \prod_{l=2}^{m}z_{k_l,I}(t) \\
=& -\frac{1}{2}z_{i,S}(t) \sum_{m=1}^{M} \frac{1}{m!} \sum_{ \substack{\underline{k} \in V^m \\ j \in \underline{k}}} \bar{r}^{(m)}_{\underline{k}i}=-\frac{1}{2}\underbrace{\tilde{r}_{ji}}_{=\tilde{r}_{max}}z_{i,S}(t), \\
z_{i,S}(t)=& e^{-\frac{1}{2}\tilde{r}_{max}t}, \qquad z_{i,I}(t)=1-e^{-\frac{1}{2}\tilde{r}_{max}t}.
\end{align*}
Putting these formulas together we obtain:
\begin{equation}
	z_{i,I}(1)-y_{i,I}(1)=\left(1-e^{-\frac{1}{2}\tilde{r}_{max} } \right)-\frac{1}{2}(1-e^{-\tilde{r}_{max}}) = \frac{1}{2} \left( 1-e^{-\frac{1}{2}\tilde{r}_{max}}\right)^2.
\end{equation}
The proof of \eqref{eq:mean_field_lower_bound_general} is complete.
\end{proof}

\begin{proof}[Proof of Theorem \ref{t:mean_field_lower_bound} \eqref{lilb2}]
	Using Claim \ref{r:wlog} we can set $\bar{r}_i=0$.

We take an SI process on an unweighted graph with rates $\bar{r}_{i j}=\frac{a_{ij}}{\bar{d}}$ with $\bar{d} \geq 1$. Take an arbitrary $i$ with $d_i \geq \bar{d}$ and make it infected with probability $\frac{1}{2}$ while the other vertices are initially susceptible.

The key idea is to capture an effect in NIMFA where the vertices that got infected by vertex $i$ infect back said vertex, slightly increasing its infection probability from $\frac{1}{2}$ while the analogous quantity stays constant in the real SI process. A simple heuristic calculation suggest that vertex $i$ infects roughly $ \frac{d_i}{\bar{d}}$ other vertices, while each of these vertices sends back an infection signal of order $ \frac{1}{\bar{d}}$, which amounts to an increase in the infection probability that is comparable to $ \frac{d_i}{\bar{d}^2}  \geq \frac{1}{\bar{d}} =\tilde{r}_{max}$. Now let us go through the details.

If  $\xi_{i,I}(0)=0$ then the process remains disease free. When $\xi_{i,I}(0)=1$, vertex $i$ remains infected forever, therefore we have $y_{i,I}(t) = \frac{1}{2}$.

As for NIMFA, we make the following observation: the derivative of $z_{j,S}(t)$ can be written as
\begin{equation}\label{nimfa_SI_simplegraph}
 \frac{\d}{\d t}z_{j,S}(t)=-z_{j,S}(t) \sum_{k \in V} \frac{a_{jk}}{\bar{d}}z_{k,I}(t).
\end{equation}

If we treat $\sum_{k \in V} \frac{a_{jk}}{\bar{d}}z_{k,I}(t)$ as a known function then \eqref{nimfa_SI_simplegraph} becomes a linear ODE. For any $j \in V \setminus \{ i \}$ the solution of this linear ODE satisfies
\begin{align*}
z_{j,S}(t)=\exp \left(-\int_{0}^{t}\sum_{k \in V} \frac{a_{jk}}{\bar{d}}z_{k,I}(\tau) \d \tau \right)
 \overset{\frac{1}{2} \leq z_{i,I}(\tau)}{\leq} \exp \left(-\frac{1}{2} \frac{a_{ij}}{\bar{d}} t \right),
\end{align*}
where we used the fact that $\tau \mapsto z_{k,I}(\tau)$ is a non-negative monotone increasing function (since it has non-negative derivative). Note that $1-e^{-x} \geq \frac{1}{2} x$ for $0 \leq x \leq 1$ and $0 \leq \frac{1}{2} \frac{a_{ij}}{\bar{d}} t  \leq 1$ when $0 \leq t \leq 1$, making
\begin{equation}
\label{lower_z_j_infected}
z_{j,I}(t) \geq 1-\exp \left(-\frac{1}{2} \frac{a_{ij}}{\bar{d}} t \right) \geq \frac{1}{4} \frac{a_{ij}}{\bar{d}} t, \qquad 0 \leq t \leq 1, \quad j \in V \setminus \{ i \}.
\end{equation}
We can bound $z_{i,S}(1)$ similarly:
\begin{multline*}
z_{i,S}(1)=\frac{1}{2} \exp \left(-\int_{0}^{1} \sum_{j} \frac{a_{ij}}{\bar{d}} z_{j,I}(t) \d t \right) \stackrel{\eqref{lower_z_j_infected}}{\leq}
\\
\frac{1}{2} \exp \left(-\int_{0}^{1} \sum_{j}  \frac{1}{4} \frac{a_{ij}}{\bar{d}^2} t  \d t \right) 
=  \frac{1}{2} \exp \left(-\frac{1}{8} \frac{d_i}{\bar{d}^2}\right) \leq \frac{1}{2} e^{-\frac{1}{8 \bar{d}}}, \\
z_{i,I}(1)-y_{i,I}(1) \geq \left(1-\frac{1}{2} e^{-\frac{1}{8 \bar{d}}} \right)-\frac{1}{2}=\frac{1}{2} \left(1-e^{-\frac{1}{8 \bar{d}}} \right) \geq \frac{1}{32 \bar{d}}= \frac{1}{32} \tilde{r}_{max}.
\end{multline*}
The proof of \eqref{eq:mean_field_lower_bound_tight} is complete.
\end{proof}

\subsection{Proofs of $\ell^1$ error bounds stated in  Section \ref{s:average_error}}\label{subsection_proofs_average_error}

Recall the definition of the bad event $G_i(t)$ from \eqref{def_ghost_event}.
In order to prove the $\ell^1$ error bounds on NIMFA stated in Theorem \ref{t:l1}, we need more sophisticated bounds than \eqref{hoi_ghost_bound} on the probability $G_i(t)$ that a ghost appears before $t$.

\begin{lemma}[Bounding the probability of ghosts]\label{lemma_ghost_bk}
	 Assume we have at most pair interactions $(M=1)$ with symmetric rates $R^\intercal=R.$ Then we have
	\begin{equation}\label{eq_ghost_bk}
		\pr \left(G_i(t) \right) \leq  \sum_{m \in V } \left(e^{Rt} \right)_{im}\left(e^{2Rt}-I \right)_{mm}.
	\end{equation}
\end{lemma}

In order to prove Lemma \ref{lemma_ghost_bk}, we need some notation.

Recalling Definition \ref{def_infection_path_old},  the event $i \stackrel{t}{\rightarrow} j $
occurs if there is an infection path from $i$ to $j$. There we assumed that the paths can take length $k=0$, making $ i \overset{t}{ \to } i$ a sure event. We want to modify Definition \ref{def_infection_path_old} such that length zero paths are excluded, which can be achieved by replacing $\mathbb{N}_{0}$ with $\mathbb{N}_{+}$ in the definition.

\begin{definition}\label{def_infection_path}
	For $M=1$ $i,j \in V$ let $ i \stackrel{*t}{\rightarrow} j $ denote the event that there exist $k \in \mathbb{N}_+$, $0<s_1<s_2<\dots<s_k <t$
	and vertices $i_1,\dots, i_{k+1}$ such that $i_1=i$, $i_{k+1} =j$ and $i_\ell \neq i_{\ell+1}$ for any $\ell=1,\dots,k$ and we have $s_\ell \in \underline{\mathcal{T}}_{i_\ell i_{\ell+1}}$ for any $\ell=1,\dots, k$.
\end{definition}

Note that if $i=j$ then we must have paths with length $k \geq 2$ as $\bar{r}_{ii}=0$.

Note that if  $R=R^{\intercal}$ then we obtain
\begin{equation}
	\label{inf_path_matrixexp}
	\mathbb{P}( i \stackrel{t}{\rightarrow} j) \leq
	\left(e^{Rt} \right)_{ij}, \qquad i,j \in V
\end{equation}
\begin{equation}
	\label{inf_loop_matrixexp}
	\mathbb{P}( i \stackrel{*t}{\rightarrow} i) \leq
	\left(e^{Rt} \right)_{ii} -1, \qquad i \in V
\end{equation}
by bounding a the SI process with a branching random walk.

We will  use the BK inequality for Poisson point processes proved in \cite{vdBpoiBK}. Let $\Omega$ denote the space of joint realizations of the PPPs $\underline{\mathcal{T}}_{ij}, i \neq j \in V$. According to the notation of \cite{vdBpoiBK}, elements of $\Omega$ can be viewed as \emph{finite marked point configurations} on $[0,t]$, where the mark of the point $s \in [0,t]$ is $(i,j)$ if $s \in \underline{\mathcal{T}}_{ij}$. Note that w.l.o.g.\ we may assume that the multiplicity of each point is one, thus each element $\omega$ of $\Omega$ can be identified with finite set of marked points.

\begin{definition}[Increasing event]
	If $A$ is an event defined on $\Omega$ then we say that $A$ is an increasing event if the following  holds: if a marked point  configuration $\omega$ is in $A$ and if $\omega'$ is any other marked point configuration then $\omega \cup \omega'$ is also in $A$.
\end{definition}
Note that the events $G_i(t) $,  $i \stackrel{t}{\rightarrow} j$  and $i \stackrel{*t}{\rightarrow} j$ are all increasing.
\begin{definition}[Disjoint occurrence of increasing events]
	If $A$ and $B$ are increasing events then we define the disjoint occurrence $A \square B$ of $A$ and $B$ as follows:
	$A \square B$ is also an event, and the marked point configuration $\omega$ is in $A \square B$ if and only if $\omega$ can be written as the  union of disjoint marked point configurations $\omega'$ and $\omega''$ satisfying $\omega' \in A$ and $\omega'' \in B$.
\end{definition}

Note that if $A$ and $B$ are increasing then $A \square B$ is also increasing.

Also note that the interval $[0,t]$ is bounded, so,  according to the definitions of \cite{vdBpoiBK}, our events live on a bounded region, thus the main theorem of \cite{vdBpoiBK} can be applied to conclude that in our setting
\begin{equation}\label{poi_bk_ineq}
	\mathbb{P}(A \square B) \leq \mathbb{P}(A) \mathbb{P}(B)
\end{equation}
holds if $A$ and $B$ are both increasing events.

\begin{proof}[Proof of Lemma \ref{lemma_ghost_bk}]	
Let us start the infection from vertex $i$.
If there is a ghost before time $t$, let us denote by
$\mathcal{K}$ the first vertex where a ghost appears and let $T$ denote the time when the first ghost appears. Thus, at time $T$, $\mathcal{K}$ is the only vertex that can be reached from $i$ by two different infection paths (cf.\ Definition \ref{def_infection_path}). Let $\mathcal{M}$ denote the location of the most recent common ancestor of these two paths. We have
\begin{equation}\label{ghost_split_m_k}
	\pr \left(G_i(t) \right)=\sum_{m, k \in V}  \pr \left(G_i(t), \mathcal{M}=m, \mathcal{K}=k \right).
\end{equation}
We will treat the cases $m=k$ and $m \neq k$ separately, cf.\ Figure \ref{bk_figure}. We have
\begin{equation}\label{m_equals_k_events}
	\{ G_i(t), \mathcal{M}=m, \mathcal{K}=m \} \subseteq \{  i \stackrel{t}{\rightarrow} m  \} \square \{ m \stackrel{*t}{\rightarrow} m \},
\end{equation}
thus we obtain
\begin{multline}\label{ghost_bound_m_m}
	\sum_{m \in V}  \pr \left(G_i(t), \mathcal{M}=m, \mathcal{K}=m \right) \stackrel{\eqref{poi_bk_ineq}, \eqref{m_equals_k_events} }{\leq}
	\sum_{m \in V}  \mathbb{P}( i \stackrel{t}{\rightarrow} m  )  \mathbb{P}( m \stackrel{*t}{\rightarrow} m  ) \stackrel{\eqref{inf_path_matrixexp},\eqref{inf_loop_matrixexp} }{\leq} \\
	\sum_{m \in V}
	\left(e^{Rt} \right)_{im} \left( \left(e^{Rt} \right)_{mm}  -1\right).
\end{multline}
If  $m \neq k$ then
\begin{equation}\label{m_neq_k_events}
	\{ G_i(t), \mathcal{M}=m, \mathcal{K}=k \} \subseteq \{  i \stackrel{t}{\rightarrow} m  \} \square \{ m \stackrel{t}{\rightarrow} k \} \square  \{ m \stackrel{t}{\rightarrow} k \},
\end{equation}
thus we obtain
\begin{multline}\label{ghost_bound_m_k}
	\sum_{ \substack{m, k \in V \\ k \neq m}}     \pr \left(G_i(t), \mathcal{M}=m, \mathcal{K}=k \right) \stackrel{ \eqref{poi_bk_ineq}, \eqref{m_neq_k_events}}{\leq}
	\sum_{ \substack{m, k \in V \\ k \neq m}}  \mathbb{P} (i \stackrel{t}{\rightarrow} m  ) \mathbb{P} ( m \stackrel{t}{\rightarrow} k )^2 \stackrel{\eqref{inf_path_matrixexp} }{\leq} \\
	\sum_{ \substack{m, k \in V \\ k \neq m}}   \left(e^{Rt} \right)_{im}  \left(e^{Rt} \right)_{mk}^2  \stackrel{(*)}{=} \sum_{m \in V}  \left(e^{Rt} \right)_{im} \left( \left(e^{2Rt} \right)_{mm} - \left(e^{Rt} \right)_{mm}^2 \right),
\end{multline}
where in $(*)$ we used $R^\intercal=R$ and the definitions of matrix multiplication and matrix exponential:
$$\sum_{k \in V}  \left(e^{Rt} \right)_{mk}^2 = \sum_{k \in V}  \left(e^{Rt} \right)_{mk} \left(e^{Rt} \right)_{km} =\left(e^{2Rt} \right)_{mm}.$$

Putting together \eqref{ghost_split_m_k}, \eqref{ghost_bound_m_m} and \eqref{ghost_bound_m_k} we obtain the desired \eqref{eq_ghost_bk} using that $\left(e^{Rt} \right)_{mm} \geq 1$ for each $m \in V$.  The proof of Lemma \ref{lemma_ghost_bk} is complete.
\begin{figure}[h]
	\includegraphics[width=0.9 \textwidth]{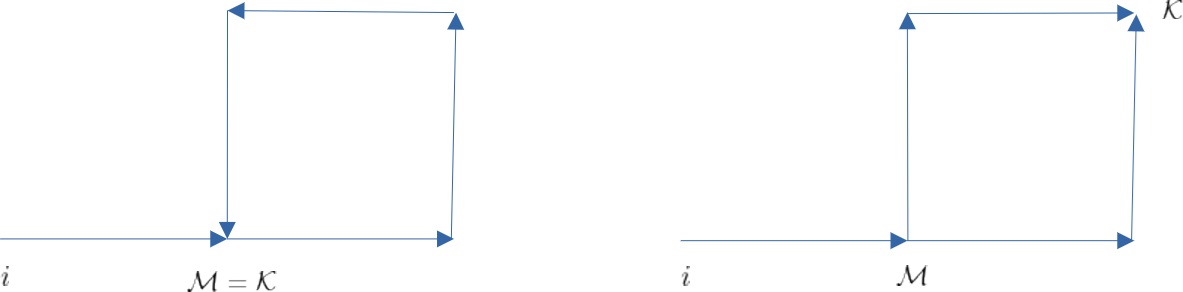} 
	\caption{\label{bk_figure} Schematic representations of the $\mathcal{M}=\mathcal{K}$ and $\mathcal{M} \neq \mathcal{K}$ cases.}
\end{figure}
	\end{proof}

The following lemma will help us convert the bound of Lemma \ref{lemma_ghost_bk} into the bounds of Theorem \ref{t:l1}.
\begin{lemma}
\label{l:mm}
Assume $R^{\intercal}=R.$ Then we have
\begin{align}
\label{eq:mm}
\left(e^{2Rt}-I \right)_{mm} \leq 4  t^2 e^{2 \|R \|_2 t} \left( R^2\right)_{mm} .
\end{align}
\end{lemma}

\begin{proof}
The eigen pairs of $R$ are denoted by $(\lambda_i,v_i)$. Let $e_{m}$ denote the $m$th unit vector. For any $k \geq 2$ we have
\begin{multline}\label{self_adjoint_tricks}
\left(R^k \right)_{mm}=e_m^\intercal R^k e_m=\sum_{l=1}^N \lambda_l^k \langle e_m, v_l \rangle^2 \leq \|R \|_2^{k-2} \sum_{l=1}^N \lambda_l^2 \langle e_m, v_l \rangle^2 \\
= \|R \|_2^{k-2} \left(R^2 \right)_{mm}.
\end{multline}

Note that for $(R)_{mm}=\bar{r}_{mm}=0$ and $I$ is subtracted, therefore we can start the expansion of the matrix exponential from the second order term:

\begin{align*}
\left(e^{2Rt}-I \right)_{mm}=&\sum_{k=2}^{\infty} \left(R^k \right)_{mm} \frac{(2t)^k}{k!} \stackrel{\eqref{self_adjoint_tricks}}{\leq} 4t^2 (R^2)_{mm} \sum_{k=2}^{\infty} \|R \|_{2}^{k-2} \frac{(2t)^{k-2}}{(k-2)!} \\
=& 4t^2e^{2 \|R \|_2 t} \left(R^2 \right)_{mm}.
\end{align*} 
The proof of \eqref{eq:mm} is complete.
\end{proof}

\begin{proof}[Proof of Theorem \ref{t:l1}]
Let $\mathds{1}$ denote the all $1$ column vector in $\mathbb{R}^V$.
	\begin{align*}
&\frac{1}{N}\sum_{i \in V} \max_{s \in \mathcal{S}} \left|y_{i,s}(t)-z_{i,s}(t) \right| \overset{\eqref{eq:ghost}}{\leq }\frac{1}{N}\sum_{i \in V} \pr \left(G_i(t) \right) \overset{\eqref{eq_ghost_bk}}{\leq }\\
&\frac{1}{N}\sum_{i \in V}\sum_{m \in V } \left(e^{Rt} \right)_{im}\left(e^{2Rt}-I \right)_{mm} \overset{\eqref{eq:mm}}{ \leq }4t^2 e^{2 \|R \|_2 t} \frac{1}{N}\sum_{m \in V} \left(\mathds{1}^{\intercal} e^{Rt} \right)_m \left( R^2\right)_{mm}.
\end{align*}

First we prove \eqref{eq:l1_delta_max}. Note that
$\|R \|_2 \leq \|R \|_{\infty}=\delta_{max}$ by \eqref{r:R_graph} and
\begin{align*}
\left(\mathds{1}^{\intercal}e^{Rt} \right)_m \leq \left \| \mathds{1}^{\intercal}  e^{Rt} \right\|_{\infty} \leq e^{\|R \|_{\infty} t} \|\mathds{1} \|_{\infty}=e^{\delta_{max}t}.
\end{align*}
Putting these bounds together we obtain the desired \eqref{eq:l1_delta_max}:
\begin{align*}
4t^2 e^{2 \|R \|_2 t} \frac{1}{N}\sum_{m \in V} \left(\mathds{1}^{\intercal}e^{Rt} \right)_m \left( R^2\right)_{mm} \leq 4t^2 e^{3 \delta_{max} t} \frac{1}{N}\sum_{m \in V}  \left( R^2\right)_{mm}.
\end{align*}
Now we prove \eqref{eq:l1_sigma}. We use Cauchy-Schwarz to obtain
\begin{align*}
\frac{1}{N}\sum_{m \in V} \left(\mathds{1}^{\intercal} e^{Rt} \right)_m \left( R^2\right)_{mm} \leq &
\sqrt{\frac{1}{N}\sum_{m \in V} \left(\mathds{1}^{\intercal} e^{Rt} \right)_m^2 } \sqrt{\frac{1}{N}\sum_{m \in V}  \left( R^2\right)^2_{mm}}  \\
\stackrel{(*)}{\leq} & e^{\|R \|_2 t}\sqrt{\frac{1}{N}\sum_{m \in V}  \left( R^2\right)^2_{mm}},
\end{align*}
where in the  inequality marked by $(*)$ we used
\begin{align*}
\sum_{m \in V} \left(\mathds{1}^{\intercal} e^{Rt} \right)_m^2= \|e^{Rt} \mathds{1} \|_2^2 \leq \|e^{R  t} \|^2 \cdot \|\mathds{1} \|_2^2
\leq 
e^{2 \|R \| t}\|\mathds{1} \|_2^2=e^{2 \|R \| t}N.
\end{align*}
The proof of \eqref{eq:l1_sigma} and thus Theorem \ref{t:l1} is complete.
\end{proof}

\begin{proof}[Proof of Proposition \ref{p:chung-lu}]

It follows from $0<\gamma<\frac{1}{3}$ that we have  \begin{equation}\label{sums_theta_N}
\sum_{k=1}^N \left( \frac{N}{k} \right)^{\gamma} =\Theta(N), \qquad \sum_{k=1}^N \left( \frac{N}{k} \right)^{2\gamma} =\Theta(N).
\end{equation}
Let us define $p_{ij}:= N^{\alpha} \frac{\left( \frac{N}{i} \right)^{\gamma}\left( \frac{N}{j} \right)^{\gamma}}{\sum_{k=1}^N \left( \frac{N}{k} \right)^{\gamma}}$ for any $i,j \in [N]$.

It follows from $\alpha<1-2 \gamma$ that 
$ 
p_{ij} \leq \frac{N^{\alpha+2 \gamma}}{\sum_{k=1}^N \left( \frac{N}{k} \right)^{\gamma}}= \Theta \left(N^{\alpha-(1-2 \gamma)} \right) \to 0$,
thus $p_{ij} <1$ holds if $N$ is large enough.

 The elements of the adjacency  matrix $A$ satisfy $a_{ij}=a_{ji}$ for all $i,j \in [N]$ and $\pr \left(a_{ij}=1 \right)=p_{ij}$ if $i \neq j$, while $a_{ii}=0$.

We first prove that $\delta_{max}=\Theta \left( N^{\gamma} \right)$ holds with high probability.

 Let us define $X_{ij}:=a_{ij}-\pr \left(a_{ij}=1 \right)$ for all $i,j \in V$ (noting that $X_{ii}=0$). The random variables $X_{i,j}, 1 \leq i<j \leq N$ are independent  with zero mean that satisfy $|X_{ij}| \leq 1$  hence, Bernstein's inequality yields
\begin{equation}\label{bernstejn}
\pr \left( \left| \sum_{j=1}^N X_{ij} \right| \geq t \right) \leq 2 \exp \left(-\frac{\frac{1}{2}t^2}{\sum_{j=1}^{N} \E \left(X_{ij}^2 \right)+\frac{1}{3}t} \right).
\end{equation}
Set $t_i=\frac{1}{2}\sum_{j=1}^{N} p_{ij}=\frac{1}{2}N^{\alpha} \left( \frac{N}{i} \right)^{\gamma}$.  We have
\begin{align}
\label{matra1}
&\sum_{j=1}^{N} \E \left(X_{ij}^2 \right)= \sum_{j=1}^{N} \operatorname{Var} \left(a_{ij} \right) \leq \sum_{j=1}^{N} \pr \left(a_{ij}=1 \right) \leq2t_i, \\
\label{matra2}
&\frac{\frac{1}{2}t_i^2}{\sum_{j=1}^{N} \E \left(X_{ij}^2 \right)+\frac{1}{3}t_i} \stackrel{\eqref{matra1}}{\geq}  \frac{\frac{1}{2}t_i^2}{2t_i+\frac{1}{3}t_i}=\frac{3}{14}t_i=\frac{3}{28} N^{\alpha} \left( \frac{N}{i} \right)^{\gamma} \geq \frac{3}{28} N^{\alpha}.
\end{align}
Noting that the degrees of our graph satisfy $\frac{1}{2}\mathbb{E}(d_i)= t_i - \frac{1}{2} p_{ii} \leq t_i$, we obtain
\begin{align}\label{degree_deviation_bounds}
\pr \left( \bigcup_{i=1}^{N} \left \{ \left|d_i-\E \left(d_i \right) \right| \geq \frac{1}{2} \E \left(d_i \right) \right \} \right) \stackrel{\eqref{bernstejn}, \eqref{matra2}}{\leq} 2N e^{-\frac{3}{28} N^{\alpha}}.    
\end{align}
We have $R=\frac{1}{N^{\alpha}}A$, thus $\delta_{i}= \frac{1}{N^{\alpha}}d_i$.
We have $\mathbb{E}(\delta_i) = \frac{1}{N^{\alpha}}\mathbb{E}(d_i)=\left( \frac{N}{i} \right)^{\gamma} -\frac{p_{ii}}{N^{\alpha}} $.
It follows from this and \eqref{degree_deviation_bounds}
that with high probability for all $i \in V $ we have $ \ \frac{1}{4} \left( \frac{N}{i} \right)^{\gamma} \leq \delta_i \leq 2 \left( \frac{N}{i} \right)^{\gamma}$, thus  $\delta_{max}=\Theta \left( N^{\gamma} \right)$ holds with high probability.

Now we prove that $\|R\|_2=O(1)$ holds with high probability.
 
 Let $\Delta:=\E \left( d_1 \right) = \theta (N^{\alpha+\gamma})$ denote the largest expected degree and let us define $\varepsilon=\frac{1}{N}$, say. We have  $\Delta > \frac{4}{9} \log \left( \frac{2N}{\varepsilon} \right) $ for large enough $N$, thus by the results of  \cite{randomgraphspectra} and our assumption $\gamma<\alpha$ we have
\begin{align*}
\|R \|_2 \leq \|\E \left(R \right)  \|_2+\frac{1}{N^{\alpha}} \sqrt{4 \Delta \log \left( \frac{2N}{\varepsilon} \right) }
= \|\E \left(R \right)  \|_2+\underbrace{O \left(\sqrt{N^{\gamma-\alpha} \log N} \right)}_{\to 0},
\end{align*}
with probability at least $1-\varepsilon$.  We still need to show that $\|\E \left(R \right)  \|_2 =O(1)$.

Let us define $R'=( \frac{p_{ij}}{N^{\alpha}})_{i,j=1}^{N}$. Note that we have $\mathbb{E}(\bar{r}_{ji})= \frac{p_{ij}}{N^{\alpha}}$ if $i \neq j$, while  $0=\bar{r}_{ii} \leq \frac{p_{ii}}{N^{\alpha}}$, thus the inequalities $0 \leq \E \left(R \right) \leq R'$ hold coordinate-wise, thus we have $ \|\E \left(R \right)  \|_2 \leq \| R' \|_2 $. Observe that the rank of $R'$ is one. More specifically, if we define 
$v_i:= \frac{\left(\frac{N}{i} \right)^{\gamma}}{\sqrt{ \sum_{k=1}^{N}\left(\frac{N}{k} \right)^{2\gamma}}}$ then $\|v\|_2=1$ and $R'  = \frac{\sum_{k=1}^{N}\left(\frac{N}{k} \right)^{2\gamma}}{\sum_{k=1}^{N}\left(\frac{N}{k} \right)^{\gamma}} vv^{\intercal}$, thus  $\|R' \|_2=O(1)$ by \eqref{sums_theta_N}.
\end{proof} 

\begin{proof}[Proof of Lemma \ref{l:former_remark}]
Using $\bar{r}_{ij}=\frac{a_{ij}}{\bar{d}}$ we have
\begin{equation*}
	\left(R^2 \right)_{mm}=\sum_{l \in V} \left(\frac{a_{ml}}{\bar{d}} \right)^2=\frac{d_m}{\bar{d}^2}, \quad
	\frac{1}{N}\sum_{m \in V}\left(R^2 \right)_{mm}= \frac{1}{N}\sum_{m \in V}\frac{d_m}{\bar{d}^2}=\frac{1}{\bar{d}}, 
\end{equation*}
 \begin{equation*}
	\sqrt{\frac{1}{N}\sum_{m \in V} \left(R^2 \right)_{mm}^2}= \sqrt{\frac{1}{N}\sum_{m \in V} \left(\frac{d_m}{\bar{d}^2} \right)^2}=  \frac{1}{\bar{d}} \sqrt{\frac{1}{N}\sum_{m \in V} \left(\frac{d_m}{\bar{d}} \right)^2}\leq \|R \|_2 \frac{1}{\bar{d}}, 
\end{equation*}
where the last inequality follows from
\begin{equation}\label{eq:sq_degrees_norm}
	\frac{1}{N}\sum_{m \in V}\left(\frac{d_m}{\bar{d}}\right)^2=\frac{1}{N}\sum_{m \in V} \left(R\mathds{1} \right)_{m}^2=\frac{1}{N}\|R \mathds{1} \|_2^2 \leq \|R \|_2^2,
\end{equation}
where $\mathds{1}$ denotes the all $1$ column vector in $\mathbb{R}^V$.
\end{proof}

Our next goal is to prove the lower bound stated in Theorem \ref{t:l1_lower_bound}.

It is known (cf.\ \cite{simon2017NIMFA}) that $y_{i,I}(t) \leq z_{j,I}(t)$ holds in the case of the SIS process. The next lemma makes this even more explicit for the case of the SI process.

\begin{lemma}[$\widetilde{\sigma}_{i}(t)$ dominates $\sigma_{i}(t)$ in the case of the SI model]
\label{l:sigma_biger}
Recall the notation from Section \ref{section_coupling} for $\sigma_{i}(t)$ and $\widetilde{\sigma}_{i}(t).$
	Take an SI process on the weighted graph $R$. Then for all $i \in V$ we must have $\sigma_{i}(t)=I \Rightarrow \widetilde{\sigma}_{i}(t)=I$.
\end{lemma}
\begin{proof}
In the original process vertex $i \in V$ is infected at time $t$ if and only if there is a vertex in $\mathcal{H}_i(t)$ which is infected at time $t=0$. Similarly, we have $\widetilde{\sigma}_{i}(t)=I$ in our auxiliary branching structure if and only if there is at least one initially infected augmented vertex in  $\widetilde{\mathcal{H}}_{i,t}(t)$. However, under the coupling of Section \ref{s:joint_construction} the set $\widetilde{\mathcal{H}}_{i,t}(t)$ contains  $\mathcal{H}_i(t)$ (if we identify each vertex $j$ in $\mathcal{H}_i(t)$ with the augmented vertex $(j,1)$).
 Recalling that we identified $\underline{\sigma}(0)$ with $(\widetilde{\sigma}_{i, 1}(0))_{i \in V}$, the proof is complete.
\end{proof}

\begin{proof}[Proof of Theorem \ref{t:l1_lower_bound}]
Recall the notation in Section \ref{s:joint_construction}. 

Let $\mathcal{A}_{ij}$ denote the event that in the auxiliary branching construction the following events all occur: (a) $(i,1)$ receives two infection signals during the time interval $t \in [0,1]$ from the augmented vertices $(j,1)$ and $ (j,2)$, but (b) $i$ does not receive an infection signal from any other vertex  $t \in [0,1]$,  (c) $(j,1)$ does not receive any infection signals while $t \in [0,1]$, moreover (d) initially $(i,1)$ and $(j,1)$ are susceptible and $(j,2)$ is infected (see Figure \ref{fig_ghosttt} for a visualization). In the original SI process that produces $\sigma_i(1)$ only the  augmented vertices with $\ell=1$ in their second coordinate matter, thus $\sigma_i(1)=S$ as $(j,1)$ remains susceptible. On the other hand we have $\widetilde{\sigma}_i(1)=I$, since vertex $i$ is infected by the augmented vertex $(j,2)$.

\begin{figure}[h]
	\includegraphics[width=0.9 \textwidth]{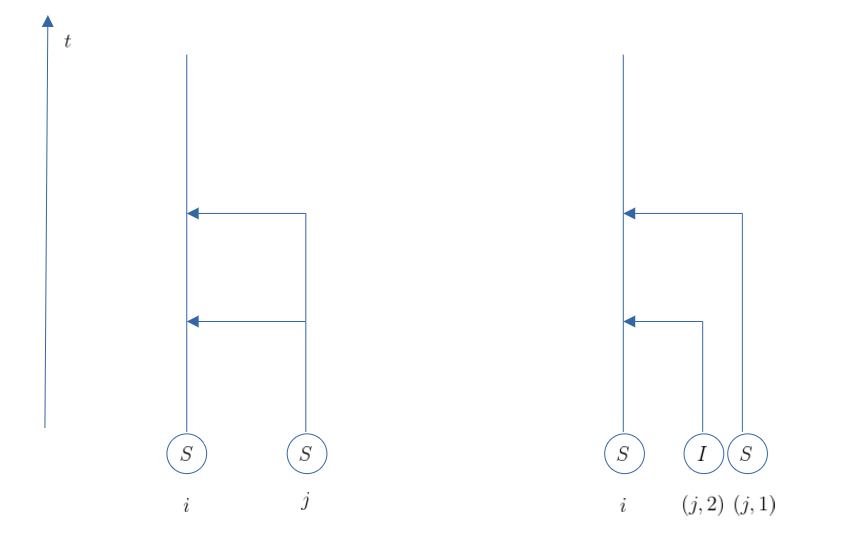} 
	\caption{\label{fig_ghosttt} Event $\mathcal{A}_{ij}$ on the original backward construction (left) and on the auxiliary branching structure (right).}
\end{figure}

Recall that the PPPs and the initial conditions are all independent. Note that the probability that $(j,1)$ does not receive any infection signals is lower bounded by $e^{-\delta_j}$ since the infection signals only need to be ``switched off" between $t=0$ and the time we create $(j,1)$, which must take place in $[0,1].$ All the other terms in the following formula can be exactly calculated:
\begin{align*}
\pr \left( \mathcal{A}_{ij} \right) \geq \frac{1}{8} \frac{\bar{r}^2_{ji}}{2!}e^{-\bar{r}_{ji}} e^{-\delta_j} \prod_{k \in V \setminus \{i,j\}}e^{-\bar{r}_{ki}}=\frac{1}{16}e^{-(\delta_i+\delta_j)}\bar{r}^2_{ji}.
\end{align*}
Using Lemma \ref{l:sigma_biger} in $(*)$ below and the fact that  the events $\mathcal{A}_{ij}$ are disjoint for different values of $j \in V$ in $(**)$ below, we obtain 
\begin{align*}
&\frac{1}{N} \sum_{i \in V} \left|y_{i,I}(1)-z_{i,I}(1) \right| \stackrel{(*)}{=} \frac{1}{N} \sum_{i \in V} \pr \left( \widetilde{\sigma}_{i}(1)=I; \sigma_{i}(1)=S \right) \geq \\
 & \frac{1}{N} \sum_{i \in V} \pr \left( \bigcup_{j \in V} \mathcal{A}_{ij} \right) \stackrel{(**)}{=} \frac{1}{N} \sum_{i,j \in V} \pr \left(  \mathcal{A}_{ij} \right) \geq \frac{1}{16}\frac{1}{N} \sum_{i,j \in V} e^{-(\delta_i+\delta_j)}\bar{r}^2_{ji}.
\end{align*}
The proof of \eqref{nimfa_l1_lower_exp_delta} is complete.
\end{proof}

\begin{proof}[Proof of Theorem \ref{t:l1_graph}]
For some  $K \in \mathbb{R}_+$ (to be chosen later), let $V_K$ denote the set of vertices
$ 	V_K:=  \{ \,  i \in V \; : \; \delta_i \leq K \, \}$.
Using Theorem \ref{t:l1_lower_bound} we obtain
\begin{multline}\label{brrr1} 
 \frac{1}{N}\sum_{i \in V} \left |y_{i,I}(1)-z_{i,I}(1) \right| \geq \frac{1}{16} \frac{1}{N}\sum_{i,j \in V}e^{-(\delta_i+\delta_j)}\bar{r}_{ji}^2 \geq \frac{1}{16} \frac{1}{N}\sum_{i,j \in V_K}e^{-(\delta_i+\delta_j)}\bar{r}_{ji}^2 \\
   \geq  \frac{e^{-2K}}{16} \frac{1}{N}\sum_{i,j \in V_K}\bar{r}_{ji}^2=\frac{e^{-2K}}{16} \left(\frac{1}{N \bar{d}} \sum_{i,j \in V_K}a_{ij} \right) \frac{1}{\bar{d}}.
\end{multline}
We will  bound the expression between the brackets from below. Let $\iota$ stand for a uniformly chosen vertex of $V$.
Noting that we have $\sum_{i,j\ \in V} a_{ij}=N \bar{d},$ we find
\begin{multline*}
1-\frac{1}{N \bar{d}} \sum_{i,j \in V_K}a_{ij}=\frac{1}{N \bar{d}} \sum_{i,j \in V}a_{ij} \1{ \{i \in V \setminus V_K \} \cup \{j \in V \setminus V_K \}} \leq \\
\frac{2}{N \bar{d}} \sum_{i,j \in V}a_{ij}\1{i \in V \setminus V_K} = \frac{2}{N} \sum_{i \in V}\delta_i\1{i \in V \setminus V_K} = \\
2 \mathbb{E}(\delta_{\iota}\mathds{1}[\delta_{\iota} > K ] ) \leq 
2 \mathbb{E}\left(\frac{\delta^2_{\iota}}{K}\mathds{1}[\delta_{\iota} > K ] \right) \leq \frac{2}{K} \E \left(\delta^2_{\iota} \right) \stackrel{ \eqref{eq:sq_degrees_norm} }{ \leq}
\frac{2 \|R \|^2_2 }{K}=\frac{1}{2},
\end{multline*}
where in the last step we set $K= 4 \|R \|_2^2$. Using this we obtain
\begin{align}\label{brrr2}
\frac{e^{-2K}}{16} \left(\frac{1}{N \bar{d}} \sum_{i,j \in V_K}a_{ij} \right) \frac{1}{\bar{d}} \geq \frac{1}{32}e^{-8 \|R \|_2^2} \frac{1}{\bar{d}}.
\end{align}
Now the proof of \eqref{lower_recipr_bar_d} follows from \eqref{brrr1} and \eqref{brrr2}.
\end{proof}

\begin{proof}[Proof of Theorem \ref{t:l1_charachterisation}]
This proof is very similar to the proof of Theorem \ref{t:l1_graph}.  We still denote $ 	V_K=  \{ \,  i \in V \, : \, \delta_i \leq K \, \}$ and $\iota$ again denotes a uniformly chosen vertex of $V$. 
Note that we have $\mathbb{E}(\delta_{\iota})=\frac{1}{N} \mathds{1}^{\intercal} R \mathds{1} \leq  \|R \|_2 $, thus
 Markov's inequality gives 
 \begin{equation}\label{V_V_K_weighted_markov}
 \frac{1}{N}\sum_{i \in V} \1{i \in V \setminus V_K } =\mathbb{P}(\delta_{\iota} >K) \leq \frac{\|R \|_2}{K}.
 \end{equation} 
Using Theorem \ref{t:l1_lower_bound} we obtain
\begin{align}\label{frob_lb}
\frac{1}{N}\sum_{i \in V} \left |y_{i,I}(1)-z_{i,I}(1) \right| \geq 
\frac{1}{16}\frac{1}{N}\sum_{i,j \in V}e^{-(\delta_i+\delta_j)}\bar{r}_{ij}^2 \geq \frac{e^{-2K}}{16} \frac{1}{N}\sum_{i,j \in V_K}\bar{r}_{ij}^2.
\end{align}
It is enough to show that the expression on the r.h.s.\ of \eqref{frob_lb} separated from $0$.

Since $\theta=\frac{1}{N} \sum_{i,j \in V}\bar{r}_{ij}^2$, we have
\begin{align*}
 &\theta- \frac{1}{N}\sum_{i,j \in V_k}\bar{r}_{ij}^2 =\frac{1}{N}\sum_{i,j \in V}\bar{r}_{ij}^2 \1{ \{i \in V \setminus V_K\} \cup \{j \in V \setminus V_K\}} \leq \\ 
 & \frac{2}{N}\sum_{i,j \in V}\bar{r}_{ij}^2\1{i \in V \setminus V_K} \stackrel{(*)}{=}\frac{2}{N}\sum_{i \in V}\left(R^2 \right)_{ii}\1{i \in V \setminus V_K} \stackrel{ \eqref{R_squared_and_norm} }{\leq} \\
 &\frac{2 \|R \|_2^2}{N}\sum_{i \in V} \1{i \in V \setminus V_K} \stackrel{\eqref{V_V_K_weighted_markov}}{\leq} \frac{2 \|R \|_2^3}{K}=\frac{1}{2}\theta,
\end{align*}
where in $(*)$ we used that $\sum_{j \in V} \bar{r}_{ij}^2=\sum_{j\in V} \bar{r}_{ij} \bar{r}_{ji}=\left(R^2 \right)_{ii}$, and
in the last step we set $K= \frac{4\|R \|_2^3}{\theta}.$ Thus we can bound the r.h.s.\ of \eqref{frob_lb} from below by
\begin{align*}
 \frac{e^{-2K}}{16} \frac{1}{N}\sum_{i,j \in V}\bar{r}_{ij}^2 \geq \frac{1}{32}e^{-\frac{8 \|R \|_2^3}{\theta}} \theta.   
\end{align*}
\end{proof}

\bibliographystyle{abbrv}
\bibliography{salad3}

\begin{thebibliography}{10}

\bibitem{ihomogenous_refined_meanfield}
S.~Allmeier and N.~Gast.
\newblock Mean field and refined mean field approximations for heterogeneous
  systems: It works!
\newblock {\em Proc. ACM Meas. Anal. Comput. Syst.}, 6(1), feb 2022.

\bibitem{Bakhshi2010a}
R.~Bakhshi, L.~Cloth, W.~Fokkink, and B.~R. Haverkort.
\newblock {Mean-field framework for performance evaluation of push-pull gossip
  protocols}.
\newblock {\em Performance Evaluation}, 68(2):157--179, Feb. 2011.

\bibitem{Barbour}
A.~Barbour and G.~Reinert.
\newblock {Approximating the epidemic curve}.
\newblock {\em Electronic Journal of Probability}, 18(none):1 -- 30, 2013.

\bibitem{ER_spectrum}
F.~Benaych-Georges, C.~Bordenave, and A.~Knowles.
\newblock Largest eigenvalues of sparse inhomogeneous erdos-renyi graphs.
\newblock {\em The Annals of Probability}, 47(3):pp. 1653--1676, 2019.

\bibitem{Juli2014}
S.~Bhamidi, R.~van~der Hofstad, and J.~Komjáthy.
\newblock The front of the epidemic spread and first passage percolation.
\newblock {\em Journal of Applied Probability}, 51(A):101–121, 2014.

\bibitem{randomgraphspectra}
F.~Chung and M.~Radcliffe.
\newblock On the spectra of general random graphs.
\newblock {\em Electr. J. Comb.}, 18, 10 2011.

\bibitem{simplicial_SIS}
P.~Cisneros-Velarde and F.~Bullo.
\newblock Multigroup sis epidemics with simplicial and higher order
  interactions.
\newblock {\em IEEE Transactions on Control of Network Systems}, 9(2):695--705,
  2022.

\bibitem{SAIS}
F.~Darabi~Sahneh, C.~Scoglio, and P.~Van~Mieghem.
\newblock Generalized epidemic mean-field model for spreading processes over
  multilayer complex networks.
\newblock {\em IEEE/ACM Transactions on Networking}, 21(5):1609--1620, 2013.

\bibitem{volzproof}
L.~Decreusefond, J.-S. Dhersin, P.~Moyal, and V.~C. Tran.
\newblock Large graph limit for an {SIR} process in random network with
  heterogeneous connectivity.
\newblock {\em Annals of Applied Probability}, 22, 04 2012.

\bibitem{Delmas2023}
J.-F. Delmas, P.~Frasca, F.~Garin, V.~C. Tran, A.~Velleret, and P.-A. Zitt.
\newblock Individual based sis models on (not so) dense large random networks,
  2023.

\bibitem{flip_process}
F.~Garbe, J.~Hladký, M.~Šileikis, and F.~Skerman.
\newblock From flip processes to dynamical systems on graphons, 2023.

\bibitem{refined_meanfield}
N.~Gast and B.~Van~Houdt.
\newblock A refined mean field approximation.
\newblock {\em Proc. ACM Meas. Anal. Comput. Syst.}, 1(2), dec 2017.

\bibitem{hht2014}
R.~A. Hayden, I.~Horv{\'a}th, and M.~Telek.
\newblock Mean field for performance models with generally-distributed timed
  transitions.
\newblock In G.~Norman and W.~Sanders, editors, {\em Quantitative Evaluation of
  Systems}, volume 8657 of {\em Lecture Notes in Computer Science}, pages
  90--105. Springer International Publishing, 2014.

\bibitem{jacob2022}
E.~Jacob, A.~Linker, and P.~Mörters.
\newblock The contact process on dynamical scale-free networks, 2022.

\bibitem{jk14}
S.~Jansen and N.~Kurt.
\newblock {On the notion(s) of duality for Markov processes}.
\newblock {\em Probability Surveys}, 11(none):59 -- 120, 2014.

\bibitem{NIMFA_Illes}
D.~Keliger and I.~Horváth.
\newblock Accuracy criterion for mean field approximations of markov processes
  on hypergraphs.
\newblock {\em Physica A: Statistical Mechanics and its Applications},
  609:128370, 2023.

\bibitem{Keliger2022}
D.~Keliger, I.~Horváth, and B.~Takács.
\newblock Local-density dependent markov processes on graphons with
  epidemiological applications.
\newblock {\em Stochastic Processes and their Applications}, 148:324--352,
  2022.

\bibitem{Juli2015}
I.~Kolossváry and J.~Komjáthy.
\newblock First passage percolation on inhomogeneous random graphs.
\newblock {\em Advances in Applied Probability}, 47(2):589–610, 2015.

\bibitem{Kuehn2023}
C.~Kuehn and C.~Xu.
\newblock Vlasov equations on directed hypergraph measures, 2023.

\bibitem{kurtz70}
T.~Kurtz.
\newblock Solutions of ordinary differential equations as limits of pure jump
  {M}arkov processes.
\newblock {\em Journal of Applied Probability}, 7:49--58, 04 1970.

\bibitem{kurtz78}
T.~G. Kurtz.
\newblock Strong approximation theorems for density dependent {M}arkov chains.
\newblock {\em Stochastic Processes and their Applications}, 6(3):223 -- 240,
  1978.

\bibitem{Kavita2023}
D.~Lacker, K.~Ramanan, and R.~Wu.
\newblock {Local weak convergence for sparse networks of interacting
  processes}.
\newblock {\em The Annals of Applied Probability}, 33(2):843 -- 888, 2023.

\bibitem{Glauber}
E.~Lubetzky.
\newblock Glauber dynamics for spin systems at high and critical temperatures.
\newblock 03 2010.

\bibitem{credit_risk}
T.~Lux.
\newblock A model of the topology of the bank – firm credit network and its
  role as channel of contagion.
\newblock {\em Journal of Economic Dynamics and Control}, 66:36--53, 2016.

\bibitem{msss20}
T.~Mach, A.~Sturm, and J.~M. Swart.
\newblock {Recursive tree processes and the mean-field limit of stochastic
  flows}.
\newblock {\em Electronic Journal of Probability}, 25(none):1 -- 63, 2020.

\bibitem{NIMFA2011}
P.~Mieghem.
\newblock {The N-intertwined SIS epidemic network model}.
\newblock {\em Computing}, 93:147--169, 12 2011.

\bibitem{Piankoranee_2018}
S.~Piankoranee and S.~Limkumnerd.
\newblock Effects of global and local rewiring on sis epidemic adaptive
  networks.
\newblock {\em Journal of Physics: Conference Series}, 1144(1):012080, dec
  2018.

\bibitem{schapira2023}
B.~Schapira and D.~Valesin.
\newblock The contact process on dynamic regular graphs: monotonicity and
  subcritical phase, 2023.

\bibitem{simon2017NIMFA}
P.~L. Simon and I.~Z. Kiss.
\newblock On bounding exact models of epidemic spread on networks, 2017.
\newblock https://arxiv.org/abs/1704.01726.

\bibitem{Sridhar_Kar}
A.~Sridhar and S.~Kar.
\newblock Mean-field approximations for stochastic population processes with
  heterogeneous interactions.
\newblock {\em SIAM Journal on Control and Optimization}, 61(6):3442--3466,
  2023.

\bibitem{vdBpoiBK}
J.~van~den Berg.
\newblock A note on disjoint-occurrence inequalities for marked poisson point
  processes.
\newblock {\em Journal of Applied Probability}, 33(2):420–426, 1996.

\end{thebibliography}

\end{document}